\documentclass[10pt]{amsart}
\usepackage{amsmath,epsf}
\usepackage{mathrsfs}
\usepackage{hyperref}
\usepackage{bm}

\usepackage{amssymb,stmaryrd,mathtools,amsthm}
\usepackage{enumitem,xcolor}

\usepackage{tikz} \usetikzlibrary{matrix,arrows,decorations.pathmorphing}
\usepackage{tikz-cd}

\theoremstyle{plain}
\newtheorem{theorem}[subsection]{Theorem}
\newtheorem*{theorem*}{Theorem}
\newtheorem{corollary}[subsection]{Corollary}
\newtheorem*{corollary*}{Corollary}
\newtheorem{proposition}[subsection]{Proposition}
\newtheorem{proposition*}{Proposition}
\newtheorem{lemma}[subsection]{Lemma}
\newtheorem{fact}[subsection]{Fact}

\theoremstyle{definition}
\newtheorem{definition}[subsection]{Definition}
\newtheorem*{definition*}{Definition}
\newtheorem{example}[subsection]{Example}

\theoremstyle{remark}
\newtheorem{remark}[subsection]{Remark}

\newtheorem{notation}[subsection]{Notation}

\providecommand{\Z}{\mathbb{Z}}
\providecommand{\C}{\mathbb{C}}
\providecommand{\Q}{\mathbb{Q}}

\providecommand{\N}{\mathbb{N}}

\providecommand{\bbA}{\mathbb{A}}
\providecommand{\bbG}{\mathbb{G}}

\providecommand{\bbP}{\mathbb{P}}

\providecommand{\bbX}{\mathbb{X}}

\providecommand{\bbD}{\mathbb{D}}

\providecommand{\p}{\mathfrak{p}}

 \providecommand{\cF}{\mathscr{F}}
\providecommand{\cG}{\mathscr{G}}
\providecommand{\cA}{\mathscr{A}}

\providecommand{\cC}{\mathscr{C}}

\providecommand{\cL}{\mathscr{L}}
\providecommand{\cO}{\mathscr{O}}

\providecommand{\cP}{\mathscr{P}}
\providecommand{\cS}{\mathscr{S}}

\providecommand{\cU}{\mathscr{U}}
\providecommand{\cV}{\mathscr{V}}
\providecommand{\cX}{\mathscr{X}}

\providecommand{\spec}{\mathop{\rm Spec}}
\providecommand{\spf}{\mathop{\rm Spf}}

\providecommand{\id}{\mathop{\rm id}\nolimits}

\providecommand{\Aut}{\mathop{\rm Aut}\nolimits}
\providecommand{\Hom}{\mathop{\rm Hom}\nolimits}
\providecommand{\Der}{\mathop{\rm Der}\nolimits}

\providecommand{\op}{{\mathop{\rm op}\nolimits}}

\providecommand{\Const}{\mathop{\rm Const}\nolimits}

\renewcommand{\ker}{\mathop{\rm ker}\nolimits}

\providecommand{\im}{\mathop{\rm im}\nolimits}
\renewcommand{\Im}{\mathop{\rm Im}\nolimits}

\providecommand{\coker}{\mathop{\rm coker}\nolimits}

\providecommand{\Eq}{\mathop{\rm Eq}\nolimits}

\providecommand{\Coeq}{\mathop{\rm Coeq}\nolimits}

\providecommand{\Traj}{\mathop{\rm Traj}\nolimits}

\providecommand{\GL}{\mathop{\rm GL}\nolimits}
\providecommand{\SL}{\mathop{\rm SL}\nolimits}

\providecommand{\Ga}{\widehat{\mathbb{G}}_a}

\providecommand{\Gal}{\mathop{\rm Gal}}

\providecommand{\Split}{\text{\rm Split}}

\providecommand{\DD}{{\rm DD}}
\providecommand{\Ar}{{\rm Ar}}
\providecommand{\Sub}{\text{\rm Sub}}
\providecommand{\Equiv}{\text{\rm Equiv}}
\providecommand{\Cat}{\mathbf{Cat}}
\providecommand{\Precat}{\mathbf{PreCat}}
\providecommand{\Sch}{\mathrm{Sch}}
\providecommand{\Aff}{\text{\rm Aff}}
\providecommand{\DAff}{\delta\text{\rm -Aff}}

\providecommand{\Rng}{\text{\rm Rng}}
\providecommand{\DRng}{\delta\text{\rm -Rng}}

\providecommand{\da}{\text{\rm-}}

\providecommand{\Disc}[1]{\mathbf{\Delta}(#1)}

\providecommand{\DSch}{{\delta\text{\rm -Sch}}}

\providecommand{\ov}{{\kern-1pt\mathord{/}\kern-.7pt}}
\providecommand{\Ov}{{\kern-1pt\sslash\kern-.7pt}}

\providecommand{\lexp}[2]{{\vphantom{#2}}^{#1}{\kern-.1ex#2}}
\providecommand{\lsub}[2]{{\vphantom{#2}}_{#1}{\kern-.1ex#2}}

\providecommand{\lrexp}[3]{{\vphantom{#2}}^{#1}{\kern-.1ex#2^#3}}
\providecommand{\ldexp}[3]{{\vphantom{#2}}^{#1}{\kern-.1ex#2_#3}}

\makeatletter
\newcommand*{\doublerightarrow}[2]{\mathrel{
  \settowidth{\@tempdima}{$\scriptstyle#1$}
  \settowidth{\@tempdimb}{$\scriptstyle#2$}
  \ifdim\@tempdimb>\@tempdima \@tempdima=\@tempdimb\fi
  \mathop{\vcenter{
    \offinterlineskip\ialign{\hbox to\dimexpr\@tempdima+1em{##}\cr
    \rightarrowfill\cr\noalign{\kern.5ex}
    \rightarrowfill\cr}}}\limits^{\!#1}_{\!#2}}}
\makeatother

\begin{document}



\title{Galois theory of differential schemes}     
\date{\today}


\author{Behrang Noohi} 
\address{Behrang Noohi\\
School of Mathematical Sciences\\
  	Queen Mary University of London\\
         London, E1 4NS\\
        United Kingdom}
\email{b.noohi@qmul.ac.uk}

\author{Ivan Toma{\v s}i{\'c}} 
\address{Ivan Toma{\v s}i{\'c}\\
         School of Mathematical Sciences\\
  	Queen Mary University of London\\
         London, E1 4NS\\
        United Kingdom}
\email{i.tomasic@qmul.ac.uk}
\thanks{Supported by EPSRC Standard grant EP/V028812/1}

\begin{abstract}
Since 1883, Picard-Vessiot theory had been developed as the Galois theory of differential field extensions associated with linear differential equations. Inspired by categorical Galois theory of Janelidze, and by using novel methods of precategorical descent applied to algebraic-geometric situations, we develop a Galois theory that applies to morphisms of differential schemes, and vastly generalises the linear Picard-Vessiot theory, as well as the strongly normal theory of Kolchin. 
\end{abstract}

\maketitle



\tableofcontents

\section{Introduction}

\subsection{History}
Clasically, a Picard-Vessiot extension $(L,\delta_L)/(K,\delta_K)$ of differential fields yields a linear algebraic group
$$
G={\Gal}^\mathrm{PV}(L/K)
$$
over the common field of constants $k=\Const(L,\delta_L)=\Const(K,\delta_K)$ such that there is a Galois correspondence between the intermediate differential field extensions and Zariski closed subgroups of $G$, and
$$
G(k)\simeq \Aut((L,\delta_L)/(K,\delta_K)),
$$
as explained in \ref{pv-classical}.

The Picard-Vessiot ring is a $(K,\delta_K)$-algebra $$(A,\delta_A)$$ with fraction field $(L,\delta_L)$, such that $(X,\delta_X)=\spec(A,\delta_A)$ is a $G$-torsor over $(Y,\delta_Y)=\spec(K,\delta_K)$ in the sense that
$$
(X,\delta_X)\times_{(Y,\delta_Y)}(X,\delta_X)\simeq (X,\delta_X)\times_{(X_0,0)}(G,0),
$$
where we consider $X_0=\spec(k)$ and $G$ as differential schemes endowed with zero derivations. 
 
Janelidze realised that Picard-Vessiot theory fits into the framework of his categorical Galois theory through the adjunction
$$
 \begin{tikzpicture} 
 [cross line/.style={preaction={draw=white, -,
line width=3pt}}]
\matrix(m)[matrix of math nodes, minimum size=1.7em,
inner sep=0pt, 
row sep=3.3em, column sep=1em, text height=1.5ex, text depth=0.25ex]
 { 
  |(dc)|{\DAff}	\\
 |(c)|{\Aff} 	      \\ };
\path[->,font=\scriptsize,>=to, thin]
%
(dc) edge [bend right=30] node (ss) [left]{$S$} (c)
(c) edge [bend right=30] node (ps) [right]{$C$} (dc)
(ss) edge[draw=none] node{$\dashv$} (ps)
;
\end{tikzpicture}
$$
between the categories of affine differential schemes and affine schemes, where 
$$
S(\spec(R,\delta_R))=\spec(\Const(R,\delta_R)),  \ \ \ \  C(\spec(R))=\spec(R,0).
$$
He emphasised in \cite{janelidze-pv} and \cite{janelidze-pv2} that the key object, the morphism of \emph{relative Galois descent}, is in fact the torsor 
$$f:(X,\delta_X)\to (Y,\delta_Y),$$
and not the extension of differential fields, and that the Picard-Vessiot Galois group agrees with the categorical Galois group,
$$
{\Gal}^\mathrm{PV}(L/K)\simeq \Gal[f]=S(X\times_YX),
$$
as explained in \ref{janelidze-affine}.

In 2023, Akira Masuoka asked the second author whether categorical Galois theory approach automatically recovers the classical Picard-Vessiot Galois correspondence. To our surprise, we observed that Carboni-Janelidze-Magid correspondence from \cite{carboni} in the affine context only establishes a correspondence between split affine quotients of $X$ over $Y$ and \emph{effective} subgroups of $G=\Gal[f]$, i.e., those closed subgroups $H$ such that $G/H$ is represented by an affine scheme, see \ref{affine-pv-corr}.

We realised that, contrary to the popular belief, Picard-Vessiot theory is not an entirely affine affair, and that the categorical Picard-Vessiot theory must be extended to the correspondence between the split \emph{quasi-projective} fpqc quotients of $(X,\delta_X)$ over $(Y,\delta_Y)$, and closed subgroups of $G$, and that the intermediate differential fields from the classical correspondence appear as function fields of those quasi-projective quotients. 

Yves Andr\'e also noted the failure of the affine Picard-Vessiot correspondence in \cite{andre-sol}. He generalises the classical correspondence to that between \emph{solution fields} generated by some but not necessarily all solutions of the differential equation and \emph{observable} subgroups of $G$, i.e., those closed subgroups $H$ such that $G/H$ is quasi-affine. In the geometric direction somewhat different from ours, the associated  \emph{solution algebras} correspond to aﬃne quasi-homogeneous varieties for $G$. In our framework, the extensions generated by a partial set of solutions can be treated through the notion of \emph{pre-Picard-Vessiot} morphisms explained below. 

More generally, the Galois theory of a \emph{strongly normal} differential field extension $(L,\delta_L)/(K,\delta_K)$ of Kolchin \cite{kolchin-sn} is known to give rise to a general algebraic group $G$ over $k$ and a general differential scheme $(X,\delta_X)$ that acts as a torsor for $G$ through a relation analogous to the above, \cite{byalnicki}, \cite{buium}, \cite{umemura}. 

Michael Wibmer informed us of the idea of a number of researchers in differential algebra, including Jerry Kovacic, to interpret the strongly normal theory using  categorical Galois theory, and this was independently posed as a desirable project by Janelidze in \cite{janelidze-pv2}. The difficulty lies in the fact that, with no reasonable notion of differential scheme does the natural extension of the functor $C$ above to a functor between schemes and differential schemes
$$
 \begin{tikzpicture} 
 [cross line/.style={preaction={draw=white, -,
line width=3pt}}]
\matrix(m)[matrix of math nodes, minimum size=1.7em,
inner sep=0pt, 
row sep=3.3em, column sep=1em, text height=1.5ex, text depth=0.25ex]
 { 
  |(dc)|{\DSch}	\\
 |(c)|{\Sch} 	      \\ };
\path[->,font=\scriptsize,>=to, thin]
%
(c) edge [bend right=30] node (ps) [right]{$C$} (dc)
;
\end{tikzpicture}
$$
admit a left adjoint, so the classical categorical Galois theory cannot be invoked. 

We resolve the apparent conundrum by constructing a  partial left adjoint using Bardavid's idea of \emph{categorical scheme of leaves} \cite{bardavid}, and by using an \emph{indexed version} of categorical Galois theory from \cite{borceux-janelidze},  
eventually proving Theorem~\ref{qproj-gal-th} that simultaneously explains the quasi-projective aspect of Picard-Vessiot theory and applies to strongly normal theory. 

We proceed much further and develop a Galois theory of arbitrary differential scheme morphisms, as explained below. We hope that our work fulfils some of the wishes of Jerry Kovacic to bring the techniques of modern algebraic geometry to differential Galois theory. 


\subsection{Differential algebraic geometry}

In this paper, a differential scheme 
$$
(X,\delta_X)=(X,(\cO_X,\delta_X))
$$
is a scheme $(X,\cO_X)$ endowed with a derivation $\delta_X$ on its structure sheaf $\cO_X$, i.e., with a vector field.

Given a scheme $S$, we define a category of $S$-differential schemes 
$$
\DSch_S
$$
consisting of schemes over $S$ endowed by $S$-derivations. 

The spectrum of a differential ring has a natural structure of a differential scheme, affording a right adjoint
$$
\spec:\DRng^\op\to \DSch
$$
to the global sections functor. A differential scheme is affine, if it is isomorphic to a spectrum of a differential ring, and the category of affine differential schemes is anti-equivalent to the category of differential rings, i.e., 
$$
\DAff\simeq \DRng^\op,
$$
which is consistent with our discussion of affine differential schemes above. 

The natural functor
$$
C:\Sch\to \DSch,  \ \ \ \ (X,\cO_X)\mapsto (X,(\cO_X,0))
$$
allows us to consider a scheme as a differential scheme with the trivial derivation/zero vector field, but it does not have a left adjoint. 
Based on the ideas of Bardavid \cite[4.2]{bardavid}, we define a \emph{categorical scheme of leaves}, or a \emph{categorical quotient}
$$\pi_0(X)$$ of a differential scheme $(X,\delta_X)$ to be a morphism 
$$
(X,\delta_X)\to C(\pi_0(X))
$$
that is universal from $(X,\delta_X)$ to $C$. It exists only for rare differential schemes $(X,\delta_X)$, but this construction provides a partial left adjoint $\pi_0$ to $C$ when it does, as needed for the indexed version of categorical Galois theory. 

In this paper, we exploit the interplay between two perspectives on differential schemes. We consider them as precategory actions to access techniques of descent theory, and as formal group actions to construct an algebraic geometric approach to schemes of leaves. 


\subsection{Differential schemes as precategory actions and descent}

We make a key novel observation that $S$-differential schemes can be viewed as actions of an internal precategory 
$$
\bbD(S)
$$
in $\Sch_{\ov S}$ associated to infinitesimal augmentations of $S$ that we can symbolically write
$$
 \begin{tikzpicture} 
\matrix(m)[matrix of math nodes, row sep=0em, column sep=1em, text height=1.5ex, text depth=0.25ex]
 {
|(2)|{S[\epsilon_0,\epsilon_1]/(\epsilon_0^2,\epsilon_0\epsilon_1,\epsilon_1^2)}  & [1em] |(1)|{S[\epsilon]/(\epsilon^2)}		&[1em] |(0)|{S} \\
 }; 
\path[->,font=\scriptsize,>=to, thin]
([yshift=.4em]2.east) edge node[above=-2pt]{} ([yshift=.4em]1.west) 
(2) edge node[above=-2pt]{} (1)
([yshift=-.4em]2.east) edge node[above=-2pt]{} ([yshift=-.4em]1.west) 
([yshift=.4em]1.east) edge node[above=-2pt]{} ([yshift=.4em]0.west) 
(0)  edge node[above=-2pt]{} (1) 
([yshift=-.4em]1.east) edge node[above=-2pt]{} ([yshift=-.4em]0.west) 
;
\end{tikzpicture}
$$
i.e., that there is an equivalence of categories (\ref{diff-sch-precats})
$$
\DSch_S\simeq (\Sch_{\ov S})^{\bbD(S)}.
$$
It was previously known that differential schemes can be viewed as `actions of a pointed set' associated to the augmentation structure $S[\epsilon]/(\epsilon^2)\to S$, but the observation that we can view them as precategory actions gives us a very direct access to descent theory, given that precategory actions are essentially generalised descent data. More generally, we can translate problems from differential algebraic geometry into the familiar realm of algebraic geometry. 

In algebraic geometry, descent usually works only for very specific indexed data on schemes, so we work with a chosen pseudofunctor
$$
\cP:(\Sch_{\ov S})^\op\to \Cat,
$$
and extend it in a natural way, using precategory actions, to a pseudofunctor on differential schemes
$$
\delta\da\cP:\DSch_S^\op\to \Cat. 
$$

In the special case when $\cP(V)$ is a certain class of scheme morphisms with target $V$, the category $\delta\da\cP(X,\delta_X)$ consists of differential scheme morphisms with target $(X,\delta_X)$ whose underlying scheme morphism belongs to the class $\cP$. 
The most important instances of $\cP$ for our applications will come from the class of quasi-projective and polarised quasi-projective morphisms. We will use the most powerful classical descent results identifying the morphisms of effective descent in those cases from \cite{sga1}.

If $\bbX=(X_2,X_1,X_0)$ is the precategory in $\Sch_{\ov S}$ that corresponds to a differential scheme $(X,\delta_0)$ viewed as an $\bbD(S)$-action in the sense of \ref{s:diffschprect}, then the categorical scheme of leaves of $(X,\delta_X)$ agrees with the scheme of connected components of $\bbX$,
$$
\begin{tikzcd}[cramped, column sep=normal, ampersand replacement=\&]
{\pi_0(X)\simeq\pi_0(\bbX)=\Coeq\left(X_1 \right.}\ar[yshift=2pt]{r}{d_0} \ar[yshift=-2pt]{r}[swap]{d_1} \&{\left.X_0\right)}
\end{tikzcd}
$$
whenever either of them exist in $\Sch_{\ov S}$. 

We define $(X,\delta_X)$ to be \emph{simple} with respect to $\cP$, provided the above coequaliser exists and is \emph{universal} for $\cP$. 
This condition ensures that the precategory morphism
$$
\eta_X:\bbX\to \Disc{\pi_0(\bbX)}
$$
to the dicrete precategory associated to the scheme $\pi_0(X)$
is of \emph{precategorical descent}, i.e., the induced functor of precategory actions
$$
C_X:\cP(\pi_0(X))\simeq \cP^{\Disc{\pi_0(\bbX)}}\stackrel{\eta_{\bbX}^*}{\longrightarrow}\cP^{\bbX}\simeq \delta\da\cP(X,\delta_X)
$$
is fully faithful. 

By developing the theory of descent in the category of differential schemes as precategory actions, our main result in this direction \ref{diffl-desc} is that a morphism
$$
f:(X,\delta_X)\to (Y,\delta_Y)
$$
is of effective descent for $\delta\da\cP$, if the associated morphism $(f_0,f_1,f_2)$ of precategory actions satisfies that the underlying scheme morphism $f_0:X\to Y$ is of effective descent for $\cP$, and $f_1$ is descent for $\cP$. 

Consequently, such an $f$ is of effective descent for the fibration of polarised quasi-projective differential morphisms if $f_0$ is faithfully flat quasi-compact (fpqc). Moreover, if the target of $f$ is the spectrum of a differential field, then it is of effective descent for the class of quasi-projective differential morphisms. 

\subsection{Differential schemes as formal group actions and geometric quotients}

We can also view differential schemes as actions of the formal additive group 
$$
\Ga
$$
and apply the principles from geometric invariant theory from \cite{mumford-git} and ideas of Bardavid \cite{bardavid} to 
refine the concept of the categorical scheme of leaves of a differential scheme.

A morphism
$$
\eta:(X,\delta_X)\to (Q,0)
$$
is a \emph{geometric quotient} if the underlying scheme morphism is a quotient in the topological sense, its geometric fibres agree with the scheme-theoretic orbits under the action, and the structure sheaf of $Q$ coincides with the constants of the structure sheaf of $X$, as specified in  \ref{geom-quot}. 

It follows that a geometric scheme of leaves $Q$ is also a categorical scheme of leaves, i.e., 
$$Q\simeq \pi_0(X).$$ 

In order to prove the \emph{universality} of geometric quotients, i.e., their stability under base change, we use  various tools of algebraic geometry, such as base-change theorems and technical results on cohomologically flat morphisms due to Grothendieck \cite{ega3.2} and more recent extensions of the theory from \cite{gabber-lodh}. We obtain effective criteria for universal geometric quotients in terms of good global behaviour of $\eta$ and classical simplicity of its geometric fibres in \ref{univ-geom-quot}. 

In turn, this yields \emph{effective} criteria for categorical simplicity essential for the subsequent development of Galois theory,


%

\subsection{Galois theory of differential schemes}

A good understanding of classical and precategorical descent, and of simplicity, provides all the ingredients needed for a Picard-Vessiot style categorical Galois theory of differential schemes. The framework for the indexed version of categorical Galois theory needed for this purpose involves the following choices:
\begin{enumerate}
\item   a scheme $S$, and $$\cX=\Sch_{\ov S}$$ be the category of schemes over $S$;
\item  the full subcategory $\cA$ of $\DSch_S$ of those $S$-differential schemes that admit a categorical scheme of leaves, hence we obtain a functor
$$
\pi_0:\cA\to \cX;
$$
\item a pseudo-functor $$\cP:\cX^\op\to\Cat,$$ determining a pseudofunctor $\delta\da\cP:\cA^\op\to\Cat$;
\item the pseudo-natural transformation
$$C=C^\cP:\cP\circ\pi_0\to \delta\da\cP$$ 
whose component at $(X,\delta)\in\cA$ is the functor $C_X$ discussed above.
\end{enumerate}

We define the category 
$$
\Split_C(f)
$$
as the full subcategory of objects $P\in\delta\da\cP(Y)$ such that 
$f^*P\simeq C_X(Q)$ for some $Q\in\cP(\pi_0(X))$.

 
Using this notation, we can formulate the key concepts of the paper. 

 \begin{definition*}
 We define a morphism $f:(X,\delta_X)\to (Y,\delta_Y)$ in $\cA$ to be \emph{pre-Picard-Vessiot} with respect to $\cP$, if
 \begin{enumerate}
 \item[(i)]\label{cls-desc} $f$ is a morphism of effective descent for $\delta\da\cP$;
 \item[(ii)]\label{cube-simple} $X$, $X\times_YX$, $X\times_YX\times_YX$ are simple for $\cP$.
 \end{enumerate}
 \end{definition*}

If we have 
 \begin{enumerate}
 \item[(5)] a suitable (in the sense of \ref{sch-pv-setup}) pseudo-natural transformation 
 $$U:\cP\Rightarrow\cS$$ to a fibration $\cS$ associated with a class of scheme morphisms,
 \end{enumerate}
 then we can make the following definition. 
 \begin{definition*}
 A morphism $f$ as above is \emph{Picard-Vessiot} with respect to $U$ if, in addition to the first, it satisfies the strengthening of  the second condition, 
  \begin{enumerate}
 \item[(ii')] $X$ is simple for $\cS$ and $f$ is auto-split for $\cS$ in the sense that $f\in\Split_{C^\cS}(f)$. 
 \end{enumerate}
\end{definition*}
 
 The assumption that $f$ is pre-Picard-Vessiot for $\cP$ ensures that the kernel-pair groupoid 
$$
 \begin{tikzpicture} 
\matrix(m)[matrix of math nodes, row sep=0em, column sep=1em, text height=1.5ex, text depth=0.25ex]
 {
|(n)|{\bbG_f:} &[1em] |(2)|{X\times_Y X\times_Y X}  & [1em] |(1)|{X\times_Y X}		&[1em] |(0)|{X} \\
 }; 
\path[->,font=\scriptsize,>=to, thin]
([yshift=.4em]2.east) edge node[above=-2pt]{} ([yshift=.4em]1.west) 
(2) edge node[above=-2pt]{} (1)
([yshift=-.4em]2.east) edge node[above=-2pt]{} ([yshift=-.4em]1.west) 
([yshift=.4em]1.east) edge node[above=-2pt]{} ([yshift=.4em]0.west) 
(0)  edge node[above=-2pt]{} (1) 
([yshift=-.4em]1.east) edge node[above=-2pt]{} ([yshift=-.4em]0.west) 
;
\end{tikzpicture}
$$
is a category in $\cA$, and that we can define the \emph{Galois precategory}
$$
\Gal[f]=\pi_0(\bbG_f)
$$
as a precategory in $\cX$. Moreover, if $f$ is Picard-Vessiot for $U$, then it is also pre-Picard-Vessiot for $\cP$ and $G[f]$ happens to be an internal \emph{groupoid} in $\cX$.

\begin{theorem*}[Galois theorem for differential schemes, \ref{scheme-dif-Galois}]  
A pre-Picard-Vessiot morphism $f$ for $\cP$ induces an equivalence 
$$
\Split_C(f)\simeq \cP^{\Gal[f]}
$$
between the category of objects of $\delta\da\cP(Y)$ that are $C$-split by $f$ and the category of $\cP$-actions of the precategory $\Gal[f]$.

If $f$ is Picard-Vessiot for $U$, 
the latter becomes the category of $\cP$-actions of the groupoid $\Gal[f]$. 
\end{theorem*}

From our perspective, 

\medskip
\centerline{\emph{Differential Galois theory $=$ (precategorical) descent $+$ categorical Galois theory. }}
\medskip
More precisely, condition (i) from the definition of pre-Picard-Vessiot morphisms ensures effective classical descent for precategory actions, while condition (ii) ensures the new type of precategorical descent  needed to apply the powerful abstract techniques of categorical Galois theory.

Our Galois precategory appears to be a differential scheme version of truncated to level 2 higher differential Galois groups developed by Ayoub in \cite{ayoub} as foliated homotopy types of a differential field. We hope to be able to pursue our techniques in the direction of higher category theory along similar lines. 
 
 \subsection{Application: quasi-projective differential Galois theory}
 
 Using the above template Galois theorem, we prove the following result, which simultaneously refines classical Picard-Vessiot theory and strongly normal theory of differential field extensions. 
 
\begin{corollary*}[Quasi-projective differential Galois correspondence~\ref{qproj-pv-corr}]
Let $(K,\delta)$ be a differential field of characteristic $0$ with the field of constants $k$, let $S=\spec(k)$ and assume
$$
f:(X,\delta)\to (Y,\delta)=\spec(K,\delta)
$$
is a quasi-projective morphism of $S$-differential schemes such that $X$ is integral and its only leaf is its generic point and $f$ is auto-split.  Then $f$ is Picard-Vessiot, and there is a Galois correspondence between split $S$-differential quasi-projective fpqc quotients of $f$ in $\cA$ and closed subgroups of the $S$-algebraic group $\Gal[f]$, which takes
$$
 \begin{tikzpicture}
[cross line/.style={preaction={draw=white, -,
line width=4pt}}]
\matrix(m)[matrix of math nodes, row sep=.9em, column sep=.5em, text height=1.5ex, text depth=0.25ex]
{			& |(ij)| {X}	&				\\   [.02em]
	|(i)|{P} 	&			& 		\\  [.02em]
			& |(0)|{Y} 		&				\\};
\path[->,font=\scriptsize,>=to, thin]
(ij) edge node[left,pos=0.2]{$$} (i) edge node[right,pos=0.5]{$f$} (0) 
(i) edge node[below left,pos=0.5]{$p$} (0) 
;
\end{tikzpicture}
$$ 
to $$\Gal[X\to P].$$
Conversely, a closed subgroup $G'$ of $\Gal[f]$ corresponds to the quotient
$$
X/G'.
$$
\end{corollary*}
 
To the best of our knowledge,  the following is the first result in differential Galois theory which applies to differential schemes with arbitrary categorical  scheme of constants.

\begin{theorem*}[Polarised quasi-projective differential Galois theory, \ref{polarised-pv-thm}]

Let $f:(X,\delta_X)\to (Y,\delta_Y)$ be a morphism of $S$-differential schemes such that 
\begin{enumerate}
\item the underlying scheme morphism $X\to Y$ is fpqc;
\item $(X,\delta_X)$ is simple with respect to $S$-scheme morphisms, with scheme of leaves $G_0$;
\item there is an $S$-morphism $G_1\to G_0$ such that
$$
(X,\delta_X)\times_{(Y,\delta_Y)}(X,\delta_X)\simeq (X,\delta_X)\times_{(G_0,0)}(G_1,0).
$$
\end{enumerate}
Then $f$ is Picard-Vessiot for the forgetful functor $U$ from polarised morphisms to scheme morphisms, $\Gal[f]$ is the groupoid $(G_1\doublerightarrow{}{} G_0)$ and we have an equivalence between the category of quasi-projective polarised $S$-differential morphisms $(P,\cL_P)\to Y$ split by $f$ and the category of quasi-projective polarised actions $(Q,\cL_Q)\to G_0$ of $\Gal[f]$.
\end{theorem*}

\subsection{Application: Parametric differential equations.}

If $f:X\to Y$ is a pre-Picard-Vessiot morphisms with respect to some class of morphisms $\cS$, by using the full strength of categorical simplicity, we prove the specialisation formula \ref{special-gal} for the Galois precategory,
$$
\Gal[q^*(f)]\simeq q^*\Gal[f],
$$
where the pullbacks are taken along a morphism $q:Q\to \pi_0(Y)$ from $\cS$. 

This result makes our Galois theory suitable for treating parametric systems of differential equations where the parameters come from a constant scheme. We illustrate this in \ref{s:elliptic} on an example of an $S$-differential family $f:X\to Y$ of elliptic curves endowed with a vector field in such a way that $\pi_0(X)=\pi_0(Y)=S$ and the Galois groupoid is a family of elliptic curves over $S$
$$
\Gal[f]=E\doublerightarrow{}{} S.
$$
As the parameter $s$ varies over $S(L)$, the specialisation formula \ref{special-gal} gives that 
$$
\Gal[f]_s=\Gal[f_s]=E_s,
$$
as an algebraic group over $L$, so our Galois groupoid specialises to classical Galois groups of strongly normal extensions associated with $f_s$ calculated by Kolchin in \cite{kolchin-sn}.

\subsection{Application: the Galois groupoid of a differential equation.}
In \ref{s:airy}, we show that our notion of a Picard-Vessiot morphism can be used to study symmetries of linear differential equations through a canonical Galois groupoid, without making a non-canonical choice of a Picard-Vessiot extension as in the classical theory. 

We study a Picard-Vessiot morphism of differential schemes $f:X\to Y$ where $X$ is associated to the full universal solution algebra of the Airy equation
$$
y''=xy,
$$
which yields a Galois groupoid of the form
$$
\Gal[f]=G_1\doublerightarrow{}{} G_0,
$$
where the points of the object of objects $G_0$ correspond to choices of Picard-Vessiot extensions, and the object of morphisms $G_1$ encodes isomorphisms between different choices of Picard-Vession extensions. 

This is a scheme-theoretic generalisation of Deligne's notion of the fundamental groupoid of a Tannakian category from \cite{deligne-tannakien}, where the object of objects consists of fibre functors, and the object of morphisms consists of isomorphisms between fibre functors. Indeed, given the Tannakian category of finite-dimensional vector spaces with a connection over a differential field, fibre functors correspond to Picard-Vessiot extensions, and we obtain a perfect analogy with our groupoid.

 \subsection{Layout of the paper}

A reader interested only in differential Galois theory can start perusing the paper from Section~\ref{s:dif-alg}, and occasionally look up the prerequisites from previous sections.  There, we develop differential algebraic geometry, and discuss the numerous benefits of our approach to differential schemes as precategory actions, including the consideration of the categorical schemes of leaves as connected components of precategories, simplicity of differential schemes through universality of connected components and polarised differential schemes. One of the most important topics we develop here is the theory of descent for differential scheme morphisms.

In the same section, we also view differential schemes as actions of the formal additive group and develop the theory of their geometric quotients. Moreover, we identify effective criteria  for geometric quotients to be universal, providing us with a toolkit for proving simplicity of differential schemes in practice. 

In Section~\ref{s:aff-PV}, we explore the extent to which the classical categorical Galois theory of Janelidze explains the affine Picard-Vessiot Galois correspondence. 

In Section~\ref{s:sat-gal-dif}, we reap the benefits of the flexibility that the indexed version of categorical Galois theory brings to differential algebraic geometry, and formulate a very general template for the Galois theory of differential schemes. 

In Section~\ref{s:applications}, we apply this theory in concrete settings of quasi-projective differential Galois correspondence that works over a base differential field, and polarised quasi-projective differential Galois theory that works over arbitrary differential schemes. We also provide concrete examples that illustrate the scope of our theory. 

On the other hand, we must emphasise that previous sections contain several results that may be of independent interest in category theory, especially in view of connections to recent work related to descent theory 
\cite{le-creurer}, 
\cite{nunes}, 
\cite{nunes-prezado-sousa}  
\cite{prezado-nunes}, 
that we do not fully understand yet. 

In Section~\ref{s:descent}, we discuss precategories and their actions, which, viewed as generalised descent data, give rise to a new form of \emph{precategorical descent,} and find sufficient conditions for effective descent. We develop a whole calculus of pullbacks of descent data, and apply it to obtain a result on descent of quasi-projective morphisms. We also discuss classical descent of precategory actions that gets applied in the differential context later.

In Section~\ref{s:cat-gal}, we recall the foundations of Janelidze's categorical Galois theory, its indexed form from \cite{borceux-janelidze}, and we expand the minute details of the Carboni-Magid-Janelidze categorical Galois correspondence that are implicit in the original paper \cite{carboni}. 

This paper was largely motivated by numerous discussions with George Janelidze, Andy Magid, Akira Masuoka,  Tom Scanlon and Michael Wibmer, so we thank them all for sharing their time and knowledge with us. 

\section{Descent}\label{s:descent}

Throughout this appendix, let $\cC$ denote a category with pullbacks, and let
$$
\cP\to \cC
$$
be a fibred category equipped with a cleavage. Equivalently, we have an indexed category associated to a pseudofunctor
$$
P:\cC^\op\to \Cat
$$
where, for an object $U$ in $\cC$, the fibre of $\cP$ over $U$ is the $U$-indexed component of $P$,
$$
\cP(U)=P(U),
$$
and, for a morphism $U\stackrel{f}{\to}V$ in $\cC$, we have a pullback functor
$$
f^*=P(f):\cP(V)\to\cP(U).
$$

\subsection{Precategories and their actions}

Let $\Delta_2$ be the diagram category
$$
 \begin{tikzpicture} 
\matrix(m)[matrix of math nodes, row sep=0em, column sep=3em, text height=1.5ex, text depth=0.25ex]
 {
|(2)|{\bullet}  & [1em] |(1)|{\bullet}		&[1em] |(0)|{\bullet} \\
 }; 
 \node at (2.90) {\scriptsize{2}};
 \node at (1.90) {\scriptsize{1}};
 \node at (0.90) {\scriptsize{0}};
\path[->,font=\scriptsize,>=to, thin]
([yshift=1em]2.east) edge node[above=-2pt]{$r_0$} ([yshift=1em]1.west) 
(2) edge node[above=-2pt]{$m$} (1)
([yshift=-1em]2.east) edge node[above=-2pt]{$r_1$} ([yshift=-1em]1.west) 
([yshift=1em]1.east) edge node[above=-2pt]{$d_0$} ([yshift=1em]0.west) 
(0)  edge node[above=-2pt]{$n$} (1) 
([yshift=-1em]1.east) edge node[above=-2pt]{$d_1$} ([yshift=-1em]0.west) 
;
\end{tikzpicture}
$$
with
\begin{align*}
d_0 r_1 &=d_1 r_0, 	& d_0 n &= \id_0, \\
d_0 m &= d_0 r_0,	& d_1 n &=\id_0, \\
d_1 m &= d_1 r_1. 	&\\
\end{align*}

A \emph{precategory} in a category $\cC$ is a functor
$$
\C:\Delta_2\to \cC.
$$
Equivalently, it is a diagram 
$$
 \begin{tikzpicture} 
\matrix(m)[matrix of math nodes, row sep=0em, column sep=3em, text height=1.5ex, text depth=0.25ex]
 {
|(2)|{C_2}  & [1em] |(1)|{C_1}		&[1em] |(0)|{C_0} \\
 }; 
\path[->,font=\scriptsize,>=to, thin]
([yshift=1em]2.east) edge node[above=-2pt]{$r_0$} ([yshift=1em]1.west) 
(2) edge node[above=-2pt]{$m$} (1)
([yshift=-1em]2.east) edge node[above=-2pt]{$r_1$} ([yshift=-1em]1.west) 
([yshift=1em]1.east) edge node[above=-2pt]{$d_0$} ([yshift=1em]0.west) 
(0)  edge node[above=-2pt]{$n$} (1) 
([yshift=-1em]1.east) edge node[above=-2pt]{$d_1$} ([yshift=-1em]0.west) 
;
\end{tikzpicture}
$$
in $\cC$, where the morphisms satisfy the relations indicated above. 

The category of precategories in $\cC$ is the functor category
$$
\Precat(\cC)=[\Delta_2,\cC].
$$

\begin{definition}\label{precat-discr}
The discrete precategory associated to an object $U$ of $\cC$ is the diagram
$$
 \begin{tikzpicture} 
\matrix(m)[matrix of math nodes, row sep=0em, column sep=3em, text height=1.5ex, text depth=0.25ex]
 {
|(n)|{\Disc{U}:} &[1em] |(2)|{U}  & [1em] |(1)|{U}		&[1em] |(0)|{U} \\
 }; 
\path[->,font=\scriptsize,>=to, thin]
([yshift=.4em]2.east) edge node[above=-2pt]{} ([yshift=.4em]1.west) 
(2) edge node[above=-2pt]{} (1)
([yshift=-.4em]2.east) edge node[above=-2pt]{} ([yshift=-.4em]1.west) 
([yshift=.4em]1.east) edge node[above=-2pt]{} ([yshift=.4em]0.west) 
(0)  edge node[above=-2pt]{} (1) 
([yshift=-.4em]1.east) edge node[above=-2pt]{} ([yshift=-.4em]0.west) 
;
\end{tikzpicture}
$$
where all the morphisms are identities on $U$. 
\end{definition}

\begin{definition}\label{precat-ker-pair}
The precategory associated to the kernel pair of a morphism $u:U'\to U$ in a category $\cC$ admitting pullbacks is the diagram
$$
 \begin{tikzpicture} 
\matrix(m)[matrix of math nodes, row sep=0em, column sep=3em, text height=1.5ex, text depth=0.25ex]
 {
|(n)|{\bbG_u:} &[1em] |(2)|{U'\times_U U'\times_U U'}  & [1em] |(1)|{U'\times_UU'}		&[1em] |(0)|{U'} \\
 }; 
\path[->,font=\scriptsize,>=to, thin]
([yshift=1em]2.east) edge node[above=-2pt]{$\pi_{01}$} ([yshift=1em]1.west) 
(2) edge node[above=-2pt]{$\pi_{02}$} (1)
([yshift=-1em]2.east) edge node[above=-2pt]{$\pi_{12}$} ([yshift=-1em]1.west) 
([yshift=1em]1.east) edge node[above=-2pt]{$\pi_0$} ([yshift=1em]0.west) 
(0)  edge node[above=-2pt]{$\Delta$} (1) 
([yshift=-1em]1.east) edge node[above=-2pt]{$\pi_1$} ([yshift=-1em]0.west) 
;
\end{tikzpicture}
$$
which happens to be a groupoid without the inversion of morphisms named. Note that, using the notation from \ref{precat-discr},
$$
\Disc{U}=\bbG_{\id_U}.
$$
\end{definition}

\begin{definition}
For $\C\in\Precat(\cC)$, the category of \emph{$\C$-actions in $\cP$} is the bilimit
$$
\cP^\C=\lim(\cP\circ\C)
$$
of the diagram of categories and functors
$$
 \begin{tikzpicture} 
\matrix(m)[matrix of math nodes, row sep=0em, column sep=3em, text height=1.5ex, text depth=0.25ex]
 {
|(2)|{\cP(C_2)}  & [1em] |(1)|{\cP(C_1)}		&[1em] |(0)|{\cP(C_0).} \\
 }; 
\path[->,font=\scriptsize,>=to, thin]
([yshift=1em]1.west) edge node[above=-2pt]{$r_0^*$} ([yshift=1em]2.east) 
(1) edge node[above=-2pt]{$m^*$} (2)
([yshift=-1em]1.west) edge node[above=-2pt]{$r_1^*$} ([yshift=-1em]2.east) 

 ([yshift=1em]0.west) edge node[above=-2pt]{$d_0^*$} ([yshift=1em]1.east) 
(1)  edge node[above=-2pt]{$n^*$} (0) 
([yshift=-1em]0.west) edge node[above=-2pt]{$d_1^*$} ([yshift=-1em]1.east) 
;
\end{tikzpicture}
$$
More explicitly, an action $P\in \cP^\C$ consists of objects
\begin{align*}
P_2&\in \cP(C_2), & P_1&\in\cP(C_1), & P_0&\in \cP(C_0),
\end{align*}
and isomorphisms
\begin{align*}
P_2 &\stackrel{\rho_0}{\to} r_0^* P_1, & P_1&\stackrel{\delta_0}{\to}d_0^*P_0, & P_0\stackrel{\eta}{\to}n^*P_1,\\
P_2 &\stackrel{\mu}{\to} m^* P_1, & P_1&\stackrel{\delta_1}{\to}d_1^*P_0, & \\
P_2 &\stackrel{\rho_1}{\to} r_1^* P_1, & &  & \\
\end{align*}
such that the diagram
$$
 \begin{tikzpicture}
[cross line/.style={preaction={draw=white, -,
line width=4pt}}]
\matrix(m)[matrix of math nodes, row sep=1.5em, column sep=1.7em, text height=1.5ex, text depth=0.25ex]
{	|(i)|{n^*d_0^*P_0} 	&	|(c)|{n^*P_1}		& |(k)|{n^*d_1^*P_0} 		\\  
			& |(j)|{P_0} 	&				\\};
\path[->,font=\scriptsize,>=to, thin]
(c) edge node[above,pos=0.5]{$n^*\delta_0$} (i) edge node[above,pos=0.5]{$n^*\delta_1$} (k) 
(j) edge node[above=-2pt, sloped]{$\sim$} (i) edge node[right,pos=0.4]{$\eta$} (c) edge node[above=-2pt, sloped]{$\sim$} (k) 
;
\end{tikzpicture}
$$ 
commutes in $\cP(C_0)$, and the diagram
$$
 \begin{tikzpicture}
[cross line/.style={preaction={draw=white, -,
line width=4pt}}]
\matrix(m)[matrix of math nodes, row sep=1.2em, column sep=.8em, text height=1.5ex, text depth=0.25ex]
{					&					& |(m0)|{m^*d_0^*P_0}	&[1em] |(mm)|{m^*P_1} 	&[1em] |(m1)|{m^*d_1^*P_0} 	&  			&				\\   [2em]
|(r0d0)|{r_0^*d_0^*P_0}  	&					& 					& |(p2)|{P_2} 		& 					&			&|(r1d1)|{r_1^*d_1^*P_0} 		\\  [.2em]
					&	|(r0)|{r_0^*P_1} 	& 					& 				& 					& |(r1)|{r_1^*P_1} & 		 		\\  [.02em]
					&					& |(r0d1)|{r_0^*d_1^*P_0} 	& 				&  |(r1d0)|{r_1^*d_0^*P_0}	& 			 & 		 		\\};
\path[->,font=\scriptsize,>=to, thin]
(r0) edge node[below left,pos=0.5]{$r_0^*\delta_0$} (r0d0) edge node[below left,pos=0.5]{$r_0^*\delta_1$} (r0d1)
(mm) edge node[above,pos=0.5]{$m^*\delta_0$} (m0) edge node[above, pos=0.5]{$m^*\delta_1$} (m1) 
(r1) edge node[below right,pos=0.5]{$r_1^*\delta_0$} (r1d0) edge node[below right,pos=0.5]{$r_1^*\delta_1$} (r1d1)
(p2) edge node[right]{$\mu$} (mm) edge node[above]{$\rho_0$} (r0) edge node[above]{$\rho_1$} (r1) 
(r0d0) edge node[above=-2pt, sloped]{$\sim$}  (m0) 
(m1) edge node[above=-2pt, sloped]{$\sim$}  (r1d1) 
(r0d1) edge node[above=-2pt, sloped]{$\sim$}  (r1d0) 
;
\end{tikzpicture}
$$ 
commutes in $\cP(C_2)$, where the unnamed arrows are coherence isomorphisms. 
\end{definition}

\begin{definition}\label{actions-self-ind}
The category of \emph{$\C$-actions in $\cC$} is  
$$
\cC^\C=\mathop{\rm Self}(\cC)^\C,
$$
where $\mathop{\rm Self}(\cC)(C)=\cC_{\ov C}$ is the self-indexing of $\cC$ by slicing. 
\end{definition}

\begin{remark}\label{gen-dd}
By considering the composite
$$
 \begin{tikzpicture}
[cross line/.style={preaction={draw=white, -,
line width=4pt}}]
\matrix(m)[matrix of math nodes, row sep=1.5em, column sep=1.7em, text height=1.5ex, text depth=0.25ex]
{	|(i)|{d_0^*P_0} 	&	|(c)|{P_1}		& |(k)|{d_1^*P_0} 		\\};
\path[->,font=\scriptsize,>=to, thin]
(c) edge node[above,pos=0.3]{$\delta_0$} (i) edge node[above,pos=0.3]{$\delta_1$} (k) 
(i) edge[bend left=30] node[above,pos=0.5]{$\varphi$} (k)
;
\end{tikzpicture}
$$ 
we see that an action $P$ is equivalently given by an object $P_0\in \cP(C_0)$ and an isomorphism $\varphi:d_0^*P_0\to d_1^*P_0$ in $\cP(C_1)$ such that the diagram
$$
 \begin{tikzpicture}
[cross line/.style={preaction={draw=white, -,
line width=4pt}}]
\matrix(m)[matrix of math nodes, row sep=1.5em, column sep=1.7em, text height=1.5ex, text depth=0.25ex]
{	|(i)|{n^*d_0^*P_0} 	&			& |(k)|{n^*d_1^*P_0} 		\\  
			& |(j)|{P_0} 	&				\\};
\path[->,font=\scriptsize,>=to, thin]
(i) edge node[above,pos=0.5]{$n^*\varphi$} (k) 
(j) edge node[above=-2pt, sloped]{$\sim$} (i)  edge node[above=-2pt, sloped]{$\sim$} (k) 
;
\end{tikzpicture}
$$ 
commutes in $\cP(C_0)$, and the \emph{cocycle condition} holds, i.e., the diagram
$$
 \begin{tikzpicture}
[cross line/.style={preaction={draw=white, -,
line width=4pt}}]
\matrix(m)[matrix of math nodes, row sep=1.2em, column sep=.8em, text height=1.5ex, text depth=0.25ex]
{					&					& |(m0)|{m^*d_0^*P_0}	&[1em] 	&[1em] |(m1)|{m^*d_1^*P_0} 	&  			&				\\   [2em]
|(r0d0)|{r_0^*d_0^*P_0}  	&					& 					& 				& 					&			&|(r1d1)|{r_1^*d_1^*P_0} 		\\  [.2em]
					&					& 					& 				& 					&  			& 		 		\\  [.02em]
					&					& |(r0d1)|{r_0^*d_1^*P_0} 	& 				&  |(r1d0)|{r_1^*d_0^*P_0}	& 			 & 		 		\\};
\path[->,font=\scriptsize,>=to, thin]
 (r0d0) edge node[below left,pos=0.5]{$r_0^*\varphi$} (r0d1)
 (m0) edge node[above, pos=0.5]{$m^*\varphi$} (m1) 
(r1d0) edge node[below right,pos=0.5]{$r_1^*\varphi$} (r1d1)
(r0d0) edge node[above=-2pt, sloped]{$\sim$}  (m0) 
(m1) edge node[above=-2pt, sloped]{$\sim$}  (r1d1) 
(r0d1) edge node[above=-2pt, sloped]{$\sim$}  (r1d0) 
;
\end{tikzpicture}
$$ 
commutes in $\cP(C_2)$.

A morphism $$(P,\varphi)\to (P',\varphi')$$
between two $\C$-actions considered this way is given by a morphism $\psi: P_0\to P_0'$ in $\cP(C_0)$ such that the diagram
$$
 \begin{tikzpicture}
[cross line/.style={preaction={draw=white, -,
line width=4pt}}]
\matrix(m)[matrix of math nodes, row sep=1.8em, column sep=1.8em, text height=1.5ex, text depth=0.25ex]
{	|(0i)|{d_0^*P}		& |(1i)| {d_0^*P'}	\\   [.02em]
	|(0j)|{d_1^*P}		& |(1j)| {d_1^*P'}	\\};
\path[->,font=\scriptsize,>=to, thin]
(0i) edge node[above,pos=0.5]{$d_0^*\psi$} (1i) edge node[left,pos=0.5]{$\varphi$} (0j) 
(0j) edge node[above,pos=0.5]{$d_1^*\psi$} (1j) 
(1i) edge node[right,pos=0.5]{$\varphi'$} (1j)
;
\end{tikzpicture}
$$ 
commutes in $\cP(C_1)$. 
\end{remark}

\begin{remark}\label{action-fibred}
In terms of the fibration $\cP\to \cC$, an action in $\cP^\C$ is given by a precategory $\bbP\in\Precat(\cP)$ 
$$
 \begin{tikzpicture} 
\matrix(m)[matrix of math nodes, row sep=0em, column sep=3em, text height=1.5ex, text depth=0.25ex]
 {
|(2)|{P_2}  & [1em] |(1)|{P_1}		&[1em] |(0)|{P_0} \\
 }; 
\path[->,font=\scriptsize,>=to, thin]
([yshift=1em]2.east) edge node[above=-2pt]{$r_0$} ([yshift=1em]1.west) 
(2) edge node[above=-2pt]{$m$} (1)
([yshift=-1em]2.east) edge node[above=-2pt]{$r_1$} ([yshift=-1em]1.west) 
([yshift=1em]1.east) edge node[above=-2pt]{$d_0$} ([yshift=1em]0.west) 
(0)  edge node[above=-2pt]{$n$} (1) 
([yshift=-1em]1.east) edge node[above=-2pt]{$d_1$} ([yshift=-1em]0.west) 
;
\end{tikzpicture}
$$
living above $\C$, where the arrows in $\bbP$ are cartesian. 
\end{remark}

\subsection{Connected components of precategories}

\begin{definition}\label{def-conn-comp}
The \emph{object of connected components} of a precategory $\C\in\Precat(\cC)$ is the reflexive coequaliser
$$
\begin{tikzcd}[cramped, column sep=2em, ampersand replacement=\&]
C_1 \ar[yshift=2pt]{r}{d_0} \ar[yshift=-2pt]{r}[swap]{d_1} \& C_0 \ar{r} \& \pi_0(\C),
\end{tikzcd}
$$  
provided it exists in $\cC$. 

If $\cP:\cC^\op\to\Cat$ is a pseudofunctor, the object of connected components of an action $P\in\cP^\C$ is defined as the object
of connected components of the precategory $\bbP\in\Precat(\cP)$ associated to $\cP$ via \ref{action-fibred}, 
$$
\pi_0(P)=\pi_0(\bbP),
$$
provided it exists in the fibred category $\cP$ over  $\cC$.
\end{definition}

\begin{lemma}\label{coeq-above-coeq}
With the notation from \ref{def-conn-comp}, suppose that $\pi_0(\C)$ and $\pi_0(P)$ exist. Then $\pi_0(P)$ projects to $\pi_0(\C)$, i.e., a reflexive coequaliser 
$$
\begin{tikzcd}[cramped, column sep=2em, ampersand replacement=\&]
P_1 \ar[yshift=2pt]{r}{d_0} \ar[yshift=-2pt]{r}[swap]{d_1} \& P_0 \ar{r}{\rho} \& \pi_0(P)
\end{tikzcd}
$$  
in $\cP$
projects to a reflexive coequaliser
$$
\begin{tikzcd}[cramped, column sep=2em, ampersand replacement=\&]
C_1 \ar[yshift=2pt]{r}{d_0} \ar[yshift=-2pt]{r}[swap]{d_1} \& C_0 \ar{r}{c} \& \pi_0(\C)
\end{tikzcd}
$$  
in $\cC$.
\end{lemma}

\begin{proof}
Writing $p:\cP\to \cC$ for the given fibration, and using the fact that $c$ is a coequaliser, there exists a unique morphism $t:S\to p(\pi_0(P))$ in $\cC$ such that $p(\rho)=t c$. Writing $\theta: t^*\pi_0(P)\to \pi_0(P)$ for a cartesian lift of $t$, there is a unique $\gamma:P_0\to t^*\pi_0(P)$ such that $p(\gamma)=c$ and $\theta\gamma=\rho$. 

Since $\theta$ is cartesian, we obtain that $\gamma$ coequalises $d_0^P$ and $d_1^P$, so, since $\rho$ is a coequaliser, there exists a unique morphism $\sigma:\pi_0(P)\to t^*\pi_0(P)$ such that $\sigma\rho=\gamma$. 

From the above, $\theta\sigma\rho=\theta\gamma=\rho$, and, since $\rho$ is an epimorphism, we obtain that
$$
\theta\sigma=\id.
$$
Thus, $\theta\sigma\theta=\theta$, and, since $\theta$ is cartesian, in order to show that
$$
\sigma\theta=\id,
$$
it suffices to verify that $p(\sigma\theta)=\id$. Indeed, $p(\sigma\theta)c=p(\sigma)tc=p(\sigma)p(\rho)=p(\sigma\rho)=p(\gamma)=c$, and the conclusion follows since $c$ is an epimorphism. 

\end{proof}

\begin{proposition}\label{part-adj-pi0}
With notation and assumptions of \ref{coeq-above-coeq}, the pullback functor
$$
\eta_\C^*:\cP(\pi_0(\C))\simeq \cP^{\Disc{\pi_0(\C)}}\to \cP^\C
$$
induced by the unique precategory morphism 
$$
\eta_\C:\C\to \Disc{\pi_0(\C)}
$$
determined by $\eta_0=c$ has a partial left adjoint $\pi_0$ defined at $P$, i.e., we have a bijection
$$
\cP(\pi_0(\C))(\pi_0(P), Q)\simeq \cP^\C(P,\eta_\C^*Q),
$$
natural in $Q\in\cP(\pi_0(\C))$, and $P\in\cP^\C$, whenever $\pi_0(P)$ exists.
\end{proposition}

\begin{proof}
The assumptions yield the diagram
$$
 \begin{tikzpicture}
 [cross line/.style={preaction={draw=white, -,line width=4pt}}, proj/.style={dotted,-}]
 \def\triple#1#2{
 ([yshift=.4em]#1.east) edge 
 ([yshift=.4em]#2.west) 
(#1) edge 
(#2)
([yshift=-.4em]#1.east) edge 
([yshift=-.4em]#2.west) }
\def\double#1#2{
([yshift=.4em]#1.east) edge 
([yshift=.4em]#2.west) 
(#2)  edge 
(#1) 
([yshift=-.4em]#1.east) edge 
([yshift=-.4em]#2.west) }
\def\tripleov#1#2{
 ([yshift=.4em]#1.east) edge[cross line] 
 ([yshift=.4em]#2.west) 
(#1) edge[cross line] 
(#2)
([yshift=-.4em]#1.east) edge[cross line] 
([yshift=-.4em]#2.west) }
\def\doubleov#1#2{
([yshift=.4em]#1.east) edge[cross line] 
([yshift=.4em]#2.west) 
(#2)  edge[cross line] 
(#1) 
([yshift=-.4em]#1.east) edge[cross line] 
([yshift=-.4em]#2.west) }
\matrix(m)[matrix of math nodes, row sep=1em, column sep=1em,  
text height=1.2ex, text depth=0.25ex]
{
|(P2)|{P_2} 		&		&[0em] |(P1)|{P_1}  	&[0em]		& |(P0)|{P_0}	&			\\[1.2em]
			& |(p2)|{\pi_0(P)}  	& 		&|(p1)|{\pi_0(P)} 	&		& |(p0)|{\pi_0(P)}\\[.6em]    
|(Q2)|{\eta_2^*Q} 		&		&[0em] |(Q1)|{\eta_1^*Q}  	&[0em]		& |(Q0)|{\eta_0^*Q}	&			\\[1.2em]
			& |(q2)|{Q}  	& 		&|(q1)|{Q} 	&		& |(q0)|{Q}\\[.6em]    			      
|(C2)|{C_2}	&		&|(C1)|{C_1}	&		& |(C0)|{C_0} &			\\[1.2em]
			&|(c2)|{\pi_0(\C)}	&		&|(c1)|{\pi_0(\C)} &			& |(c0)|{\pi_0(\C)}   \\};
\path[->,font=\scriptsize,>=to, thin,inner sep=1pt]
\triple{P2}{P1}
\double{P1}{P0}

\triple{Q2}{Q1}
\double{Q1}{Q0}

\triple{C2}{C1}
\double{C1}{C0}

(Q2) edge[proj] (C2)
(Q1) edge[proj] (C1)
(Q0) edge[proj] (C0)

(P2) edge[dashed] node[pos=0.25,left]{$f_2$}(Q2)
(P1) edge[dashed] node[pos=0.25,left]{$f_1$}(Q1)
(P0) edge[dashed] node[pos=0.25,left]{$f_0$}(Q0)
\tripleov{p2}{p1}
\doubleov{p1}{p0}

\tripleov{q2}{q1}
\doubleov{q1}{q0}

\triple{c2}{c1}
\double{c1}{c0}

(Q2) edge node[pos=0.5,below left]{$q_2$} (q2)
(Q1) edge node[pos=0.5,below left]{$q_1$}(q1)
(Q0) edge node[pos=0.5,below left]{$q_0$}(q0)

(P2) edge node[pos=0.6,above right]{$\eta^P_2$} (p2)
(P1) edge node[pos=0.6,above right]{$\eta^P_1$} (p1)
(P0) edge node[pos=0.6,above right]{$\eta^P_0$} (p0)

(C2) edge node[pos=0.5,below left]{$\eta_2$}(c2)
(C1) edge node[pos=0.5,below left]{$\eta_1$}(c1)
(C0) edge node[pos=0.5,below left]{$\eta_0$}(c0)

(q2) edge[cross line, proj] (c2)
(q1) edge[cross line, proj] (c1)
(q0) edge[cross line, proj] (c0)

(p2) edge[cross line, dashed] node[pos=0.25,left]{$f$}(q2)
(p1) edge[cross line, dashed] node[pos=0.25,left]{$f$}(q1)
(p0) edge[cross line, dashed] node[pos=0.25,left]{$f$}(q0)
;
\end{tikzpicture}
$$ 
without the dashed arrows, where all horizontal arrows in $\cP$ except possibly $\eta^P_0$, $\eta^P_1$ and $\eta^P_2$ are cartesian. 

By construction, $q_0$ coequalises the source and the target morphisms of $\eta^*Q$, whence, if we take a morphism $(f_0,f_1,f_2)\in \cP^\C(P,\eta^*Q)$, $q_0f_0$ coequalises $d_0^P$, $d_1^P$, so there exists a unique morphism $f:\pi_0(P)\to Q$ that makes the rightmost vertical square, and ultimately the whole diagram, commutative. 

Conversely, if we start with a morphism $f:\pi_0(P)\to Q$, using that $q_0$ is cartesian, we obtain a unique morphism $f_0:P_0\to \eta_0^*Q$ that makes the rightmost vertical square commutative. Since the composite $\eta_1^*Q\to\eta_0^*Q\to Q$ is cartesian, we obtain a unique morphism $f_1:P_1\to\eta_1^*Q$ which makes the square involving $P_1,\eta_1^*Q, Q, \pi_0(P)$ commutative, and, given that $q_0$ is cartesian, $f_1$ and $f_0$ commute with source and target morphisms. Using that $q_1$ is cartesian, $f_1$ and $f_0$ commute with the identity sections. Continuing in the same fashion, we obtain a unique $f_2:P_2\to\eta_2^*Q$ that makes the whole diagram commutative. 
\end{proof}

\subsection{Precategorical descent}

\begin{definition}\label{def:precat-desc}
Let 
$$
f:\C\to\bbD
$$
be a morphism in $\Precat(\cC)$. We say that
\begin{enumerate}
\item $f$ is a \emph{descent morphism} for $\cP$, if the functor $f^*:\cP^\bbD\to \cP^\C$ is fully faithful, and
\item $f$ is of \emph{effective descent} for $\cP$, if the functor $f^*:\cP^\bbD\to \cP^\C$ is an equivalence of categories. 
\end{enumerate}

\end{definition}

\begin{lemma}\label{univ-ce-desc}
Suppose that the reflexive coequaliser
$$
\begin{tikzcd}[cramped, column sep=2em, ampersand replacement=\&]
C_1 \ar[yshift=2pt]{r}{d_0} \ar[yshift=-2pt]{r}[swap]{d_1} \& C_0 \ar{r} \& \pi_0(\C)
\end{tikzcd}
$$  
exists in $\cC$. 

The morphism of precategories 
$$
\eta:\C\to \Disc{\pi_0(\C)}
$$
is a descent morphism for $\cP$ if and only if the above reflexive coequaliser is universal for $\cP$, i.e., for every $Q\in \cP(\pi_0(\C))$, the diagram
$$
\begin{tikzcd}[cramped, column sep=2em, ampersand replacement=\&]
\eta_1^*Q \ar[yshift=2pt]{r}{} \ar[yshift=-2pt]{r}[swap]{} \& \eta_0^*Q \ar{r} \& Q
\end{tikzcd}
$$  
remains a coequaliser. 
\end{lemma}
\begin{proof}
The above diagram is a coequaliser if and only if $\pi_0(\eta^*Q)\simeq Q$, making $\eta^*$ fully faithful. 
\end{proof}

\subsection{Descent data}

\begin{definition}\label{def-dd}
Given a family $\cU=\{U_i\stackrel{u_i}{\to}U\,:\,i\in I\}$ of morphisms in $\cC$, the \emph{category of descent data}
$$
\DD(\cU)
$$
consists of tuples
$$
\left( (P_i)_{i\in I}, (\varphi_{ij})_{(i,j)\in I^2}\right), 
$$
where 
\begin{itemize}
\item for every $i\in I$, $P_i\in \cP(U_i)$;
\item for every $i,j\in I$, $\varphi_{ij}:\pi^*_{ij,i}P_i\to\pi^*_{ij,j}P_j$ is  an isomorphism in $\cP(U_{ij})$, 
where $U_{ij}=U_i\times_UU_j$ and the corresponding projections are given by the pullback diagram
$$
 \begin{tikzpicture}
[cross line/.style={preaction={draw=white, -,
line width=4pt}}]
\matrix(m)[matrix of math nodes, row sep=.9em, column sep=.5em, text height=1.5ex, text depth=0.25ex]
{			& |(ij)| {U_{ij}}	&				\\   [.02em]
	|(i)|{U_i} 	&			& |(j)|{U_j} 		\\  [.02em]
			& |(0)|{U} 		&				\\};
\path[->,font=\scriptsize,>=to, thin]
(ij) edge node[left,pos=0.2]{$\pi_{ij,i}$} (i) edge node[right,pos=0.2]{$\pi_{ij,j}$} (j) 
(i) edge node[left,pos=0.6]{$u_i$} (0) 
(j) edge node[right,pos=0.6]{$u_j$} (0)
;
\end{tikzpicture}
$$ 
\end{itemize}
satisfying the \emph{cocycle condition}: for every $i,j,k\in I$, considering the triple pullback $U_{ijk}=U_i\times_UU_j\times_UU_k$ and the corresponding projections
$$
 \begin{tikzpicture}
[cross line/.style={preaction={draw=white, -,
line width=4pt}}]
\matrix(m)[matrix of math nodes, row sep=1.2em, column sep=.8em, text height=1.5ex, text depth=0.25ex]
{				& |(ijk)| {U_{ijk}}	&				\\   [.02em]
	|(ij)|{U_{ij}} 	& |(ik)|{U_{ik}} 	& |(jk)|{U_{jk}} 		\\  [.2em]
	|(i)|{U_{i}} 		& |(j)|{U_{j}} 	& |(k)|{U_{k}} 		\\  [.02em]
				& |(0)|{U} 		&				\\};
\path[->,font=\scriptsize,>=to, thin]
(ijk) edge node[left,pos=0.2]{$$} (ij) edge node[right,pos=0.2]{$$} (ik) edge node[right,pos=0.2]{$$} (jk) 
(ik) edge node[left,pos=0.2]{$$} (i) edge node[right,pos=0.2]{$$} (k)
(ij) edge[cross line] node[left,pos=0.2]{$$} (i) edge[cross line] node[right,pos=0.2]{$$} (j) 
(jk) edge[cross line] node[left,pos=0.2]{$$} (j) edge[cross line] node[right,pos=0.2]{$$} (k)  
(i) edge node[left,pos=0.6]{$u_i$} (0) 
(j) edge node[right,pos=0.4]{$u_j$} (0)
(k) edge node[right,pos=0.6]{$u_k$} (0)
;
\end{tikzpicture}
$$ 
the diagram 
$$
 \begin{tikzpicture}
[cross line/.style={preaction={draw=white, -,
line width=4pt}}]
\matrix(m)[matrix of math nodes, row sep=2em, column sep=1.2em, text height=1.5ex, text depth=0.25ex]
{	|(i)|{\pi^*_{ijk,i}P_i} 	&			& |(k)|{\pi^*_{ijk,k}P_k} 		\\  [.02em]
			& |(j)|{\pi^*_{ijk,j}P_j} 	&				\\};
\path[->,font=\scriptsize,>=to, thin]
(i) edge node[left,pos=0.6]{$\pi^*_{ijk,ij}\varphi_{ij}$} (j) edge node[above,pos=0.5]{$\pi^*_{ijk,ik}\varphi_{ik}$} (k) 
(j) edge node[right,pos=0.4]{$\pi^*_{ijk,jk}\varphi_{jk}$} (k) 
;
\end{tikzpicture}
$$ 
commutes in $\cP(U_{ijk})$.

A $\DD(\cU)$-\emph{morphism} $$\left((P_i)_i, (\varphi_{ij})_{ij}\right)\to \left((P'_i)_i, (\varphi'_{ij})_{ij}\right)$$
consists of a family of morphisms $(P_i\stackrel{\psi_i}{\to}P'_i)_{i\in I}$ in $\cP(U_i)$ such that the diagrams
$$
 \begin{tikzpicture}
[cross line/.style={preaction={draw=white, -,
line width=4pt}}]
\matrix(m)[matrix of math nodes, row sep=1.8em, column sep=1.8em, text height=1.5ex, text depth=0.25ex]
{	|(0i)|{\pi^*_{ij,i}P_i}		& |(1i)| {\pi^*_{ij,i}P'_i}	\\   [.02em]
	|(0j)|{\pi^*_{ij,j}P_j}		& |(1j)| {\pi^*_{ij,i}P'_j}	\\};
\path[->,font=\scriptsize,>=to, thin]
(0i) edge node[above,pos=0.5]{$\pi^*_{ij,i}\psi_i$} (1i) edge node[left,pos=0.5]{$\varphi_{ij}$} (0j) 
(0j) edge node[above,pos=0.5]{$\pi^*_{ij,j}\psi_j$} (1j) 
(1i) edge node[right,pos=0.5]{$\varphi'_{ij}$} (1j)
;
\end{tikzpicture}
$$ 
commute in $\cP(U_{ij})$. 

\end{definition}

\subsection{Pullback of descent data}

\begin{definition}
Let $\cU=\{ U_i\stackrel{u_i}{\to} U \,:\, i\in I\}$ and $\cV=\{ V_j\stackrel{v_j}{\to} V \,:\, j\in J\}$ be two families in $\cC$. 

A \emph{morphism} $\gamma:\cU\to \cV$ is given by a map of indices $\alpha:I\to J$, a morphism $g:U\to V$ in $\cC$ and a family of commutative diagrams 
$$
 \begin{tikzpicture}
[cross line/.style={preaction={draw=white, -,
line width=4pt}}]
\matrix(m)[matrix of math nodes, row sep=1.8em, column sep=1.8em, text height=1.5ex, text depth=0.25ex]
{	|(0i)|{U_i}		& |(1i)| {V_{\alpha(i)}}	\\   [.02em]
	|(0j)|{U}		& |(1j)| {V}	\\};
\path[->,font=\scriptsize,>=to, thin]
(0i) edge node[above,pos=0.5]{$g_i$} (1i) edge node[left,pos=0.5]{$u_i$} (0j) 
(0j) edge node[above,pos=0.5]{$g$} (1j) 
(1i) edge node[right,pos=0.5]{$v_{\alpha(i)}$} (1j)
;
\end{tikzpicture}
$$ 
for $i\in I$.

If $\gamma':(\alpha',g',g'_i)$ is another morphism $\cU\to\cV$ with $g=g'$, we say that morphisms $\gamma$ and $\gamma'$ are \emph{homotopy equivalent} and write
$$
\gamma\sim\gamma'.
$$
\end{definition}

\begin{lemma}[Pullback of descent data, {\cite[8.3.3]{stacks-project}}]\label{pbdesc}
A \emph{morphism} $\gamma:\cU\to \cV$ as above gives a functor
\begin{align*}
\gamma^*: &\  \DD(\cV)\to DD(\cU),  \\
 &  (P_j,\varphi_{jj'})\mapsto (g_i^*P_{\alpha(i)}, (g_i\times g_{i'})^* \varphi_{\alpha(i)\alpha(i')}).
\end{align*}
Moreover,  if $\gamma\sim \gamma'$, their associated pullback functors are canonically isomorphic, 
$$\gamma^*\simeq \gamma'^*.$$
\end{lemma}

\begin{definition}\label{def-effdesc}
With notation from \ref{pbdesc}, we say that 
\begin{enumerate}
\item $\gamma:\cU\to\cV$ is a \emph{descent morphism}, if the functor $\gamma^*:\DD(\cV)\to \DD(\cU)$ is fully faithful, and
\item $\gamma:\cU\to\cV$ is of \emph{effective descent}, if $\gamma^*:\DD(\cV)\to \DD(\cU)$ establishes an equivalence of categories. 
\end{enumerate}

\end{definition}

\begin{remark}
Given a morphism $f:U'\to U$ in $\cC$, the above definition applied to the morphism of families $f_\Box$ given by the diagram
 $$
 \begin{tikzpicture}
[cross line/.style={preaction={draw=white, -,
line width=4pt}}]
\matrix(m)[matrix of math nodes, row sep=1.8em, column sep=1.8em, text height=1.5ex, text depth=0.25ex]
{	 |(1i)| {U} 	&	|(0i)|{U'}	\\   [.02em]
	|(1j)| {U}	 		&  |(0j)|{U} \\};
\path[->,font=\scriptsize,>=to, thin]
(0i) edge node[above,pos=0.5]{$f$} (1i) edge node[right,pos=0.5]{$f$} (0j) 
(0j) edge node[above,pos=0.5]{$\id$} (1j) 
(1i) edge node[left,pos=0.5]{$\id$} (1j)
(1i) edge[draw=none] node [above=-.3em]{$f_\Box$} (0j)
;
\end{tikzpicture}
$$ 
tells us that $f_\Box$ is a morphism of descent (resp., effective descent) if the functor
$$
f^*_\Box :\cP(U)=\DD(\id_U)\to \DD(f)
$$
is fully faithful (resp., an equivalence). 

Hence, $f$ is a morphism of descent/effective descent in the classical sense whenever $f_\Box$ is descent/effective descent in the sense of \ref{def-effdesc}.
\end{remark}

\begin{definition}[Horizontal composition of boxes]\label{def-horbox} 
The horizontal composition of boxes $\varphi$ and $\psi$ in $\cC$
$$
 \begin{tikzpicture}
[cross line/.style={preaction={draw=white, -,
line width=4pt}}]
\matrix(m)[matrix of math nodes, row sep=1.8em, column sep=1.8em, text height=1.5ex, text depth=0.25ex]
{	 |(1i)| {U} 	&	|(2i)|{V}	& 	|(3i)|{W} \\   [.02em]
	 |(1j)| {U'} 	&	|(2j)|{V'}   & 	|(3j)|{W'} \\};
\path[->,font=\scriptsize,>=to, thin]
(2i) edge node[above,pos=0.5]{$f$} (1i) edge node[right,pos=0.4]{$v$} (2j) 
(2j) edge node[above,pos=0.5]{$f'$} (1j) 
(1i) edge node[left,pos=0.4]{$u$} (1j)
(1i) edge[draw=none] node [above=-.2em]{$\varphi$} (2j)
(3i) edge node[above,pos=0.5]{$g$} (2i) edge node[right,pos=0.4]{$w$} (3j) 
(3j) edge node[above,pos=0.5]{$g'$} (2j) 
(2i) edge[draw=none] node [above=-.2em]{$\psi$} (3j)
;
\end{tikzpicture}
$$ 
is the box $\varphi\circ\psi$ given by the diagram
$$
 \begin{tikzpicture}
[cross line/.style={preaction={draw=white, -,
line width=4pt}}]
\matrix(m)[matrix of math nodes, row sep=1.8em, column sep=1.8em, text height=1.5ex, text depth=0.25ex]
{	 |(1i)| {U} 	&	|(2i)|{W}	\\   [.02em]
	 |(1j)| {U'} 	&	|(2j)|{W'}   \\};
\path[->,font=\scriptsize,>=to, thin]
(2i) edge node[above,pos=0.5]{$f\circ g$} (1i) edge node[right,pos=0.4]{$w$} (2j) 
(2j) edge node[above,pos=0.5]{$f'\circ g'$} (1j) 
(1i) edge node[left,pos=0.4]{$u$} (1j)
(1i) edge[draw=none] node [above=-.2em]{$\varphi\circ\phi$} (2j)
;
\end{tikzpicture}
$$ 
\end{definition}

\begin{remark}\label{rem-horbox}
With notation of \ref{def-horbox}, as an immediate consequence of the definition of pullback functors from \ref{pbdesc}, we obtain an isomorphism of functors
$$
(\varphi\circ\psi)^*\simeq \psi^*\circ\varphi^*.
$$

\end{remark}

\begin{definition}[Vertical composition of boxes]\label{def-vertbox}
The vertical composition of boxes $\varphi$ and $\varphi'$ in $\cC$ 
$$
 \begin{tikzpicture}
[cross line/.style={preaction={draw=white, -,
line width=4pt}}]
\matrix(m)[matrix of math nodes, row sep=1.8em, column sep=1.8em, text height=1.5ex, text depth=0.25ex]
{	 |(1i)| {U} 	&	|(2i)|{V}	\\   [.02em]
	 |(1j)| {U'} 	&	|(2j)|{V'}	\\   [.02em]
	|(1k)| {U''}	 &     |(2k)|{V''} \\};
\path[->,font=\scriptsize,>=to, thin]
(2i) edge node[above,pos=0.5]{$f$} (1i) edge node[right,pos=0.5]{$v$} (2j) 
(2j) edge node[above,pos=0.5]{$f'$} (1j) edge node[right,pos=0.4]{$v'$} (2k)
(1i) edge node[left,pos=0.5]{$u$} (1j)
(1j) edge node[left,pos=0.5]{$u'$} (1k)
(2k) edge node[above,pos=0.5]{$f''$} (1k)
(1i) edge[draw=none] node [above=-.2em]{$\varphi$} (2j)
(1j) edge[draw=none] node [above=-.2em]{$\varphi'$} (2k)
;
\end{tikzpicture}
$$ 
is the box $\varphi'*\varphi$ given by the diagram
$$
 \begin{tikzpicture}
[cross line/.style={preaction={draw=white, -,
line width=4pt}}]
\matrix(m)[matrix of math nodes, row sep=1.8em, column sep=1.8em, text height=1.5ex, text depth=0.25ex]
{	 |(1i)| {U} 	&	|(2i)|{V}	\\   [.02em]
	 |(1j)| {U''} 	&	|(2j)|{V''}   \\};
\path[->,font=\scriptsize,>=to, thin]
(2i) edge node[above,pos=0.5]{$f$} (1i) edge node[right,pos=0.4]{$v'\circ v$} (2j) 
(2j) edge node[above,pos=0.5]{$f''$} (1j) 
(1i) edge node[left,pos=0.4]{$u'\circ u$} (1j)
(1i) edge[draw=none] node [above=-.2em]{$\varphi'*\varphi$} (2j)
;
\end{tikzpicture}
$$ 

\begin{lemma}\label{lem-vertbox}
With the notation of \ref{def-vertbox}, if $\varphi$ is a descent morphism, and $$(f\times f)^*:\cP(U\times_{U''}U)\to \cP(V\times_{V''}V)$$ is faithful, then $\varphi'*\varphi$ is a descent morphism.
\end{lemma}
\begin{proof}
The boxes
$$
\begin{tikzpicture}
[cross line/.style={preaction={draw=white, -,
line width=4pt}}]
\matrix(m)[matrix of math nodes, row sep=1.8em, column sep=1.8em, text height=1.5ex, text depth=0.25ex]
{	 |(1i)| {U} 	&	|(0i)|{U}	\\   [.02em]
	|(1j)| {U''}	 		&  |(0j)|{U'} \\};
\path[->,font=\scriptsize,>=to, thin]
(0i) edge node[above,pos=0.5]{$\id$} (1i) edge node[right,pos=0.4]{$u$} (0j) 
(0j) edge node[above,pos=0.5]{$u'$} (1j) 
(1i) edge node[left,pos=0.5]{$u'\circ u$} (1j)
(1i) edge[draw=none] node [above=-.2em]{$\rho_U$} (0j)
;
\end{tikzpicture}
\hspace{3em}
\begin{tikzpicture}
[cross line/.style={preaction={draw=white, -,
line width=4pt}}]
\matrix(m)[matrix of math nodes, row sep=1.8em, column sep=1.8em, text height=1.5ex, text depth=0.25ex]
{	 |(1i)| {V} 	&	|(0i)|{V}	\\   [.02em]
	|(1j)| {V''}	 		&  |(0j)|{V'} \\};
\path[->,font=\scriptsize,>=to, thin]
(0i) edge node[above,pos=0.5]{$\id$} (1i) edge node[right,pos=0.4]{$v$} (0j) 
(0j) edge node[above,pos=0.5]{$v'$} (1j) 
(1i) edge node[left,pos=0.5]{$v'\circ v$} (1j)
(1i) edge[draw=none] node [above=-.2em]{$\rho_V$} (0j)
;
\end{tikzpicture}
$$
satisfy 
$$
(\varphi'*\varphi)\circ\rho_V=\rho_U\circ\varphi,
$$
so we obtain a diagram of categories
$$
 \begin{tikzpicture}
[cross line/.style={preaction={draw=white, -,
line width=4pt}}]
\matrix(m)[matrix of math nodes, row sep=1.8em, column sep=2.2em, text height=1.5ex, text depth=0.25ex]
{	|(0i)|{\DD(u'\circ u)}		& |(1i)| {\DD(v'\circ v)}	\\   [.02em]
	|(0j)|{\DD(u)}		& |(1j)| {\DD(v)}	\\};
\path[->,font=\scriptsize,>=to, thin]
(0i) edge node[above,pos=0.5]{$(\varphi'*\varphi)^*$} (1i) edge node[left,pos=0.5]{$\rho^*_U$} (0j) 
(0j) edge node[above,pos=0.5]{$\varphi^*$} (1j) 
(1i) edge node[right,pos=0.5]{$\rho^*_V$} (1j)
;
\end{tikzpicture}
$$ 
where the vertical arrows are faithful. 

Indeed, $\rho_U^*$ takes objects $(P,\alpha)\in\DD(u'\circ u)$ to $(P,(\id\times\id)^*\alpha)=(P,(U\times_{U'}U\to U\times_{U''}U)^*\alpha)$, and it acts on morphisms $p:(P,\alpha)\to (P',\alpha')$ as identity, hence it is faithful. A similar argument applies to $\rho_V^*$.

By assumption, the bottom arrow $\varphi^*$ is also faithful, and it follows that the top arrow $(\varphi'*\varphi)^*$ is too. 

It suffices to verify that $(\varphi'*\varphi)^*$ is full. Let $(P,\alpha), (P',\alpha')\in \DD(u'\circ u)$, let 
$(Q,\beta)=(\varphi'*\varphi)^*(P,\alpha)$, $(Q',\beta')=(\varphi'*\varphi)^*(P',\alpha')$, and let
$q:(Q,\beta)\to(Q',\beta')$ be a morphism in $\DD(v'\circ v)$, given by a morphism $q:Q\to Q'$ in $\cP(V)$ such that the diagram
$$
 \begin{tikzpicture}
[cross line/.style={preaction={draw=white, -,
line width=4pt}}]
\matrix(m)[matrix of math nodes, row sep=1.8em, column sep=1.8em, text height=1.5ex, text depth=0.25ex]
{	|(0i)|{\pi^*_1Q}		& |(1i)| {\pi^*_1Q'}	\\   [.02em]
	|(0j)|{\pi^*_2Q}		& |(1j)| {\pi^*_2Q'}	\\};
\path[->,font=\scriptsize,>=to, thin]
(0i) edge node[above,pos=0.5]{$\pi^*_1 q$} (1i) edge node[left,pos=0.5]{$\beta$} (0j) 
(0j) edge node[above,pos=0.5]{$\pi^*_2 q$} (1j) 
(1i) edge node[right,pos=0.5]{$\beta'$} (1j)
;
\end{tikzpicture}
$$ 
commutes in $\cP(V\times_{V''}V)$.

Since $\varphi^*$ is fully faithful, there exists a unique morphism $p:\rho_U^*(P,\alpha)\to \rho_U^*(P',\alpha')$ such that 
$\varphi^*(p)=\rho_V^*(q)$, i.e., a morphism $p:P\to P'$ in $\cP(U)$ with $f^*(p)=q$. 

We claim that $p$ is a morphism $(P,\alpha)\to (P',\alpha')$ in $\DD(u'\circ u)$, i.e., that the diagram
$$
 \begin{tikzpicture}
[cross line/.style={preaction={draw=white, -,
line width=4pt}}]
\matrix(m)[matrix of math nodes, row sep=1.8em, column sep=1.8em, text height=1.5ex, text depth=0.25ex]
{	|(0i)|{\pi^*_1P}		& |(1i)| {\pi^*_1P'}	\\   [.02em]
	|(0j)|{\pi^*_2P}		& |(1j)| {\pi^*_2P'}	\\};
\path[->,font=\scriptsize,>=to, thin]
(0i) edge node[above,pos=0.5]{$\pi^*_1 p$} (1i) edge node[left,pos=0.5]{$\alpha$} (0j) 
(0j) edge node[above,pos=0.5]{$\pi^*_2 p$} (1j) 
(1i) edge node[right,pos=0.5]{$\alpha'$} (1j)
;
\end{tikzpicture}
$$ 
commutes in $\cP(V\times_{V''}V)$. This is indeed the case, since pulling the diagram back to $\cP(V\times_{V''}V)$ via the faithful functor $(f\times f)^*$ gives the above diagram for $q$, which is commutative. 

\end{proof}

\end{definition}

\begin{lemma}[A morphism admitting a section is of effective descent]\label{secteff}
Suppose we have a morphism $f:U'\to U$ admitting a section $s:U\to U'$, so that $f\circ s=\id_U$. 
Then $f$ is of effective descent. More explicitly, the functor $$f^*=f^*_\Box:\cP(U)\to\DD(f)$$ has a quasi-inverse $\sigma^*$ associated to the box
$$
 \begin{tikzpicture}
[cross line/.style={preaction={draw=white, -,
line width=4pt}}]
\matrix(m)[matrix of math nodes, row sep=1.8em, column sep=1.8em, text height=1.5ex, text depth=0.25ex]
{	 |(1i)| {U'} 	&	|(0i)|{U}	\\   [.02em]
	|(1j)| {U}	 		&  |(0j)|{U} \\};
\path[->,font=\scriptsize,>=to, thin]
(0i) edge node[above,pos=0.5]{$s$} (1i) edge node[right,pos=0.4]{$id$} (0j) 
(0j) edge node[above,pos=0.5]{$\id$} (1j) 
(1i) edge node[left,pos=0.5]{$f$} (1j)
(1i) edge[draw=none] node [above=-.2em]{$\sigma$} (0j)
;
\end{tikzpicture}
$$
\end{lemma}

\begin{proof}
The composite $\sigma\circ f_\Box$ represents the outer box of the diagram
$$
 \begin{tikzpicture}
[cross line/.style={preaction={draw=white, -,
line width=4pt}}]
\matrix(m)[matrix of math nodes, row sep=1.8em, column sep=1.8em, text height=1.5ex, text depth=0.25ex]
{	 |(1i)| {U'} 	&	|(0i)|{U'}	\\   [.02em]
	|(1j)| {U}	 		&  |(0j)|{U} \\};
\path[->,font=\scriptsize,>=to, thin]
(0i.190) edge node[below,pos=0.5]{$\id$} (1i.-10) 
(0i.170) edge node[above,pos=0.5]{$s\circ f$} (1i.10) 
(0i) edge node[right,pos=0.4]{$f$} (0j) 
(0j) edge node[above,pos=0.5]{$\id$} (1j) 
(1i) edge node[left,pos=0.5]{$f$} (1j)
;
\end{tikzpicture}
$$
and let us denote the inner box by $\id_f$. These boxes are homotopy equivalent in the sense of the `moreover' clause of \ref{pbdesc}, whence 
$$
f^*_\Box \sigma^*\simeq (\sigma\circ f_\Box)^*\simeq \id_f^*=\id: \DD(f)\to \DD(f).
$$
Conversely, using the fact that $s$ is a section of $f$,  $$\sigma^*f^*_\Box\simeq (f_\Box\circ\sigma)^*=\id:\DD(\id_U)\to\DD(\id_U).$$

\end{proof}

\subsection{Classical and precategorical descent}

Let $u:U'\to U$ be a morphism in $\cC$, and consider the groupoid 
$$
\bbG_u
$$
associated to the kernel pair of $u$, as in \ref{precat-ker-pair}

Comparing \ref{gen-dd} and \ref{def-dd}, we see that we have an isomorphism
$$
\DD_{\cP}(u)=\cP^{\bbG_u}.
$$
Moreover, a box
$$
 \begin{tikzpicture}
[cross line/.style={preaction={draw=white, -,
line width=4pt}}]
\matrix(m)[matrix of math nodes, row sep=1.8em, column sep=1.8em, text height=1.5ex, text depth=0.25ex]
{	|(0i)|{U'}		& |(1i)| {V'}	\\   [.02em]
	|(0j)|{U}		& |(1j)| {V}	\\};
\path[->,font=\scriptsize,>=to, thin]
(0i) edge node[above,pos=0.5]{$f'$} (1i) edge node[left,pos=0.5]{$u$} (0j) 
(0j) edge node[above,pos=0.5]{$f$} (1j) 
(1i) edge node[right,pos=0.5]{$v$} (1j)
(1i) edge[draw=none] node [above=-.2em]{$\varphi$} (0j)
;
\end{tikzpicture}
$$ 
induces a morphism 
$$F=(f'\times f'\times f', f'\times f', f'):\bbG_v\to \bbG_u,$$
and the functors
$$
\varphi^*:\DD(u)\to\DD(v), \ \ \ \text{ and } \ \ \ \ F^*:\cP^{\bbG_u}\to \cP^{\bbG_v}
$$
are isomorphic. 

Hence, $\varphi$ is a morphism of descent/effective descent in the classical sense of \ref{def-effdesc} if and only if $F$ is a morphism of descent/effective precategorical descent in the sense of \ref{def:precat-desc}.

In particular, a morphism $u:U'\to U$ is a morphism of descent/effective descent in the classical sense if and only if $u_\Box$ is of descent/effective descent, if and only if the precategory morphism $\bbG_u\to \Disc{U}=\bbG_{\id_U}$ is a morphism of descent/effective precategorical descent. 

\subsection{Descent of precategory actions}

\begin{proposition}\label{prop:desc-precat}
Let $\cP:\cC^\op\to\Cat$ be a pseudofunctor, and consider the pseudofunctor
$$
\tilde{\cP}:\Precat(\cC)^\op\to\Cat, \ \ \ \bbX\mapsto \cP^{\bbX}.
$$

Let $f:\C\to\bbD\in\Precat(\cC)$ be a morphism of precategories in $\cC$ such that
\begin{enumerate}
\item $f_0$ is of effective descent for $\cP$; 
\item $f_1$ is descent morphism for $\cP$;
\item $f_2^*$ is faithful. 
\end{enumerate}
Then $f$ is of effective descent for $\tilde{\cP}$. 
\end{proposition}

\begin{proof}
We must show that the canonical morphism $$\cP^\bbD=\tilde{\cP}(\bbD)\to \tilde{\cP}^{\bbG_f}=\DD_{\tilde{\cP}}(f)$$ is an equivalence of categories, where 
$$
\bbG_f\in\Precat(\Precat(\cC))
$$ 
is given by
$$
\begin{tikzpicture} 
\matrix(m)[matrix of math nodes, row sep=0em, column sep=2em, text height=1.5ex, text depth=0.25ex]
 {
|(2)|{\bbG_2=\C\times_{\bbD} \C\times_{\bbD}\C}  & [1em] |(1)|{\bbG_1=\C\times_{\bbD}\C}		&[1em] |(0)|{\bbG_0=\C} \\
 }; 
\path[->,font=\scriptsize,>=to, thin]
([yshift=.4em]2.east) edge 
([yshift=.4em]1.west) 
(2) edge 
(1)
([yshift=-.4em]2.east) edge 
([yshift=-.4em]1.west) 
([yshift=.4em]1.east) edge 
 ([yshift=.4em]0.west) 
(0)  edge 
(1) 
([yshift=-.4em]1.east) edge 
([yshift=-.4em]0.west) 
;
\end{tikzpicture}
$$
with $\bbG_0, \bbG_1, \bbG_2\in \Precat(\cA)_{\ov\bbD}$. Expanding the components of these precategories as columns, we obtain a diagram
$$
 \begin{tikzpicture}
 [cross line/.style={preaction={draw=white, -,line width=4pt}}, proj/.style={dotted,-}]
 \def\triple#1#2{
 ([yshift=.4em]#1.east) edge 
 ([yshift=.4em]#2.west) 
(#1) edge 
(#2)
([yshift=-.4em]#1.east) edge 
([yshift=-.4em]#2.west) }
\def\double#1#2{
([yshift=.4em]#1.east) edge 
([yshift=.4em]#2.west) 
(#2)  edge 
(#1) 
([yshift=-.4em]#1.east) edge 
([yshift=-.4em]#2.west) }
\def\vtriple#1#2{
 ([xshift=.4em]#1.south) edge 
 ([xshift=.4em]#2.north) 
(#1) edge 
(#2)
([xshift=-.4em]#1.south) edge 
([xshift=-.4em]#2.north) }
\def\vdouble#1#2{
([xshift=.4em]#1.south) edge
([xshift=.4em]#2.north) 
(#2)  edge 
(#1) 
([xshift=-.4em]#1.south) edge[cross line] 
([xshift=-.4em]#2.north) }
\matrix(m)[matrix of math nodes, row sep=2em, column sep=2em,  
text height=1.2ex, text depth=0.25ex]
{
|(P22)|{C_2\times_{D_2}C_2\times_{D_2}C_2} 	&[0em] |(P12)|{C_2\times_{D_2}C_2}  	&[0em]	 |(P02)|{C_2}		\\[0em]
|(P21)|{C_1\times_{D_1}C_1\times_{D_1}C_1} 	&[0em] |(P11)|{C_1\times_{D_1}C_1}  	&[0em]	 |(P01)|{C_1}		\\[0em]
|(P20)|{C_0\times_{D_0}C_0\times_{D_0}C_0} 	&[0em] |(P10)|{C_0\times_{D_0}C_0}  	&[0em]	 |(P00)|{C_0}		\\};
\path[->,font=\scriptsize,>=to, thin,inner sep=1pt]
\triple{P22}{P12}
\double{P12}{P02}

\triple{P21}{P11}
\double{P11}{P01}

\triple{P20}{P10}
\double{P10}{P00}

\vtriple{P22}{P21}
\vdouble{P21}{P20}

\vtriple{P12}{P11}
\vdouble{P11}{P10}

\vtriple{P02}{P01}
\vdouble{P01}{P00}

;
\end{tikzpicture}
$$ 
where the rows are the groupoids $\bbG_{f_2}$, $\bbG_{f_1}$, $\bbG_{f_0}$ associated to kernel pairs of morphisms $f_2$, $f_1$, $f_0$ that constitute $f$.

An action $P\in \tilde{P}^{\bbG_f}$ consists of a diagram 
$$
\begin{tikzpicture} 
\matrix(m)[matrix of math nodes, row sep=0em, column sep=2em, text height=1.5ex, text depth=0.25ex]
 {
|(2)|{P_2}  & [1em] |(1)|{P_1}		&[1em] |(0)|{P_0} \\
 }; 
\path[->,font=\scriptsize,>=to, thin]
([yshift=.4em]2.east) edge 
([yshift=.4em]1.west) 
(2) edge 
(1)
([yshift=-.4em]2.east) edge 
([yshift=-.4em]1.west) 
([yshift=.4em]1.east) edge 
 ([yshift=.4em]0.west) 
(0)  edge 
(1) 
([yshift=-.4em]1.east) edge 
([yshift=-.4em]0.west) 
;
\end{tikzpicture}
$$
consisting of $P_2\in \tilde{\cP}(\bbG_2)=\cP^{\bbG_2}$,  $P_1\in \tilde{\cP}(\bbG_1)=\cP^{\bbG_1}$, $P_0\in \tilde{\cP}(\bbG_0)=\cP^{\bbG_0}$ and cartesian arrows in the fibration associated to $\tilde{\cP}$. Expanding the components of $P_2$, $P_1$ and $P_0$ as columns, we obtain a diagram
$$
 \begin{tikzpicture}
 [cross line/.style={preaction={draw=white, -,line width=4pt}}, proj/.style={dotted,-}]
 \def\triple#1#2{
 ([yshift=.4em]#1.east) edge 
 ([yshift=.4em]#2.west) 
(#1) edge 
(#2)
([yshift=-.4em]#1.east) edge 
([yshift=-.4em]#2.west) }
\def\double#1#2{
([yshift=.4em]#1.east) edge 
([yshift=.4em]#2.west) 
(#2)  edge 
(#1) 
([yshift=-.4em]#1.east) edge 
([yshift=-.4em]#2.west) }
\def\vtriple#1#2{
 ([xshift=.4em]#1.south) edge 
 ([xshift=.4em]#2.north) 
(#1) edge 
(#2)
([xshift=-.4em]#1.south) edge 
([xshift=-.4em]#2.north) }
\def\vdouble#1#2{
([xshift=.4em]#1.south) edge
([xshift=.4em]#2.north) 
(#2)  edge 
(#1) 
([xshift=-.4em]#1.south) edge[cross line] 
([xshift=-.4em]#2.north) }
\matrix(m)[matrix of math nodes, row sep=2em, column sep=2em,  
text height=1.2ex, text depth=0.25ex]
{
|(P22)|{P_{2,2}} 	&[1em] |(P12)|{P_{1,2}}  	&[1em]	 |(P02)|{P_{0,2}}		\\[0em]
|(P21)|{P_{2,1}} 	&[1em] |(P11)|{P_{1,1}}  	&[1em]	 |(P01)|{P_{0,1}}		\\[0em]
|(P20)|{P_{2,0}} 	&[1em] |(P10)|{P_{1,0}}  	&[1em]	 |(P00)|{P_{0,0}}		\\};
\path[->,font=\scriptsize,>=to, thin,inner sep=1pt]
\triple{P22}{P12}
\double{P12}{P02}

\triple{P21}{P11}
\double{P11}{P01}

\triple{P20}{P10}
\double{P10}{P00}

\vtriple{P22}{P21}
\vdouble{P21}{P20}

\vtriple{P12}{P11}
\vdouble{P11}{P10}

\vtriple{P02}{P01}
\vdouble{P01}{P00}

;
\end{tikzpicture}
$$ 
in the fibred category associated to $\cP$, with all morphisms cartesian. 

Hence, the rows yield actions 
$\bar{P}_2\in \cP^{\bbG_{f_2}}\simeq\DD_{\cP}(f_2)$, $\bar{P}_1\in \cP^{\bbG_{f_1}}\simeq\DD_{\cP}(f_1)$, $\bar{P}_0\in \cP^{\bbG_{f_0}}\simeq\DD_{cP}(f_0)$. Considering descent data as a pseudofunctor on the arrow category
$$
\DD_{\cP}:\Ar(\cC)^\op\to\Cat,
$$
the diagram
$$
\begin{tikzpicture} 
\matrix(m)[matrix of math nodes, row sep=0em, column sep=2em, text height=1.5ex, text depth=0.25ex]
 {
|(2)|{\bar{P}_2}  & [1em] |(1)|{\bar{P}_1}		&[1em] |(0)|{\bar{P}_0} \\
 }; 
\path[->,font=\scriptsize,>=to, thin]
([yshift=.4em]2.east) edge 
([yshift=.4em]1.west) 
(2) edge 
(1)
([yshift=-.4em]2.east) edge 
([yshift=-.4em]1.west) 
([yshift=.4em]1.east) edge 
 ([yshift=.4em]0.west) 
(0)  edge 
(1) 
([yshift=-.4em]1.east) edge 
([yshift=-.4em]0.west) 
;
\end{tikzpicture}
$$
determines an action 
$$
\bar{P}\in\DD_{\cP}^f
$$
where $f=(f_2,f_1,f_0)\in\Precat(\Ar(\cC))$. 

Hence, we have shown that
$$
\DD_{\tilde{\cP}}(f)\simeq \DD_{\cP}^f.
$$
Since $f_0$ is effective descent, there is an object $Q_0\in \cP(D_0)$ such that, writing $\bar{Q}_0=(Q_0,\id)$ for the trivial descent datum, we have $\bar{P}_0\simeq f_{0\Box}^*\bar{Q}_0$. The action isomorphism $d_{0,f}^*\bar{P}_0\stackrel{\alpha}{\to} d_{1,f}^*\bar{P}_0$ yields an isomorphism
$$
f_{1\Box}^*\bar{d}_0^*\bar{Q}_0\simeq d_{0,f}^*f_{0\Box}^*\bar{Q}_0\simeq d_{0,f}^*\bar{P}_0\stackrel{\alpha}{\to} d_{1,f}^*\bar{P}_0\simeq d_{1,f}^*f_{0\Box}^*\bar{Q}_0\simeq f_{1\Box}^*\bar{d}_1^*\bar{Q}_0,
$$
where we wrote $\bar{d}_0$ and $\bar{d}_1$ for the obvious boxes/morphisms $\id_{D_1}\to \id_{D_0}$ in $\Ar(\cC)$.  
Given that $f_1$ is descent, we obtain a unique action morphism 
$$
\bar{d}_0^*\bar{Q}_0\stackrel{\bar{\beta}}{\to}\bar{d}_1^*\bar{Q}_0
$$
such that $f_{1\Box}^*\bar{\beta}\simeq \alpha$. Note that $\bar{\beta}$ is uniquely determined by an isomorphism
$$
d_0^*Q_0\stackrel{\beta}{\to}d_1^*Q_0,
$$
and it remains to verify that $\beta$ satisfies the cocycle condition
$$
r_1^*\beta\circ r_0^*\beta\simeq m^*\beta
$$
in $\cP(D_2)$, or, equivalently, that $\bar{r}_1^*\bar{\beta}\circ\bar{r}_0^*\bar{\beta}$ and $\bar{m}^*\bar{\beta}$ agree up to coherence. Applying the functor $f_{2\Box}^*$ to both yields
$$
f_{2\Box}^*\bar{r}_1^*\bar{\beta}\circ f_{2\Box}^*\bar{r}_0^*\bar{\beta}\simeq r_{1,f}^*f_{1\Box}^*\bar{\beta}\circ\ r_{0,f}^*f_{1\Box}^*\bar{\beta}\simeq r_{1,f}^*\alpha\circ \circ\ r_{0,f}^*\alpha
$$
and
$$
f_{2\Box}^*\bar{m}^*\bar{\beta}\simeq m_{f}^*f_{1\Box}^*\bar{\beta}\simeq m_f^*\alpha,
$$
which agree up to coherence by the cocycle condition for $\alpha$. By faithfulness of $f_{2\Box}^*$, we obtain the cocycle condition for $\beta$, and we have constructed a unique action $(Q_0,\beta)\in\cP^{\bbD}$ (up to isomorphism) that lifts to $P$, as desired. 
\end{proof}

\subsection{Descent for quasi-projective morphisms}

\begin{proposition}\label{efdesc-qp}
A scheme morphism $f:X\to Y$ whose codomain $Y$ is the spectrum of a field $k$ is of effective descent for quasi-projective morphisms. 
\end{proposition}

\begin{proof}
Let $k'$ be a finite extension of $k$ with $X(k')\neq \emptyset$. Writing $Y'=\spec(k')$, we have a pullback diagram
$$
 \begin{tikzpicture}
[cross line/.style={preaction={draw=white, -,
line width=4pt}}]
\matrix(m)[matrix of math nodes, row sep=1.8em, column sep=1.8em, text height=1.5ex, text depth=0.25ex]
{	 |(1i)| {X} 	&	|(0i)|{X'}	\\   [.02em]
	|(1j)| {Y}	 		&  |(0j)|{Y'} \\};
\path[->,font=\scriptsize,>=to, thin]
(0i) edge node[above,pos=0.5]{$g'$} (1i) edge node[right,pos=0.4]{$f'$} (0j) 
(0j) edge node[above,pos=0.5]{$g$} (1j) 
(1i) edge node[left,pos=0.5]{$f$} (1j)
;
\end{tikzpicture}
$$
where $f'$ admits a section $s$ afforded by a $k'$-point of $X'$. The morphism $g$ is finite locally free surjective, hence of effective descent for quasi-projective morphisms by \cite[VIII, 7.7]{sga1}, and so is $g'$ as a base-change of $g$. 
Consider the diagram

 $$
 \begin{tikzpicture}
[cross line/.style={preaction={draw=white, -,
line width=4pt}}]
\matrix(m)[matrix of math nodes, row sep=2em, column sep=2em, text height=1.5ex, text depth=0.25ex]
{	 |(1i)| {X} 	&	|(2i)|{X'}	& 	|(3i)|{Y'}	& |(4i)|{X'} \\   [.02em]
	 |(1j)| {Y} 	&	|(2j)|{Y}   & 	|(3j)|{Y} 	& |(4j)|{Y} \\};
\path[->,font=\scriptsize,>=to, thin]
(2i) edge node[above,pos=0.5]{$g'$} (1i) edge node[right=-1.8em,pos=0.4]{$fg'{=}gf'$} (2j) 
(2j) edge node[above,pos=0.5]{$\id$} (1j) 
(1i) edge node[left,pos=0.4]{$f$} (1j)
(1i) edge[draw=none] node [above=-.2em]{$\alpha$} (2j)
(3i) edge node[above,pos=0.5]{$s$} (2i) edge node[right,pos=0.4]{$g$} (3j) 
(3j) edge node[above,pos=0.5]{$\id$} (2j) 
(2i) edge[draw=none] node [above=-.2em]{$\beta$} (3j)
(4i) edge node[above,pos=0.5]{$f'$} (3i) edge node[right,pos=0.4]{$gf'$} (4j) 
(4j) edge node[above,pos=0.5]{$\id$} (3j) 
(3i) edge[draw=none] node [above=-.2em]{$\gamma$} (4j)
;
\end{tikzpicture}
$$ 
defining boxes $\alpha$, $\beta$ and $\gamma$ in $\cC$. 

Writing $\sigma$ for the box
$$
 \begin{tikzpicture}
[cross line/.style={preaction={draw=white, -,
line width=4pt}}]
\matrix(m)[matrix of math nodes, row sep=1.8em, column sep=1.8em, text height=1.5ex, text depth=0.25ex]
{	 |(1i)| {X} 	&	|(0i)|{Y'}	\\   [.02em]
	|(1j)| {Y'}	 		&  |(0j)|{Y'} \\};
\path[->,font=\scriptsize,>=to, thin]
(0i) edge node[above,pos=0.5]{$s$} (1i) edge node[right,pos=0.4]{$id$} (0j) 
(0j) edge node[above,pos=0.5]{$\id$} (1j) 
(1i) edge node[left,pos=0.5]{$f'$} (1j)
(1i) edge[draw=none] node [above=-.2em]{$\sigma$} (0j)
;
\end{tikzpicture}
$$
we directly verify that $\beta=\id_g*\sigma$ and $\gamma=\id_g*f'_\Box$. 

As in the proof of \ref{secteff}, we have that $\sigma\circ f'_\Box\sim \id_{f'}$, so 
$$\beta\circ\gamma=(\id_g*\sigma)\circ(\id_g*f'_\Box)=\id_g*(\sigma\circ f'_\Box)\sim  \id_g*\id_{f'}=\id_{g\circ f'},$$ whence
$$\gamma^*\beta^*\simeq\id.$$ Conversely, $\gamma\circ\beta$ yields identity on the nose, so we conclude that $\gamma^*$ is an equivalence of categories. 

Using the fact that $g'$ is of effective descent and that $g'\times g'$ is finite faithfully flat, Lemma~\ref{lem-vertbox} gives that the functor $\alpha^*$ associated to $\alpha=\id_f*g'_\Box$ is fully faithful. 

We directly verify that $g_\Box\circ\gamma=f_\Box\circ\alpha$, whence we obtain a diagram of categories

$$
 \begin{tikzpicture}
[cross line/.style={preaction={draw=white, -,
line width=4pt}}]
\matrix(m)[matrix of math nodes, row sep=.9em, column sep=.5em, text height=1.5ex, text depth=0.25ex]
{			& |(ij)| {\DD(f)}	&				\\   [.02em]
	|(i)|{\cP(Y)} 	&			& |(j)|{\DD(fg')} 		\\  [.02em]
			& |(0)|{\DD(g)} 		&				\\};
\path[->,font=\scriptsize,>=to, thin]
(i) edge node[above left,pos=0.5]{$f^*_\Box$} (ij) edge node[below left,pos=0.5]{$g^*_\Box$} (0) 
(0) edge node[below right,pos=0.5]{$\gamma^*$} (j) 
(ij) edge node[above right,pos=0.5]{$\alpha^*$} (j)
;
\end{tikzpicture}
$$ 
where we know that $\alpha^* f^*_\Box=\gamma^* g^*_\Box$ is an equivalence of categories, and $\alpha^*$ is fully faithful, so we deduce that $f^*_\Box$ is also an equivalence. 
\end{proof}

\section{Categorical Galois theory}\label{s:cat-gal}

\subsection{Classical Janelidze's categorical Galois theory}\label{ss:janelidze-gal}

Consider an adjoint pair of functors
\begin{center}
 \begin{tikzpicture} 
 [cross line/.style={preaction={draw=white, -,
line width=3pt}}]
\matrix(m)[matrix of math nodes, minimum size=1.7em,
inner sep=0pt, 
row sep=3.3em, column sep=1em, text height=1.5ex, text depth=0.25ex]
 { 
  |(dc)|{\cA}	\\
 |(c)|{\cX} 	      \\ };
\path[->,font=\scriptsize,>=to, thin]
%
(dc) edge [bend right=30] node (ss) [left]{$S$} (c)
(c) edge [bend right=30] node (ps) [right]{$C$} (dc)
(ss) edge[draw=none] node{$\dashv$} (ps)
;
\end{tikzpicture}
\end{center}
with unit $\eta:\id\to CS$ and counit $\epsilon:SC\to \id$.
If $\cA$ admits pullbacks, for any $X\in\cA$ we obtain an adjunction
\begin{center}
 \begin{tikzpicture} 
 [cross line/.style={preaction={draw=white, -,
line width=3pt}}]
\matrix(m)[matrix of math nodes, minimum size=1.7em,
inner sep=0pt, 
row sep=3.3em, column sep=1em, text height=1.5ex, text depth=0.25ex]
 { 
  |(dc)|{\cA_{\ov X}}	\\
 |(c)|{\cX_{\ov S(X)}} 	      \\ };
\path[->,font=\scriptsize,>=to, thin]
%
(dc) edge [bend right=30] node (ss) [left]{$S_X$} (c)
(c) edge [bend right=30] node (ps) [right]{$C_X$} (dc)
(ss) edge[draw=none] node{$\dashv$} (ps)
;
\end{tikzpicture}
\end{center}
where  
$$
S_X(A\xrightarrow{a}X)=S(A)\xrightarrow{S(a)}S(X),
$$
and  $C_X(E\xrightarrow{e}S(X))$ is obtained by forming the pullback
\begin{center}
 \begin{tikzpicture} 
\matrix(m)[matrix of math nodes, row sep=2em, column sep=2em, text height=1.9ex, text depth=0.25ex]
 {
 |(1)|{C_X(e)}		& |(2)|{C(E)} 	\\
 |(l1)|{X}		& |(l2)|{CS(X)} 	\\
 }; 
\path[->,font=\scriptsize,>=to, thin]
(1) edge  (2) edge   (l1)
(2) edge node[right]{$C(e)$} (l2) 
(l1) edge node[above]{$\eta_X$}  (l2);
\end{tikzpicture}
\end{center}
in $\cA$.

A morphism $X\xrightarrow{f}Y$ {in $\cA$} gives rise to the pullback/base change functor
$$
f^*:\cA_{\ov Y}\to\cA_{\ov X},
$$
which admits a left adjoint
$$
f_!:\cA_{\ov X}\to\cA_{\ov Y},  \ \ \  a\mapsto f\circ a.
$$
{Following \cite[Def. 5.1.7]{borceux-janelidze}} an object $A\xrightarrow{a}Y\in\cA_{\ov Y}$ is \emph{split} by $X\xrightarrow{f}Y\in\cA$ when the unit $\eta^X:\id\to C_XS_X$ of adjunction $S_X\dashv C_X$ gives an isomorphism
$$
\eta^X_{f^*a}:f^*a\to C_XS_X(f^*a).
$$
If $C_X$ is fully faithful, $a$ is split by $f$ {(\cite[Cor. 5.1.13]{borceux-janelidze})}, if and only if there exists an object $E\xrightarrow{e}S(X)$ such that 
$$
f^*a\simeq C_X(e).
$$
We write
$$
\Split_Y(f)
$$
for the full subcategory of $\cA_{\ov Y}$ of objects split by $f$.

The morphism $f$ is of \emph{relative Galois descent} if
\begin{enumerate}
\item $f^*$ is monadic;
\item the counit $\epsilon^X:S_XC_X\to \id$ of adjunction $S_X\dashv C_X$ is an isomorphism;
\item for every $E\xrightarrow{e}S(X)$ in $\cX_{\ov S(X)}$, the object $f_!\, C_X(e)\in\cA_{\ov Y}$ is split by $f$.
\end{enumerate} 
 If $X\xrightarrow{f}Y$ is of relative Galois descent, the Galois precategory
$$
\Gal[f]=S(\bbG_f)
$$
is actually an internal groupoid in $\cX$ given by the data

\begin{center}
 \begin{tikzpicture} 
\matrix(m)[matrix of math nodes, row sep=0em, column sep=3em, text height=1.5ex, text depth=0.25ex]
 {
|(0)|{S(X\times_YX)\times_{S(X)}S(X\times_YX)}  &[3em]  |(1)|{S(X\times_YX)}		&[1em] |(2)|{S(X)} \\
 }; 
\path[->,font=\scriptsize,>=to, thin]
(0) edge node[above]{$(S(\pi_1),S(\pi_4))$} (1)
([yshift=1em]1.east) edge node[above=-2pt]{$S(\pi_1)$} ([yshift=1em]2.west) 
(2)  edge node[above=-2pt]{$S(\Delta)$} (1) 
([yshift=-1em]1.east) edge node[above=-2pt]{$S(\pi_2)$} ([yshift=-1em]2.west) 
 (1) edge [loop below] node {$S(\tau)$} (1)
;
\end{tikzpicture}
\end{center}
where $\tau$ is the morphism interchanging the copies of $X$, and $\Delta$ is the diagonal.

Janelidze's \emph{Galois theorem} {(\cite[Thm. 5.1.24]{borceux-janelidze})} gives  an equivalence of categories
$$
\Split_Y(f)\simeq \cX^{\Gal[f]}.
$$
of $f$-split objects and the actions of the internal groupoid $\Gal[f]$ in $\cX$, as in \ref{actions-self-ind}.

The proof consists in verifying that the monad $\mathbb{T}$ associated to the adjunction 
\begin{center}
 \begin{tikzpicture} 
 [cross line/.style={preaction={draw=white, -,
line width=3pt}}]
\matrix(m)[matrix of math nodes, minimum size=1.7em,
inner sep=0pt, 
row sep=3.3em, column sep=1em, text height=1.5ex, text depth=0.25ex]
 { 
  |(dc)|{\Split_Y(f)}	\\
 |(c)|{\cX_{\ov S(X)}} 	      \\ };
\path[->,font=\scriptsize,>=to, thin]
%
(dc) edge [bend left=30] node (ss) [right]{$U=S_X\,f^*$} (c)
(c) edge [bend left=30] node (ps) [left]{$F=f_!\,C_X$} (dc)
(ss) edge[draw=none] node{$\dashv$} (ps)
;
\end{tikzpicture}
\end{center}
of the monadic functor $U$ {(\cite[Cor. 5.1.21]{borceux-janelidze})}, with functorial part $T=UF$, is isomorphic to the monad $\mathbb{T'}$ on $\cX_{\ov S(X)}$ associated to the adjunction 
\begin{center}
 \begin{tikzpicture} 
 [cross line/.style={preaction={draw=white, -,
line width=3pt}}]
\matrix(m)[matrix of math nodes, minimum size=1.7em,
inner sep=0pt, 
row sep=3.3em, column sep=1em, text height=1.5ex, text depth=0.25ex]
 { 
  |(dc)|{\cX^{\Gal[f]}}	\\
 |(c)|{\cX_{\ov S(X)}} 	      \\ };
\path[->,font=\scriptsize,>=to, thin]
%
(dc) edge [bend left=30] node (ss) [right]{$U'$} (c)
(c) edge [bend left=30] node (ps) [left]{$F'$} (dc)
(ss) edge[draw=none] node{$\dashv$} (ps)
;
\end{tikzpicture}
\end{center}
where $U'$ is the forgetful functor omitting the groupoid action, and $F'$ is the `representable internal diagram' functor, whose functorial part is $T'=d_{1!}d_0^*$ and the category of algebras is the category $\cX^{\Gal[f]}$. Hence, we obtain equivalences
$$
\Split_Y[f]\stackrel{K^\mathbb{T}}{\longrightarrow} (\cX_{\ov S(X)})^\mathbb{T}\simeq (\cX_{\ov S(X)})^{\mathbb{T}'}\stackrel{K^{\mathbb{T}'}}{\longleftarrow} \cX^{\Gal[f]},
$$
where we wrote $K^\mathbb{T}$ and $K^{\mathbb{T}'}$ for the comparison functors of the respective monads. 

In this case, modulo the identification of the category of $\mathbb{T}$-algebras $(\cX_{\ov S(X)})^{\mathbb{T}}$ with  $\cX^{\Gal[f]}$, 
the functors realising the sought-after equivalence are the comparison functor 
$$
\Phi=K^\mathbb{T}:\Split_Y(f)\xrightarrow{\sim} (\cX_{\ov S(X)})^{\mathbb{T}},\ \ \   p\mapsto (U(p),U(\varepsilon_p)),
$$
and its left adjoint
$$
\Psi: (q,\nu)\mapsto
\begin{tikzcd}[cramped, column sep=normal, ampersand replacement=\&]
{\Coeq\left(FUF(q) \right.}\ar[yshift=2pt]{r}{F(\nu)} \ar[yshift=-2pt]{r}[swap]{\varepsilon_{Fq}} \&{\left.F(q)\right)}
\end{tikzcd},
$$
where $(Q\stackrel{q}{\to}X_0,\nu)$ is a $\mathbb{T}$-algebra and $\varepsilon:FU\to \id$ is the counit of the adjunction $F\dashv U$, and the coequaliser exists by the proof of Beck's monadicity criterion as in \cite[3.14]{barr-wells}.

By identifying $\mathbb{T}$ and $\mathbb{T}'$, and writing $\Gal[f]=(S(X\times_YX),S(X))=(G_1,X_0)$, the top arrow appearing in the coequaliser is obtained by applying $F$ to the action $G_1\times_{X_0} Q\stackrel{\nu}{\to} Q$, which, modulo the identification $X\times_{X_0}(G_1\times_{X_0}Q)\simeq X\times_{X_0}G_1\times_{X_0}Q$ gives the morphism
$$
\id\times \nu: X\times_{X_0}G_1\times_{X_0}Q\to X\times_{X_0}Q.
$$
If $p\in \Split_Y[f]$ then $f^*(p)\simeq C_X(q)$ for some $q$, whence the counit $\varepsilon_p$ is
$$
FUp=f_!C_XS_Xf^*p\simeq f_!C_XS_XC_Xq\simeq f_!C_Xq\simeq f_!f^*p\to p,
$$
so the bottom arrow $\varepsilon_{Fq}$ identifies with $f_!f^*Fq\to Fq$, which, modulo the isomorphism
$$
X\times_Y(X\times_{X_0}Q)\simeq (X\times_YX)\times_X(X\times_{X_0}Q)\simeq (X\times_{X_0}G_1)\times_X(X\times_{X_0}Q)\simeq X\times_{X_0}G_1\times_{X_0}Q,
$$
identifies with 
$$
\mu\times \id: X\times_{X_0}G_1\times_{X_0}Q\to X\times_{X_0}Q,
$$
where $\mu$ denotes the $\Gal[f]$-action on $X$. Thus, we may symbolically write the above coequaliser as the quotient
$$
\Psi(q,\nu)=\Coeq(\id\times\nu,\mu\times\id)=F(q)/\Gal[f]=(X\times_{X_0}Q)/\Gal[f],
$$
by the twisted-diagonal action of the Galois groupoid on $X\times_{X_0}Q$.


\subsection{Carboni-Magid-Janelidze Galois correspondence}

\begin{fact}[{\cite{carboni}}]
Let $\cC$ be a category with pullbacks and coequalisers of equivalence relations. In the presence of pullbacks, regular epimorphisms coincide with effective epimorphisms, and we call them \emph{quotients} for short. 

Let $G=(G_1, G_0,d_0,d_1,e,m)$ be an internal groupoid in $\cC$. Let us write $\tilde{G}\in\cC^G$ for the canonical action of $G$ on itself. There is a bijection
$$
\Sub(G)\simeq \Equiv(\tilde{G})
$$
between the set of subgroupoids of $G$ with the same object of objects and the set of equivalence relations on $\tilde{G}$ in $\cC^G$ as follows.

Given a subgroupoid $G'=(G_1',G_0)$ with $\iota:G_1'\hookrightarrow G_1$, the corresponding equivalence relation on $\tilde{G}$ is
$$
\begin{tikzcd}[cramped, column sep=4em, ampersand replacement=\&]
R_{G'}=G_1\times_{G_0}G_1'  \ar[yshift=2pt]{r}{\pi_1} \ar[yshift=-2pt]{r}[swap]{m(\id_{G_1}\times\iota)} \&G_1
\end{tikzcd}
$$  

Conversely, if $R\hookrightarrow \tilde{G}\times\tilde{G}$ is an equivalence relation on $\tilde{G}$ in $\cC^G$, we define the corresponding subgroupoid by the pullback
$$
 \begin{tikzpicture}
[cross line/.style={preaction={draw=white, -,
line width=4pt}}]
\matrix(m)[matrix of math nodes, row sep=1.8em, column sep=3.2em, text height=1.5ex, text depth=0.25ex]
{	|(0i)|{G_1'}		& |(1i)| {R}	\\   [.02em]
	|(0j)|{G_1}		& |(1j)| {G_1\times_{G_0}G_1}	\\};
\path[->,font=\scriptsize,>=to, thin]
(0i) edge node[above,pos=0.5]{} (1i) edge[right hook->] node[left,pos=0.5]{} (0j) 
(0j) edge node[above,pos=0.5]{$(\id_{G_1}, ed_1)$} (1j) 
(1i) edge[right hook->]  node[right,pos=0.5]{} (1j)
;
\end{tikzpicture}
$$ 
where $G_1\times_{G_0}G_1$ is the kernel pair of $d_1$.

A subgroupoid is called \emph{effective} is the associated equivalence relation is effective in the sense that it is the kernel pair of its coequaliser. 
\end{fact}

\begin{theorem}\label{th-cmj-corr}
With notation from \ref{ss:janelidze-gal}, suppose that $f:X\to Y$ is of relative Galois descent with Galois groupoid $G=\Gal[f]$ internal in $\cX$. There is an anti-isomorphism
$$
\text{\rm SplitQuo}[f]\simeq \text{\rm EffSub}(\Gal[f])
$$
that assigns
$$
 \begin{tikzpicture}
[cross line/.style={preaction={draw=white, -,
line width=4pt}}]
\matrix(m)[matrix of math nodes, row sep=.9em, column sep=.5em, text height=1.5ex, text depth=0.25ex]
{			& |(ij)| {X}		&				& 			\\   [.02em]
	|(i)|{P} 	&			& 	\mapsto		& \Gal[X\to P]	\\  [.02em]
			& |(0)|{Y} 		&				&			\\ [1em]
			& X/G'				&	\mapsfrom	&	G'		\\};			
\path[->,font=\scriptsize,>=to, thin]
(ij) edge node[left,pos=0.2]{$$} (i) edge node[right,pos=0.5]{$f$} (0) 
(i) edge node[below left,pos=0.5]{$p$} (0) 
;
\end{tikzpicture}
$$ 
between the ordered set of quotients of $X$ over $Y$ in $\Split[f]$ and the ordered set of effective subgroupoids of $\Gal[f]$.
\end{theorem}
\begin{proof}
We use the equivalence established by functors $\Phi$ and $\Psi$ from \ref{ss:janelidze-gal} and the fact that, in the presence of pullbacks, quotients (regular epimorphisms) agree with effective epimorphisms.

If the quotient $p$ of $X$ from the above diagram is $f$-split by $Q\stackrel{q}{\to}X_0=S(X)$, i.e., $f^*p\simeq C_X(q)$, applying the comparison functor $\Phi$ 
gives an effective quotient
 $$
 \begin{tikzpicture}
[cross line/.style={preaction={draw=white, -,
line width=4pt}}]
\matrix(m)[matrix of math nodes, row sep=.9em, column sep=.5em, text height=1.5ex, text depth=0.25ex]
{			& |(ij)| {\tilde{G}}			\\   [.02em]
	|(i)|{Q} 	&						\\  [.02em]
			& |(0)|{X_0} 				\\};
\path[->,font=\scriptsize,>=to, thin]
(ij) edge node[left,pos=0.2]{$$} (i) edge node[right,pos=0.5]{$$} (0) 
(i) edge node[below left,pos=0.5]{$q$} (0) 
;
\end{tikzpicture}
$$ 
and an effective equivalence relation $\tilde{G}\times_Q\tilde{G}\hookrightarrow\tilde{G}\times_{X_0}\tilde{G}$ on $\tilde{G}$. The corresponding  effective subgroupoid $G_P$ of $G$ is given by
$$
 \begin{tikzpicture}
[cross line/.style={preaction={draw=white, -,line width=4pt}}]
\matrix(m)[matrix of math nodes, row sep=1.8em, column sep=3.2em, text height=1.5ex, text depth=0.25ex]
{	|(0i)|{G_{P,1}}		& |(1i)| {G_1\times_Q G_1}	\\   [.02em]
	|(0j)|{G_1}		& |(1j)| {G_1\times_{G_0}G_1}	\\};
\path[->,font=\scriptsize,>=to, thin]
(0i) edge node[above,pos=0.5]{} (1i) edge[right hook->] node[left,pos=0.5]{} (0j) 
(0j) edge node[above,pos=0.5]{$(\id_{G_1}, ed_1)$} (1j) 
(1i) edge[right hook->]  node[right,pos=0.5]{} (1j)
;
\end{tikzpicture}
$$

Conversely,  an effective subgroupoid $G'$ of $G$ is associated with the action $$\tilde{G}/G'=\Coeq(R_{G'}\doublerightarrow{}{}\tilde{G})$$ 
in $\cX^G$ given as the coequaliser of its associated effective equivalence relation, and then taken to the split quotient
\begin{multline*}
\Psi(\tilde{G}/G')\simeq (X\times \tilde{G}/G')/G\simeq X/{G'}=f/G'
\\
=
\begin{tikzcd}[cramped, column sep=normal, ampersand replacement=\&]
{\Coeq\left(F(G_1'\to X_0)\to F(G_1\to X_0) \right.}\ar[yshift=2pt]{r}{\mu} \ar[yshift=-2pt]{r}[swap]{\text{\rm proj}} \&{\left.f\right)}
\end{tikzcd},
\end{multline*}
where $\mu:X\times_{X_0}G_1\to X$ denotes the action of $G$ on $X$.  

The assignments given above clearly establish an equivalence because they are constructed as restrictions of $\Phi$ and $\Psi$ to the appropriate full subcategories, but we find it useful to provide a direct proof of the correspondence.

By construction, the groupoid associated to $X/{G'}$ is
$$
G_{X/G'}=G\times_{G\times_{X_0}G}(G\times_{G/G'}G)\simeq G_{G\times_{X_0}G}(G\times_{X_0}G')\simeq G',
$$
where the isomorphism holds by effectivity. 

Conversely, if $P$ is $f$-split by $Q$, then 
$$
P\simeq \Psi(Q)=\Psi(G/G_P)= (X\times_{X_0} G/G_P)/G\simeq X/G_P.
$$

It remains to show that $G_{P,1}\simeq S(X\times_PX)$, i.e., 
$$
G_P\simeq \Gal[X\to P].
$$
By applying $C_X$ to the pullback diagram defining $G_{P,1}$ above, and using the fact that the right adjoint $C_X$ commutes with pullbacks, as well as the relations witnessing splitting of the objects involved, we obtain a pullback diagram
$$
 \begin{tikzpicture}
[cross line/.style={preaction={draw=white, -,line width=4pt}}]
\matrix(m)[matrix of math nodes, row sep=1.8em, column sep=3.2em, text height=1.5ex, text depth=0.25ex]
{	|(0i)|{C_X(G_{P,1})}		& |(1i)| {f^*(X)\times_{f^*(P)} f^*(X)}	\\   [.02em]
	|(0j)|{f^*(X)}		& |(1j)| {f^*(X)\times_{X}f^*(X)}	\\};
\path[->,font=\scriptsize,>=to, thin]
(0i) edge node[above,pos=0.5]{} (1i) edge[right hook->] node[left,pos=0.5]{} (0j) 
(0j) edge node[above,pos=0.5]{$(\id, \Delta\circ\text{proj})$} (1j) 
(1i) edge  node[right,pos=0.5]{} (1j)
;
\end{tikzpicture}
$$ 
which, using $f^*X\simeq X\times_YX$ and the groupoid structure of  $\bbG_f$, simplifies to $C_X(G_{P,1})\simeq X\times_PX$, whence $G_{P,1}\simeq S_X(X\times_PX)$. 
\end{proof}

\subsection{Indexed categorical Galois theory}

\begin{theorem}[{\cite[7.5.3, discussion after 7.6.2]{borceux-janelidze}}]\label{indexed-gal-th}

Suppose we are given
\begin{enumerate}
\item a functor $S:\cA\to \cX$;
\item pseudo-functors $K:\cX^\op\to\Cat$ and $L:\cA^\op\to\Cat$;
\item a pseudo-natural transformation $\alpha: K\circ S\Rightarrow L$; 
\item a `precategorical decomposition' of a morphism $X\stackrel{f}{\to}Y$ in $\cA$, i.e., a commutative diagram of morphisms of precategories
$$
 \begin{tikzpicture}
[cross line/.style={preaction={draw=white, -,
line width=4pt}}]
\matrix(m)[matrix of math nodes, row sep=1.2em, column sep=1.2em, text height=1.5ex, text depth=0.25ex]
{	|(i)|{\Disc{X}} 	&			& |(k)|{\Disc{Y}} 		\\  [.02em]
			& |(j)|{\C} 	&				\\};
\path[->,font=\scriptsize,>=to, thin]
(i) edge node[left,pos=0.6]{$i$} (j) edge node[above,pos=0.5]{$\Disc{f}$} (k) 
(j) edge node[right,pos=0.4]{$\pi$} (k) 
;
\end{tikzpicture}
$$ 
with $i_0=\id_X$, and
such that the components $\alpha_{C_0}$, $\alpha_{C_1}$ and $\alpha_{C_2}$ are full and faithful.

\end{enumerate}

If $(f,(i,\C,\pi))$ is of effective descent with respect to $L$, we have an equivalence of categories
$$
\Split_\alpha(f)\simeq K^{S\circ\C}.
$$

\end{theorem}

\section{Differential algebraic geometry}\label{s:dif-alg}

\subsection{Differential schemes}

A \emph{differential scheme} $$(X, (\cO_X,\delta_X))$$ is a differentially ringed space where $(X,\cO_X)$ is a scheme, and $\delta_X\in\Der(\cO_X,\cO_X)$.

A \emph{morphism of differential schemes} 
$$
(f,\varphi):(X,(\cO_X,\delta_X))\to (Y,(\cO_Y,\delta_Y))
$$
is a morphism of differentially ringed spaces which is also a scheme morphism, i.e., it is a scheme morphism $(f,\varphi):(X,\cO_X)\to (Y,\cO_Y)$ whose structure homomorphism is a morphism of differential rings $\varphi:(\cO_Y,\delta_Y)\to f_*(\cO_X,\delta_X)$, or, equivalently, its mate is a morphism of differential rings $\varphi^\sharp:f^*(\cO_Y,\delta_Y)\to (\cO_X,\delta_X)$. 

They constitute the \emph{category of differential schemes} denoted
$$
\DSch.
$$
We have an obvious functor 
$$
C:\Sch\to \DSch, \ \ \ (X,\cO_X)\mapsto (X,(\cO_X,0))
$$
that turns a scheme into a differential scheme with the trivial derivation $0$. 

Given a scheme morphism $(f,\varphi):(X,\cO_X)\to (S,\cO_S)$ and an $\cO_X$-module $\cF$, we say that an additive morphism $D:\cO_X\to\cF$ is an \emph{$S$-derivation} of $\cO_X$ to $\cF$ if it is an $f^*\cO_S$-derivation via $\varphi^\sharp:f^*\cO_S\to\cO_X$, or, equivalently, if $D_x:\cO_{X,x}\to \cF_x$ is an $\cO_{S,f(x)}$-derivation via $\varphi^\sharp_x:\cO_{S,f(x)}\to\cO_{X,x}$ for every $x\in X$. The collection of all $S$-derivations of $\cO_X$ to $\cF$ is denoted
$$
\Der_S(\cO_X,\cF).
$$

A differential scheme $(X,(\cO_X,\delta_X))$ equipped with a scheme morphism $(f,\varphi):(X,\cO_X)\to (S,\cO_S)$ is called an \emph{$S$-differential scheme} provided $\delta_X\in \Der_S(\cO_X,\cO_X)$. 

Clearly, an $S$-differential scheme $(X,\delta_X)$ is a morphism of differential schemes $(X,\delta_X)\to C(S)=(S,0)$. Thus, the category of $S$-differential schemes is the slice category
$$
\DSch_S=\DSch_{\ov C(S)}.
$$

\subsection{Differential schemes and vector fields}

Let $(f,\varphi):X\to S$ be a scheme morphism. By \cite[16.5.3]{EGAIV4}, the universal differential
$$
d_{X/S}:\cO_X\to \Omega^1_{X/S}
$$
is an $S$-derivation, and composing with $d_{X/S}$ induces an isomorphism of $\Gamma(X,\cO_X)$-modules
$$
\Hom_{\cO_X}(\Omega^1_{X/S},\cF)\stackrel{\sim}{\to}\Der_S(\cO_X,\cF)
$$
for any $\cO_X$-module $\cF$. 

Let $$T_{X/S}=\mathbf{V}(\Omega^1_{X/S})$$ be the tangent bundle of $X$ relative to $S$, defined as the vector bundle associated to the quasi-coherent $\cO_X$-module $\Omega^1_{X/S}$ (i.e., the spectrum of the quasi-coherent $\cO_X$-algebra $\mathrm{Sym}(\Omega^1_{X/S})$) see \cite[16.5.12]{EGAIV4}. 

For a point $x\in X$, the \emph{tangent space of $X$ at $x$ relative to $S$} is defined \cite[16.5.13]{EGAIV4} as
$$
T_{X/S}(x)=\left(T_{X/S}\times_X\spec(\kappa(x))\right)(\kappa(x))\simeq \Hom_{\kappa(x)}(\Omega^1_{X/S}\otimes_{\cO_X}\kappa(x),\kappa(x)).
$$

Choosing an $S$-derivation $\delta_X\in\Der_S(\cO_X,\cO_X)$, we therefore obtain a morphism of $\cO_X$-modules
$$
\Omega^1_{X/S}\to \cO_X,
$$
which yields a section $X\to T_{X/S}$ of the relative tangent bundle, that we think of as a \emph{vector field} on $X$ relative to $S$.

Pointwise, for every $x\in X$, pulling back the $\cO_X$-modules to $\kappa(x)$-modules via the morphism $\spec(\kappa(x))\to X$ yields a morphism
$$
\Omega^1_{X/S}\otimes_{\cO_X}\kappa(x)\to \kappa(x),
$$
an element of $T_{X/S}(x)$. This construction is in line with the classical notion of vector field in differential geometry as a map that sends a point to a vector in the corresponding tangent space. 

A point $x\in X$ is a \emph{leaf} for the vector field given by $\delta_X$ if the corresponding tangent vector at $x$ is $0$. 

The \emph{space of leaves} of a differential scheme $X$ is a ringed space 
$$
(X^\delta,\cO_{X^\delta}),
$$
where the set of leaves $X^\delta$ is endowed with the topology induced via the inclusion $i:X^\delta\to X$, and the structure sheaf is 
$$
\cO_{X^\delta}=i^{-1}\Const(\cO_X).
$$
A morphism $f:X\to Y$ of differential schemes induces a morphism of ringed spaces
$$
f^\delta:X^\delta\to Y^\delta.
$$

\subsection{Differential spectra and affine differential schemes}

The spectrum 
of the underlying ring of a differential ring $(A,\delta_A)$ carries a natural structure of a differential scheme
$$
\spec(A,\delta_A)=(\spec(A), (\cO_{\spec(A)},\delta_{\spec(A)})).
$$
Indeed, $$\delta_{\spec(A)}:\cO_{\spec(A)}\to \cO_{\spec(A)}$$ is determined on basic opens $D(f)$ in $\spec(A)$, for $f\in A$ by setting
$$
\delta_{\spec(A), D(f)}: \cO_{\spec(A)}(D(f))=A_f\to A_f=\cO_{\spec(A)}(D(f)), \ \ \ \frac{a}{f}\mapsto \frac{\delta_A(a)f-a\delta_A(f)}{f^2}.
$$
This construction extends to a functor 
$$
\spec:\DRng^\op\to \DSch,
$$
right adjoint to the global sections functor 
$$
\Gamma:(X,(\cO_X,\delta_X))\mapsto (\cO_X(X),\delta_{X,X})
$$
as the diagram
$$
 \begin{tikzpicture} 
 [cross line/.style={preaction={draw=white, -,
line width=3pt}}]
\matrix(m)[matrix of math nodes, minimum size=1.7em,
inner sep=0pt, 
row sep=3.3em, column sep=1em, text height=1.5ex, text depth=0.25ex]
 { 
  |(dc)|{\DSch}	\\
 |(c)|{\DRng^\op} 	      \\ };
\path[->,font=\scriptsize,>=to, thin]
%
(dc) edge [bend right=30] node (ss) [left]{$\Gamma$} (c)
(c) edge [bend right=30] node (ps) [right]{$\spec$} (dc)
(ss) edge[draw=none] node{$\dashv$} (ps)
;
\end{tikzpicture}
$$
depicts. 

A differential scheme is \emph{affine}, if it is isomorphic to a spectrum of a differential ring. Clearly, the category of affine differential schemes is anti-equivalent to the category of differential rings, i.e., 
$$
\DAff\simeq \DRng^\op.
$$
The set of leaves 
$$
\spec(A,\delta_A)^\delta
$$
is in bijection with the set of \emph{differential primes} of $A$.

\subsection{Differential schemes as precategory actions}\label{s:diffschprect}

Given a scheme $S$, let us consider the diagram of quasi-coherent $\cO_S$-algebras
$$
 \begin{tikzpicture} 
\matrix(m)[matrix of math nodes, row sep=0em, column sep=3em, text height=1.5ex, text depth=0.25ex]
 {
|(2)|{\cO_S[\epsilon_0,\epsilon_1]/(\epsilon_0^2,\epsilon_0\epsilon_1,\epsilon_1^2)}  & [1em] |(1)|{\cO_S[\epsilon]/(\epsilon^2)}		&[1em] |(0)|{\cO_S} \\
 }; 
\path[->,font=\scriptsize,>=to, thin]
([yshift=1em]1.west) edge node[above=-2pt]{$\epsilon_0\mapsfrom\epsilon$} ([yshift=1em]2.east) 
(1) edge node[above=-2pt]{$\epsilon_0+\epsilon_1\mapsfrom\epsilon$} (2)
([yshift=-1em]1.west) edge node[above=-2pt]{$\epsilon_1\mapsfrom\epsilon$} ([yshift=-1em]2.east) 
([yshift=1em]0.west) edge node[above=-2pt]{$\id+0 $} ([yshift=1em]1.east) 
(1)  edge node[above=-2pt]{$\epsilon\mapsto 0$} (0) 
([yshift=-1em]0.west) edge node[above=-2pt]{$\id+0$} ([yshift=-1em]1.east) 
;
\end{tikzpicture}
$$
Applying the spectrum of quasi-coherent $\cO_S$-algebras functor \cite[1.3]{EGAII}, we obtain a precategory $\bbD(S)$
$$
 \begin{tikzpicture} 
\matrix(m)[matrix of math nodes, row sep=0em, column sep=3em, text height=1.5ex, text depth=0.25ex]
 {
|(2)|{S_2}  & [1em] |(1)|{S_1}		&[1em] |(0)|{S_0} \\
 }; 
\path[->,font=\scriptsize,>=to, thin]
([yshift=1em]2.east) edge node[above=-2pt]{$r_0$} ([yshift=1em]1.west) 
(2) edge node[above=-2pt]{$m$} (1)
([yshift=-1em]2.east) edge node[above=-2pt]{$r_1$} ([yshift=-1em]1.west) 
([yshift=1em]1.east) edge node[above=-2pt]{$d_0$} ([yshift=1em]0.west) 
(0)  edge node[above=-2pt]{$n$} (1) 
([yshift=-1em]1.east) edge node[above=-2pt]{$d_1$} ([yshift=-1em]0.west) 
;
\end{tikzpicture}
$$
in $\Sch_{\ov S}$ consisting of schemes affine over $S_0=S$, and the underlying morphisms of topological spaces are all identities. 

Note, if we write 
$$
\bbD(\Z)=\bbD(\spec(\Z)), 
$$
then 
$$
\bbD(S)=\bbD(\Z)\times\Disc{S}=(S\times\spec(\Z[\epsilon_0,\epsilon_1]/(\epsilon_0^2,\epsilon_0\epsilon_1,\epsilon_1^2)),S\times\spec(\Z[\epsilon]/(\epsilon^2)), S).
$$

\begin{proposition}\label{diff-sch-precats}
The category of $S$-differential schemes is equivalent to the category of $\bbD(S)$-actions in $\Sch_{\ov S}$ (cf.~\ref{actions-self-ind}), 
$$
\DSch_S\simeq (\Sch_{\ov S})^{\bbD(S)}.
$$
\end{proposition}
\begin{proof}
Using \ref{gen-dd}, an action is determined by a scheme morphism $X_0=X\to S=S_0$ and an $S_1$-automorphism $\alpha:X_1\to X_1$, where $X_1=X\times_{S_0}{S_1}$ satisfying $n^* \alpha=\id$ and the cocycle condition. Equivalently, it is given by an $\cO_S[\epsilon]/(\epsilon^2)$-automorphism of $\cO_{X_1}=\cO_{X}[\epsilon]/(\epsilon^2)$ which, tensored by the augmentation morphism $\eta_S:\cO_S[\epsilon]/(\epsilon^2)\to\cO_S$ gives $\id_{\cO_{X}}$, and it follows that it must be of the form 
$$
\id_{\cO_{X}[\epsilon]/(\epsilon^2)}+\epsilon\delta_\alpha\circ\eta_{X},
$$
where $\delta_\alpha\in \Der_S(\cO_{X},\cO_X)$ and $\eta_{X}:\cO_X[\epsilon]/(\epsilon^2)\to\cO_X$ is the augmentation homomorphism. Note that such a morphism always has an inverse $id_{\cO_{X}[\epsilon]/(\epsilon^2)}-\epsilon\delta_\alpha\circ\eta_{X}$.

The cocycle condition is trivially satisfied/superfluous, since $r_i^*\alpha$ is determined by the automorphism of  
$\cO_X[\epsilon_0,\epsilon_1]/(\epsilon_0^2,\epsilon_0\epsilon_1,\epsilon_1^2)$ given by $\id+\epsilon_i \delta_\alpha\circ\eta_{X,2}$ for $i=0,1$, and $m^*\alpha$ is given by $\id+(\epsilon_0+\epsilon_1)\delta_\alpha\circ\eta_{X,2}$, where we wrote $\eta_{X,2}$ for the augmentation homomorphism $\cO_X[\epsilon_0,\epsilon_1]/(\epsilon_0^2,\epsilon_0\epsilon_1,\epsilon_1^2)\to\cO_X$, and these morphisms compose in the way prescribed by the cocycle condition. 

Thus, an action uniquely determines a derivation, and we obtain an $S$-differential scheme $(X,(\cO_X,\delta_\alpha))$. 

Conversely, given an $S$-differential scheme $(X,(\cO_X,\delta_X))$, we can create a $\bbD(S)$-action from the diagram
$$
 \begin{tikzpicture} 
\matrix(m)[matrix of math nodes, row sep=0em, column sep=3em, text height=1.5ex, text depth=0.25ex]
 {
|(2)|{\cO_X[\epsilon_0,\epsilon_1]/(\epsilon_0^2,\epsilon_0\epsilon_1,\epsilon_1^2)}  & [1em] |(1)|{\cO_X[\epsilon]/(\epsilon^2)}		&[1em] |(0)|{\cO_X} \\
 }; 
\path[->,font=\scriptsize,>=to, thin]
([yshift=1em]1.west) edge node[above=-2pt]{$\epsilon_0\mapsfrom\epsilon$} ([yshift=1em]2.east) 
(1) edge node[above=-2pt]{$\epsilon_0+\epsilon_1\mapsfrom\epsilon$} (2)
([yshift=-1em]1.west) edge node[above=-2pt]{$\epsilon_1\mapsfrom\epsilon$} ([yshift=-1em]2.east) 
([yshift=1em]0.west) edge node[above=-2pt]{$\id+0$} ([yshift=1em]1.east) 
(1)  edge node[above=-2pt]{$\epsilon\mapsto 0$} (0) 
([yshift=-1em]0.west) edge node[above=-2pt]{$\id+\epsilon \delta_X$} ([yshift=-1em]1.east) 
;
\end{tikzpicture}
$$
by applying the spectrum of quasi-coherent $\cO_P$-algebras functor. All the resulting morphisms are cartesian, so we obtain an action by \ref{action-fibred}.
\end{proof}

\begin{remark}\label{C-for-precat}
The scheme of connected components is
$$
\pi_0(\bbD(S))=S
$$
because it is calculated as the coequaliser of morphisms $d_0, d_1$ in $\bbD(S)$, which agree. Hence, writing 
$$
\eta:\bbD(S)\to \Disc{S}
$$
for the associated morphism of precategories,  the functor 
$$
C:\Sch_{\ov S}\to \DSch_S, \ \ \  T\mapsto (T,0)
$$
corresponds, through the equivalence \ref{diff-sch-precats}, to the functor
$$
\eta^*: \Sch_{\ov S}\to (\Sch_{\ov S})^{\bbD(S)}.
$$
\end{remark}

\begin{definition}\label{delta-P}
Let 
$$
\cP:\Sch_{\ov S}^\op\to \Cat
$$
be a pseudofunctor on the category of $S$-schemes. 

We define a pseudofunctor
$$
\delta\da\cP:\DSch_S^\op\to \Cat, \ \ \ \delta\da\cP(X,\delta_X)=\cP^{\bbX},
$$
where $\bbX\in\Precat(\Sch_{\ov S})$ is the $\bbD(S)$-action corresponding to $X$, considered as a precategory by \ref{action-fibred}.
\end{definition}

\begin{remark}\label{diff-self}
Assume that the indexed data $\cP:\Sch_{\ov S}^\op\to \Cat$ is a full sub-pseudofunctor of the self-indexing of $\Sch_{\ov S}$, 
i.e., for $Y\in \Sch_{\ov S}$, $\cP(Y)$ is a full subcategory of $\Sch_{\ov Y}$. We can think of $\cP$ as being associated to a class of morphisms in $\Sch_{\ov S}$ that is stable under pullback, 
Then, the indexed data $\delta\da\cP$ is a natural differential analogue of that class in the sense that, for $(X,\delta_X)\in\DSch_S$,
$$
\delta\da\cP(X,\delta_X)=\{(P,\delta_P)\to (X,\delta_X)\in \DSch_S: P\to X \in \cP(X)\},
$$
i.e., it consists of those morphisms of differential schemes with target $(X,\delta_X)$ whose underlying morphism of schemes belongs to the class $\cP$. 
\end{remark}

\begin{example}
If $\cP(Y)$ is the category of quasi-projective morphisms $Q\to Y$, then, given an $S$-differential scheme $(X,\delta_S)$, the fibre $\delta\da\cP(X,\delta_X)$ is the category of $S$-differential scheme morphisms $(P,\delta_P)\to (X,\delta_X)$ such that $P\to X$ is quasi-projective.
\end{example}

\subsection{Differential schemes as formal group actions}

The \emph{additive formal group scheme} is given as the formal scheme
$$
\Ga=\spf(\Z[[t]]),
$$
where the group operation is deduced from comultiplication
$$
\Z[[t]]\to \Z[[t]]\widehat{\otimes}\Z[[u]]\simeq\Z[[t,u]], \ \ t\mapsto t+u,
$$
and the identity section from the counit map
$$
\Z[[t]]\to \Z, \ \ t\mapsto 0.
$$
The associated functor of points sends a ring $R$ to the additive group of its nilpotents
$$
\Ga(R)=\mathop{\rm Nil}(R).
$$

\begin{fact}[{Bardavid, \cite[8.3.1]{bardavid}}]\label{diffsch-Ga-action}
There is bijective correspondence between formal group actions of $\Ga$ on a scheme $X$ and systems of Hasse-Schmidt derivations on its structure sheaf $\cO_X$. 

In particular, differential schemes in characteristic $0$ are precisely $\Ga$-actions, 
$$
\DSch_{\Q}= (\Sch_{\ov\Q})^{\Ga}.
$$
\end{fact}
Indeed, an action is a formal scheme morphism
$$
\rho:\Ga\times X\to X
$$
satisfying the usual axioms. It is `infinitesimal' in the sense that it does not affect the underlying topological space of $X$, and,
at the level of structure sheaves,  for $U$ open in $X$, it is given by a `Taylor expansion' map
$$
\cO_X(U)\to \cO(U)[[t]], \ \ f\mapsto \sum_i d_i(f) t^i,
$$
where $(d_i)_{i\in\N}$ is a system of Hasse-Schmidt derivations.  

In characteristic $0$, we have that $\delta_i=\frac{\delta^i}{i!}$ for a derivation $\delta$ on $\cO_X$. In terms of functors of points, the action is given by the expression
\begin{alignat}{3}
\mathop{\rm Nil}(R)&\times \Hom(\spec(R), X)  &\longrightarrow  &\Hom(\spec(R),X),\\
(\epsilon &,  (\varphi,\varphi^\sharp)) & \longmapsto & (\varphi, \sum_i \frac{\varphi^\sharp\circ\delta^i}{i!}\epsilon^i),
\end{alignat}
where $\varphi^\sharp:\cO_X\to \varphi_*\cO_{\spec(R)}$, and we consider $\epsilon\in \mathop{\rm Nil}(R)$ as a global section of  $\cO_{\spec(R)}$, and the sum is finite because $\epsilon$ is nilpotent. 

\subsection{Trajectories and leaves}

Let $(X,\delta_X)$ be an $S$-differential scheme, where $S$ is a $\Q$-scheme. Let us write $G=\Ga\times S$ for the formal additive group scheme considered over $S$, let $\rho: G\times_S X\to X$ be the corresponding infinitesimal action by $\delta_X$, and write $\psi=(\rho,p_2): G\times_S X\to X\times_S X$. 

Given a generalised point $x:T\to X$, the corresponding \emph{action map} $a_x$ is obtained as the pullback
$$
 \begin{tikzpicture} 
\matrix(m)[matrix of math nodes, row sep=2em, column sep=2em, text height=1.9ex, text depth=0.25ex]
 {
 |(1)|{G\times_S T}		& |(2)|{G\times_S X} 	\\
 |(l1)|{X\times_S T}		& |(l2)|{X\times_S X} 	\\
 }; 
\path[->,font=\scriptsize,>=to, thin]
(1) edge node[above]{} (2) edge node[left]{$a_x$}  (l1)
(2) edge node[right]{$\psi$} (l2) 
(l1) edge node[above]{$\id\times x$}  (l2);
\end{tikzpicture}
$$
and it is given more explicitly by $a_x=(\rho\circ \id_G\times x,p_2)$. 

\begin{definition}\label{orbit}
With above notation, the \emph{orbit of $x$} is the scheme-theoretic image of the action map $a_x$,
$$
O(x)=\Im(a_x),
$$
and the \emph{reduced orbit of $x$} is the underlying reduced closed subscheme 
$$
|O|(x)=|O(x)|.
$$
\end{definition}

\begin{lemma}\label{lemma-traj}
Let $x\in X$ be a scheme-theoretic point, and consider the corresponding morphism $\bar{x}:\spec(\kappa(x))\to X$. The scheme-theoretic image of the composite morphism
$$
t_x: G\times_S\spec(\kappa(x))\stackrel{a_x}{\longrightarrow}X\times_S\\spec(\kappa(x))\stackrel{\pi_1}{\longrightarrow}X
$$
is integral, and its generic point 
is a leaf.
\end{lemma}
\begin{proof}
The statement is local on $X$, so we may assume that $X=\spec(A,\delta)$ is affine. The morphism $a_x$ corresponds to the map
$$
A\to \kappa(x)[[t]], \ \ \ f\mapsto \sum_i \frac{\alpha(\delta^i f)}{i!}t^i,
$$
where $\alpha: A\to \kappa(x)$ is associated with $\bar{x}$ and the point $x$ corresponds to $\p=\ker(\alpha)$.  The kernel of the above map is 
$$
\p_{\sharp}=\{f\in A: \delta^n(f)\in \p\text{ for all }n\},
$$
and Keigher has shown that $\p_{\sharp}$ is prime in \cite[1.5]{keigher-prime-diff-idl}.
\end{proof}

\begin{definition}\label{trajectory}
If $x\in X$ is a scheme-theoretic point, we define its \emph{trajectory} as the unique leaf $\Traj(x)$ satisfying
$$
\Im(t_x)=\overline{\{\Traj(x)\}},
$$
and \ref{lemma-traj} shows that the definition is meaningful and agrees with the notion from \cite{bardavid}. We obtain a morphism of ringed spaces
$$
\Traj: X\to X^\delta,
$$
which is shown by Bardavid \cite[8.5]{bardavid} to be submersive with
$$
\cO_{X^\delta}=\Const(\Traj_*\cO_X).
$$
\end{definition}

\begin{lemma}\label{O-vs-Traj}
With the above notation, for $x\in X$, $\Traj(x)$ is the generic point of the scheme-theoretic image of $O(x)$ via $\pi_1$, i.e.
$$
\Im(t_x)=\pi_1(O(x)).
$$
\end{lemma}

\begin{lemma}\label{O-under-maps}
Let $f:X\to Y$ be a morphism of S-differential schemes. 
If $x:T\to X$ is a generalised point, and $y=f\circ x$ its image in $Y$, then $O(y)$ is the scheme-theoretic image of $O(x)$ under the morphism $f\times \id_T$,
$$
O(y)=(f\times\id_T)(O(x)).
$$
Consequently, if $x\in X$ is a scheme theoretic point, then
$$
\Traj(f(x))=f(\Traj(x)).
$$
\end{lemma}
Both lemmas are immediate using the familiar behaviour of scheme-theoretic images of the composite morphism.

\begin{lemma}\label{O-basech}
With the above notation, let $\eta':X\to S'$ be the base change of $\eta$ by a morphism $S'\to S$. Let $x':T\to X'$ be a generalised point of $X'$, and let $x:T\to X$ be its projection. Then
$$
O(x')\simeq O(x).
$$
\end{lemma}
\begin{proof}
By using the definition of orbits maps, we obtain a cartesian diagram
$$
 \begin{tikzpicture} 
\matrix(m)[matrix of math nodes, row sep=2em, column sep=2em, text height=1.9ex, text depth=0.25ex]
 {
 |(1)|{G'\times_{S'}T}		& |(2)|{G\times_ST} 	\\
 |(l1)|{X'\times_{S'}T}		& |(l2)|{X\times_ST} 	\\
 }; 
\path[->,font=\scriptsize,>=to, thin]
(1) edge node[above]{} (2) edge node[left]{$a_{x'}$}  (l1)
(2) edge node[right]{$a_x$} (l2) 
(l1) edge node[above]{}  (l2);
\end{tikzpicture}
$$
where the horizontal arrows are isomorphisms, so we conclude that scheme-theoretic images of the vertical arrows are isomorphic. 
\end{proof}

\begin{lemma}\label{O-precomp}
Consider the composite $x':T'\stackrel{t}{\longrightarrow}T\stackrel{x}{\longrightarrow}X$, where $t$ is a free affine morphism. Then
$$
O(x')\simeq O(x)\times_{X\times_ST}X\times_ST'.
$$
In particular, if $x\in X$ is a scheme-theoretic point, and $t:\spec(L)\to\spec(\kappa(x))$ for some field $L$ containing $\kappa(x)$, then
$$
\pi_1(|O|(x'))=\Traj(x). 
$$
\end{lemma}
\begin{proof}
Both squares of the diagram
$$
 \begin{tikzpicture} 
\matrix(m)[matrix of math nodes, row sep=2em, column sep=2em, text height=1.5ex, text depth=0.25ex]
 {
|(u2)|{G\times_S T'}  & [0em] |(u1)|{G\times_S T}		&[0em] |(u0)|{G\times_SX} \\
|(2)|{X\times_ST'}  & [0em] |(1)|{X\times_ST}		&[0em] |(0)|{X\times_SX} \\
 }; 
\path[->,font=\scriptsize,>=to, thin]
(2) edge node[above=-2pt]{$\id\times t$} (1)

(1) edge node[above=-2pt]{$\id\times x$} (0) 

(u2) edge (u1)

(u1) edge node[above=-2pt]{} (u0) 
(u2) edge node[right]{$a_{x'}$} (2)
(u1) edge node[right]{$a_{x}$} (1)
(u0) edge node[right]{$\psi$} (0)
;
\end{tikzpicture}
$$
are cartesian, which means that $a_{x'}$ is the base change of $a_x$ by the affine free morphism $\id\times T$. The free base change formula for scheme-theoretic images holds even in the case where morphisms are not quasi-compact by 
\cite[Prop.~4]{herrero}.

The second claim follows by taking scheme-theoretic images of the closed subschemes 
$$
 \begin{tikzpicture} 
\matrix(m)[matrix of math nodes, row sep=2em, column sep=2em, text height=1.9ex, text depth=0.25ex]
 {
 |(1)|{O(x')}		& |(2)|{O(x)} 	\\
 |(l1)|{X\times_{S}T'}		& |(l2)|{X\times_ST} 	\\
 }; 
\path[->,font=\scriptsize,>=to, thin]
(1) edge node[above]{} (2) edge node[left]{}  (l1)
(2) edge node[right]{} (l2) 
(l1) edge node[above]{}  (l2);
\end{tikzpicture}
$$
via first projections into $X$.
\end{proof}

\subsection{Categorical scheme of leaves}

The following definition is inspired by the notion of `espace des feuilles grossier' that appears in Bardavid's thesis \cite[4.2]{bardavid} and that of `quotient discret cat\'egorique' from Ayoub's paper \cite[3.2.4]{ayoub}, which ultimately stems from the notion of \emph{categorial quotient} familiar in algebraic geometry, see \cite[0.5]{mumford-git}, for example. 

\begin{definition}\label{def-cat-schl}
Let $(X,\delta_X)$ be an $S$-differential scheme. A scheme $T\in \Sch_{\ov S}$, together with a morphism of $S$-differential schemes $\eta:(X,\delta_X)\to C(T)=(T,0)$ is a \emph{categorical $S$-scheme of leaves}, if it is a universal morphism from $(X,\delta_X)$ to the functor $C$ in the sense that any other morphism of $S$-differential schemes $\eta':(X,\delta_X)\to C(T')=(T',0)$ factors through $\eta$,  i.e., there is a unique $S$-morphism $f:T\to T'$ such that $\eta'=C(f)\circ \eta$. In other words, the solid part of the diagram
$$
 \begin{tikzpicture}
[cross line/.style={preaction={draw=white, -,
line width=4pt}}]
\matrix(m)[matrix of math nodes, row sep=1.2em, column sep=1.2em, text height=1.5ex, text depth=0.25ex]
{	|(i)|{(X,\delta_X)} 	&			& |(k)|{C(T)} 		\\  [.02em]
			& |(j)|{C(T')} 	&				\\};
\path[->,font=\scriptsize,>=to, thin]
(i) edge node[below left,pos=0.5]{$\eta'$} (j) edge node[above,pos=0.5]{$\eta$} (k) 
(k) edge[style=dashed] node[below right,pos=0.5]{$\exists! C(f)$} (j) 
;
\end{tikzpicture}
$$ 
can be uniquely completed by a dashed arrow to a commutative diagram.
\end{definition}

\begin{remark}\label{no-pi0-affine}
The categorical scheme of leaves need not exist for an arbitrary differential scheme $(X,\delta_X)$, but when it does, it is unique up to unique isomorphism, and we denote it
$$
\pi_0(X).
$$
\end{remark}

\begin{remark}
Given a differential ring $(A,\delta_A)$, it is difficult to speculate whether
$
\pi_0(\spec(A,\delta_A))
$
exists. Note that the scheme $\spec(\Const(A,\delta_A))$ only satisfies the universal property required of a categorical scheme of leaves in the category of \emph{affine} differential schemes, but not necessarily in the category of all differential schemes. 
\end{remark}

\begin{lemma}\label{cat-schl-pi0}
Let $\bbX\in \Precat(\Sch_{\ov S})$ be the $\bbD(S)$-action associated to an $S$-differential scheme $(X,\delta)$, considered as a precategory. Then, the categorical scheme of leaves of $(X,\delta_X)$ is isomorphic to the scheme of connected components
$$
\begin{tikzcd}[cramped, column sep=normal, ampersand replacement=\&]
{\pi_0(\bbX)=\Coeq\left(X_1 \right.}\ar[yshift=2pt]{r}{d_0} \ar[yshift=-2pt]{r}[swap]{d_1} \&{\left.X_0\right)}
\end{tikzcd}
$$
of $\bbX$ in the category $\Sch_{\ov S}$, 
whenever either of the objects exist. In other words, 
$$
\pi_0(\bbX)\simeq\pi_0(X).
$$
\end{lemma}
\begin{proof}
Using \ref{part-adj-pi0}, we obtain that
$$
\DSch_S((X,\delta_X),C(T'))\simeq (\Sch_{\ov S})^{\bbD(S)}(\bbX,\eta^*T')\simeq \Sch_{\ov S}(\pi_0(\bbX), T'),
$$
whenever $\pi_0(\bbX)$ exists, hence $\pi_0(\bbX)$ satisfies the universal property of the categorical scheme of leaves. 
\end{proof}

\begin{definition}[{\cite[4.1]{bardavid}}]\label{class-simple}
Let $S$ be a spectrum of a field. A differential $S$-scheme $(X,\delta)$ is \emph{classically simple} if $X$ is integral whose only leaf is its generic point, and
$$
\pi_0(X)=S.
$$
\end{definition}

\begin{fact}\label{bardavid-simple}
Bardavid shows in \cite[4.1]{bardavid} that an integral differential scheme $(X,\delta)$ whose only leaf is its generic point has a categorical scheme of leaves 
$$
\pi_0(X)=\spec(\Const(H^0(X,\cO_X))),
$$  
and  that $\pi_0(X)$ remains the categorical scheme of leaves for any open subscheme of $X$. 

In particular, for a \emph{simple} differential ring $(A,\delta)$, 
$$
\pi_0(\spec(A,\delta))=\spec(\Const(A,\delta)). 
$$
\end{fact}

\subsection{Geometric scheme of leaves}

\begin{definition}\label{geom-quot}
A morphism of differential schemes $$\eta:(X,\delta)\to C(Q)=(Q,0)$$ is called a \emph{geometric quotient}, making $Q$ a \emph{geometric scheme of leaves}, if it satifies
\begin{enumerate}
\item the \emph{topological condition}: $\eta$ is surjective and submersive;
\item the \emph{orbit condition}: for any algebraically closed field $L$ and geometric points $\bar{x},\bar{x}'\in X(L)$,
$$
\eta(\bar{x})=\eta(\bar{x}')\ \ \  \text{ implies } \ \ \ |O|(\bar{x})=|O|(\bar{x}').
$$
\item the \emph{sheaf condition}: the sequence of quasicoherent $\cO_Q$-modules
$$
0\to \cO_Q\stackrel{\eta^\sharp}{\longrightarrow}\eta_*\cO_X\stackrel{\eta_*\delta_X}{\longrightarrow}\eta_*\cO_X
$$
is exact, i.e., $\cO_Q\simeq\Const(\eta_*\cO_X)$.
\end{enumerate}
A morphism $\eta$ as above is a \emph{Bardavid quotient} (\cite[8.4.2]{bardavid}), when we replace the orbit condition by 
\begin{itemize}
\item[(2')] the \emph{trajectory condition}: for $x,x'\in X$, 
$$
\eta(x)=\eta(x') \text{ implies } \Traj(x)=\Traj(x').
$$
\end{itemize}
\end{definition}

\begin{lemma}\label{class-simple-fibrewise}
Let $f:(X,\delta_X)\to (Y,\delta_Y)$ be a morphism of differential schemes such that $f^\delta:X^\delta\to Y^\delta$ is injective, or, equivalently, for every leaf $y\in Y^\delta$, the fibre $X_y$ is classically simple over $\spec(\kappa(y))$. Then, for every $x, x'\in X$, $f(x)=f(x')$ implies $\Traj(x)=\Traj(x')$. 
\end{lemma}
\begin{proof}
Note that $f^{-1}(y)$ is the underlying space of $X_y=X\times_Y\spec(\kappa(y))$, so the conditions are indeed equivalent. 
Assuming $f(x)=f(x')$ and using \ref{O-under-maps}, we have that $f(\Traj(x))=\Traj(f(x))=\Traj(f(x'))=f(\Traj(x')$. Since $f$ is injective on leaves, we deduce that $\Traj(x)=\Traj(x')$.
\end{proof}

\begin{lemma}\label{char-orbit-cond}
Let $\eta:(X,\delta)\to (Q,0)$ be a morphism of $S$-differential schemes. The following conditions are equivalent:
\begin{enumerate}
\item the orbit condition;
\item for $x,x'\in X$, $\eta(x)=\eta(x')$  implies that there exists a field $L$ extending $\kappa(x),\kappa(x')$,  with
$$
|O|(x_L)=|O|(x'_L),
$$
where we write $x_L$ for the composite $\spec(L)\to\spec(\kappa(x))\stackrel{x}{\longrightarrow}X$;
\item For any 
geometric point $\bar{s}\in S(L)$, the base change $\eta_{\bar{s}}:X_{\bar{s}}\to Q_{\bar{s}}$ satisfies the trajectory condition, i.e., for any $x, x'\in X_{\bar{s}}$, 
$$
\eta_{\bar{s}}(x)=\eta_{\bar{s}}(x')\ \ \ \text{ implies }\ \ \ \Traj(x)=\Traj(x'). 
$$
\end{enumerate}
\end{lemma}
\begin{proof}
To see that (1) implies (2), consider $x, x'\in X$ such that $\eta(x)=\eta(x')$, and choose any algebraically closed field $L$ that extends both $\kappa(x)$ and $\kappa(x')$. Then $x_L, x'_L\in X(L)$ are geometric points with $\eta(x_L)=\eta(x'_L)$, so by (1) we obtain $|O|(x_L)=|O|(x'_L)$.

To see that (2) implies (3), take $x, x'\in X_{\bar{s}}$ such that $\eta_{\bar{s}}(x)=\eta_{\bar{s}}(x')$. If $x_0, x_0'\in X$ are the images of $x, x'$ in $X$, we still have that $\eta(x_0)=\eta(x'_0)$. Hence, condition (2) implies that there exists a field $K$ extending both $\kappa(x_0)$ and $\kappa(x'_0)$ such that $|O|(x_{0,K})=|O|(x'_{0,K})$. Let $M$ be an algebraically closed field extending $K$, $\kappa(x)$ and $\kappa(x')$. Then, using \ref{O-basech} and \ref{O-precomp}, we obtain that
\begin{align*}
|O|(x_M)=|O|(x_{0,K,M}) & =|O|(x_{0,K})\times_{X_K}{X_M}  \\ 
& =|O|(x'_{0,K})\times_{X_K}{X_M}=|O|(x'_{0,K,M})=|O|(x'_M).
\end{align*}
Taking the scheme-theoretic image along the first projection and using  the second part of \ref{O-precomp}, we obtain that
$\Traj(x)=\Traj(x')$, as required. 

Finally, to verify that (3) implies (1), take geometric points $\bar{x}, \bar{x}'\in X(L)$ such that $\eta(\bar{x})=\eta(\bar{x}')$. They factor through unique scheme-theoretic points $x, x'\in X_L$ that satisfy $\eta_L(x)=\eta_L(x')$. From the definitions, we verify that the action map $a_{\bar{x}}$ and the trajectory map $t_x$ agree up to isomorphism of their domains, and similarly for $\bar{x}'$, so, assuming (3), we obtain that
$$
|O|(\bar{x})=\overline{\Traj(x)}=\overline{\Traj(x')}=|O|(\bar{x}').
$$
\end{proof}

\begin{corollary}\label{pointw-orbit}
Let $\eta:(X,\delta)\to (Q,0)$ be a morphism of $S$-differential schemes such that for every geometric point $\bar{q}\in Q(L)$, the fibre $X_{\bar{q}}$ is classically simple over $\spec(L)$. Then $\eta$ satisfies the orbit condition. 
\end{corollary}
\begin{proof}
By the assumption and \ref{class-simple-fibrewise}, we obtain the statement (3) of \ref{char-orbit-cond}.
\end{proof}

\begin{fact}[{\cite[8.4.3]{bardavid}}]\label{bardq-catq}
A Bardavid quotient is a categorical quotient. 
\end{fact}

\begin{proposition}\label{geom-categ}
If $\eta:(X,\delta)\to (Q,0)$ is a geometric quotient, it is a Bardavid quotient. Hence, a geometric scheme of leaves is a categorical scheme of leaves,
$$
Q=\pi_0(X).
$$ 
As a ringed space, it is isomorphic to the space of leaves,
$$
Q\simeq X^\delta.
$$
\end{proposition}
\begin{proof}
If $L$ is a field extending $\kappa(x)$ and $\kappa(x')$ such that $|O|(x_L)=|O|(x'_L)$, then, using \ref{O-precomp}, we obtain
$$
\Traj(x)=\pi_1(|O|(x_L))=\pi_1(|O|(x'_L))=\Traj(x'),
$$
hence the orbit condition implies the trajectory condition. The second statement follows from \ref{bardq-catq}, and the isomorphism of ringed spaces is constructed in \cite[8.5.5]{bardavid}.
\end{proof}

\begin{definition}\label{univ-quot}
A categorical/Bardavid/geometric quotient $\eta:(X,\delta)\to (Q,0)$ is \emph{universal}, if it remains such after an arbitrary base change $q:Q'\to Q$. A quotient is \emph{uniform}, if it is stable under any flat base change $q$.
\end{definition}
 
 \begin{proposition}\label{unif-geom-quot}
 A faithfully flat qcqs geometric quotient $\eta:(X,\delta)\to (Q,0)$ is uniform. In particular,  if $Q$ is the spectrum of a field, then $\eta$ is a universal geometric quotient. 
 \end{proposition}
 \begin{proof}
 Since $\eta$ is fpqc, it is universally submersive so the topological condition is universally satisfied. The orbit condition is universal by \ref{O-basech}.
 Let $\eta':X'\to Q'$ be a base change of $\eta$ by a flat morphism $q:Q'\to Q$.  Since $\eta$ satisfies the sheaf condition, we have an exact sequence of quasicoherent modules on $Q$,
$$
0\to \cO_Q\to \eta_*\cO_X\stackrel{\eta_*\delta}{\longrightarrow}\eta_*\cO_X.  
$$
Since $q$ is flat, applying $q^*$ yields an exact sequence. The flat base change formula \cite[02KE]{stacks-project} gives that $q^*\eta_*\cO_X\simeq\eta'_*\cO_{X'}$, so we obtain that $\eta'$ also satisfies the sheaf condition. 
\end{proof}

\subsection{Simple differential schemes}

\begin{definition}\label{def-simplicity}
An $S$-differential scheme $(X,\delta_X)$ is \emph{simple} with respect to the pseudofunctor $\cP:\Sch_{\ov S}\to\Cat$, if its categorical scheme of leaves $\pi_0(X)$ exists and the coequaliser
$$
\begin{tikzcd}[cramped, column sep=normal, ampersand replacement=\&]
{X_1}\ar[yshift=2pt]{r}{d_0} \ar[yshift=-2pt]{r}[swap]{d_1} \&{X_0} \ar{r}{\eta} \& \pi_0(X)
\end{tikzcd}
$$
is universal for $\cP$ in the sense of \ref{univ-ce-desc}, i.e., if $\bbX$ has a $\cP$-universal scheme of connected components. 
\end{definition}

\begin{lemma}\label{CX-P}
If an $S$-differential scheme $(X,\delta_X)$ is simple for $\cP$ with categorical scheme of leaves $\eta_X:X\to \pi_0(X)$, the canonical functor
$$
C_X:\cP(\pi_0(X))\simeq \cP^{\Disc{\pi_0(\bbX)}}\stackrel{\eta_{\bbX}^*}{\longrightarrow}\cP^{\bbX}\simeq \delta\da\cP(X,\delta_X)
$$
is fully faithful, where we wrote $\eta_{\bbX}:\bbX\to \Disc{\pi_0(\bbX)}$ for the associated morphism of precategories inducing the pullback of precategory actions $\eta_{\bbX}^*$ as in \ref{part-adj-pi0}.
\end{lemma}
\begin{proof}
Using \ref{univ-ce-desc}, we have that the arrow in the above diagram
is fully faithful, so the composite is too.
\end{proof}

\begin{remark}
When $\cP$ is a sub-pseudofunctor of the self-indexing of $\Sch_{\ov S}$ as in \ref{diff-self}, we obtain that
$$
C_X(Q\to \pi_0(X))=(X,\delta_X)\times_{C(\pi_0(X))}C(Q),
$$
so $C_X$ agrees with the functor introduced in the classical Galois context \ref{ss:janelidze-gal}. 

In this case, $(X,\delta)$ is simple with respect to $\cP$ if and only if $\eta_X$ is a categorical quotient stable under base change by morphisms from $\cP$. 

For example, \ref{unif-geom-quot} shows that a faithfully flat qcqs geometric quotient is simple with respect to the class of flat scheme morphisms.
\end{remark}

\subsection{Universality of geometric quotients}

\begin{definition}\label{univ-axioms}
Let $\eta:(X,\delta)\to (Q,0)$ be a qcqs morphism of $S$-differential schemes. We will consider the following conditions:
\begin{itemize}
\item[(BC)] for any morphism $g:Q'\to Q$, the pullback 
$$
 \begin{tikzpicture} 
\matrix(m)[matrix of math nodes, row sep=2em, column sep=2em, text height=1.9ex, text depth=0.25ex]
 {
 |(1)|{X'}		& |(2)|{Q'} 	\\
 |(l1)|{X}		& |(l2)|{Q} 	\\
 }; 
\path[->,font=\scriptsize,>=to, thin]
(1) edge node[above]{$\eta'$} (2) edge node[left]{$g'$}  (l1)
(2) edge node[right]{$g$} (l2) 
(l1) edge node[above]{$\eta$}  (l2);
\end{tikzpicture}
$$
satisfies the Beck-Chevalley condition for the structure sheaf, i.e., the natural base change morphism
$$
g^*\eta_*\cO_X{\longrightarrow}\eta'_*{g'}^*\cO_X\simeq \eta'_*\cO_{X'}
$$
is an isomorphism;

\item[(UI)] the morphism 
$$
\eta^\sharp:\cO_Q\to \eta_*\cO_X
$$
is universally injective;

\item[(CF)] the quasicoherent sheaf
$$
\coker(\eta_*\cO_X\stackrel{\eta_*\delta}{\longrightarrow}\eta_*\cO_X)
$$
is flat. 
\end{itemize}
\end{definition}

\begin{lemma}\label{qc-tor-trick}
Consider an exact sequence of quasicoherent modules 
$$
0\to\cF_1\stackrel{\alpha}{\longrightarrow}\cF_2 \stackrel{\beta}{\longrightarrow} \cF_3 \stackrel{\gamma}{\longrightarrow}\cF_4\to0
$$
on a scheme $Y$, such that
\begin{enumerate}
\item $\alpha$ is universally injective, and
\item $\cF_4$ is flat. 
\end{enumerate}
Then, for any morphism $g:Y'\to Y$, the sequence
$$
0\to g^*\cF_1\stackrel{g^*\alpha}{\longrightarrow}g^*\cF_2 \stackrel{g^*\beta}{\longrightarrow} g^*\cF_3 \stackrel{g^*\gamma}{\longrightarrow}g^*\cF_4\to0
$$
is exact.
\end{lemma}
\begin{proof}
Since all the assumptions and the claim are local in the Zariski topology, we may assume that $Y$ and $Y'$ are affine, say $Y=\spec(A)$, $Y'=\spec(A')$. 

Consider the short exact sequences of $A$-modules
$$
0\to F_1\stackrel{\alpha}{\to}F_2\to\im(\beta)\to 0
\ \ \ \text{ and }\ \ \ 
0\to \im(\beta)\to F_3\stackrel{\gamma}{\to}F_4\to 0.
$$
Since $\alpha$ is universally injective, the first sequence is universally exact, so
$$
0\to g^*F_1\stackrel{g^*\alpha}{\longrightarrow}g^*F_2\to g^*\im(\beta)\to 0
$$
is exact, where $g^*F=A'\otimes_A F$. Applying $g^*$ to the second sequence yields a long exact sequence for $\text{\rm Tor}$,
$$
\cdots\to \text{\rm Tor}_1^A(F_4,A')\to g^*\im(\beta)\to g^*F_3\stackrel{g^*\gamma}{\longrightarrow}g^*\cF_4\to0.
$$
Since $F_4$ is a flat $A$-module, we have that $\text{\rm Tor}_1^A(F_4,A')=0$, so we can splice the two short exact sequence into the required one.
\end{proof}

\begin{proposition}\label{universal-sc}
Suppose that a morphism $\eta:(X,\delta)\to (Q,0)$ satisfies the sheaf condition and conditions (BC), (UI), (CF) from \ref{univ-axioms}. Then the sheaf condition for $\eta$ holds universally.
\end{proposition}

\begin{proof}
We need to show that, for an arbitrary $g:Q'\to Q$, the sequence
$$
0\to \cO_{Q'}\stackrel{{\eta'}^\sharp}{\longrightarrow}\eta'_*\cO_{X'}\stackrel{\eta'_*\delta_{X'}}{\longrightarrow}\eta_*\cO_{X'}
$$
is exact, where we used the notation for the pullback from \ref{univ-axioms}. Using (BC), this follows from \ref{qc-tor-trick}.
\end{proof}

\begin{remark}\label{affine-CX-remark}
The affine version of \ref{universal-sc} reads as follows. 

Let $(A,\delta)$ be a differential ring with the ring of constants $k$, and $R$ be a $k$-algebra such that either
\begin{enumerate}
\item $k\to A$ is \emph{universally injective} \cite[Tag~058I]{stacks-project} and $\coker(\delta)$ is a flat $k$-module, or
\item $R$ is flat over $k$.
\end{enumerate}
Then 
$$
\Const((A,\delta)\otimes_{(k,0)}(R,0))\simeq R.
$$ 
This key condition is familiar from classical literature on differential Galois theory. 
\end{remark}

\begin{definition}
We say that a morphism $f:X\to Q$ to a locally Noetherian scheme $Q$ is \emph{cohomologically flat in dimension 0} (CFD0) if it is proper, flat and satisfies condition (BC), \cite[7.8.1]{ega3.2}, \cite{raynaud-pmihes-38}, \cite[Section~2]{gabber-lodh}.
\end{definition}

\begin{lemma}\label{suff-cond-BC}
If a morphism $f:X\to Q$ satisfies either 
\begin{enumerate}
\item $f$ is affine, or
\item $f$ is CFD0,
\end{enumerate}
then $f$ satisfies (BC). 

In both cases, for every $q\in Q$, we have
$$
f_*\cO_X\otimes \kappa(q)\simeq H^0(X_q,\cO_q),
$$
where $X_q=X\times_Q\spec(\kappa(q))$.

Furthermore, in case (2), $f_*\cO_X$ is locally free of finite rank (in particular, flat). 
\end{lemma}

\begin{proof}
In case (1), condition (BC) follows from the the affine base change \cite[02KE]{stacks-project}. In case (2), it holds by definition. 
The stated isomorphism is condition (BC) applied to the base change via the morphism $\spec(\kappa(q))\to Q$. 
The claim about local freeness is in the proof of Proposition~1, Section~2 of \cite{gabber-lodh}, or \cite[Section~7]{ega3.2}.
\end{proof}

\begin{lemma}\label{cond-cfd0}
Suppose $f:X\to Q$ is a proper flat morphism and $Q$ is locally Noetherian. The $f$ is CFD0 if any of the following are satisfied:
\begin{enumerate}
\item $f$ has geometrically reduced fibres;
\item for every $q\in Q$, $H^1(X_q,\cO_q)=0$;
\item $Q$ is reduced and $\dim H^0(X_q,\cO_q)$ is locally constant on $Q$;
\item for every $q\in Q$, $H^0(X_q,\cO_q)$ is a separable extension of $\kappa(q)$. 
\end{enumerate}
\end{lemma}
\begin{proof}
Items (1) and (3) are in Proposition~1, Section~2 of \cite{gabber-lodh}. Item (2) is in Section~5, Chapter~0 of \cite{mumford-git}. Item (4) is \cite[7.8.6]{ega3.2}.
\end{proof}

\begin{remark}\label{rem-cfd0}
The following observations refer to condition (4) of \ref{cond-cfd0}.
\begin{enumerate}
\item By properness of $f$, the extension from the condition is always finite. In characteristic 0, separability is automatic. 
\item When the condition holds, $f_*\cO_X$ is an \'etale $\cO_Q$-algebra. Thus, in the Stein factorisation $X\to Q'\to Q$, the morphism $Q'\to Q$ is finite \'etale and the Stein factorisation commutes with base change. 
\item By \cite[7.8.10]{ega3.2}, the set of points $q\in Q$ at which the condition holds is Zariski open in $Q$. Hence, when $Q$ is of finite type over a field, it is enough to check the condition at the closed points. 
\end{enumerate}
\end{remark}

\begin{lemma}\label{le4}
Suppose that $f:X\to Q$ is proper flat with geometrically reduced connected fibres, and $Q$ is locally Noetherian. Then $f$ is CFD0 and
$$
f_*\cO_X\simeq\cO_Q.
$$
\end{lemma}
\begin{proof}
For an arbitrary $q\in Q$, choose a geometric point $\bar{q}:\spec(\overline{\kappa(q)})\to Q$ factoring through $q$. By flat base change applied to 
$\spec(\overline{\kappa(q)})\to \spec(\kappa(q))$, 
we have that
$$
\dim_{\kappa(q)}H^0(X_q,\cO_q)=\dim_{\overline{\kappa(q)}}(X_{\bar{q}},\cO_{\bar{q}}).
$$
By assumption, the latter is 1, whence by part~(4) of {cond-cfd0}, it follows that $f$ is CFD0 and that $f_*\cO_X$ is locally free of rank 1. Finally, we note that it admits a nowhere vanishing global section $1\in H^0(X,\cO_X)$.
\end{proof}

\begin{proposition}\label{comp-bc}
Suppose that morphisms $f:X\to Y$ and $g:Y\to Z$ satisfy either of the following conditions:
\begin{enumerate}
\item $f_*\cO_X=\cO_Y$ and $g$ satisfies (BC);
\item $f$ satisfies (BC) and $g$ satisfies the Beck-Chevalley condition for every quasicoherent sheaf $\cF$ in place of $\cO_X$. 
\end{enumerate}
Then $g\circ f$ also satisfies (BC). 

Moreover, we have the following. 
\begin{itemize}
\item[(1')] Condition (1) is satisfied when $f$ is proper flat with geometrically reduced and connected fibres. If we further assume that $g$ is CFD0, or $g$ is affine flat, then $(g\circ f)_*\cO_X$ is flat.
\item[(2')] Condition (2) is satisfied it $g$ is affine. If we further assume that $g$ is flat, then $(g\circ f)_*\cO_X$ is flat.
\end{itemize} 
\end{proposition}
\begin{proof}
Statements (1) and (2) are straightforward. The first statement in (1') follows from \ref{le4}. The second claim in (1') follows from that fact that $f_*\cO_X$ is locally free, hence flat, on $Y$ and \ref{suff-cond-BC}. 
The third claim in (1') and (2') follow from the fact that pushforward of a flat quasicoherent sheaf along a flat affine morphism is flat. 
\end{proof}

\begin{proposition}\label{pointw-cond-cf}
Let $\eta:(X,\delta)\to (Q,0)$ be a qcqs morphism of differential schemes with $Q$ locally Noetherian such that $\eta_*\cO_X$ is flat. Suppose $\eta$ satisfies the sheaf condition both in the usual sense that the sequence
$$
0\to \cO_Q\stackrel{}{\longrightarrow}\eta_*\cO_X\stackrel{\eta_*\delta_X}{\longrightarrow}\eta_*\cO_X
$$
is exact, as well as pointwise, i.e., for every $q\in Q$, considered as a morphism $\spec(\kappa(q))\to Q$,  the sequence of quasicoherent sheaves on $\spec(\kappa(q))$
$$
0\to q^*\cO_Q\stackrel{}{\longrightarrow}q^*\eta_*\cO_X\stackrel{q^*\eta_*\delta_X}{\longrightarrow}q^*\eta_*\cO_X
$$
is exact. 

Then the condition (CF) holds, i.e., 
$$
\coker(\eta_*\cO_X\stackrel{\eta_*\delta}{\longrightarrow}\eta_*\cO_X)
$$
is flat.
\end{proposition}

\begin{proof}
Given an affine morphism $f:Y\to Z$ and quasicoherent sheaves $\cF$ on $Y$ and $\cG$ on $Z$, we have the projection formula
$$
f_*(\cF\otimes f^*\cG)\simeq f_*\cF\otimes\cG.
$$
By pushing forward the above exact sequence along the affine morphism $q:\spec(\kappa(q))\to Q$ and applying the projection formula to sheaves $\cO_{\spec(\kappa(q))}$ and $\eta_*\cO_X$, we obtain an exact sequence
\begin{equation}\label{1seq}
0\to \cO_Q\otimes\kappa\stackrel{}{\longrightarrow}\eta_*\cO_X\otimes\kappa\stackrel{}{\longrightarrow}\eta_*\cO_X\otimes\kappa,
\end{equation}
of quasicoherent sheaves on $Q$, where we wrote $\kappa$ for $q_*\cO_{\spec(\kappa(q))}$. 

From now on, we replace $X$ by the local scheme $\spec(\cO_{X,q})$. 

Since $\eta$ satisfies the sheaf condition, we have an exact sequence of quasicoherent modules on $Q$,
$$
0\to \cO_Q\to \eta_*\cO_X\stackrel{\eta_*\delta}{\longrightarrow}\eta_*\cO_X\to C\to 0,
$$
where $C=\coker{\eta_*\delta}$, which we can split into short exact sequences
$$
0\to \cO_Q\to \eta_*\cO_X\to I \to 0,
$$
and
$$
0\to I\to\eta_*\cO_X\to C\to 0,
$$
where $I=\Im(\eta_*\delta)$. 
The long exact sequences for $\text{--}\otimes\kappa$ associated to the above become
\begin{equation}\label{15seq}
\cdots\to \cO_Q\otimes\kappa\to \eta_*\cO_X\otimes\kappa\stackrel{\alpha}{\to}I\otimes\kappa\to0,
\end{equation}
and
\begin{equation}\label{2seq}
0\to\text{\rm Tor}_1(\kappa,C)\to I \otimes\kappa\stackrel{\beta}{\to}\eta_*\cO_X\otimes\kappa\to C\otimes\kappa\to 0.
\end{equation}
In (\ref{2seq}), we 
used the assumption that $\eta_*\cO_X$ is flat, which implies that $\text{\rm Tor}_1(\kappa,\eta_*\cO_X)=0$. 

Splicing the sequences, we obtain the sequence
$$
 \begin{tikzpicture}
[cross line/.style={preaction={draw=white, -,
line width=4pt}}]
\matrix(m)[matrix of math nodes, row sep=1.2em, column sep=1.2em, text height=1.5ex, text depth=0.25ex]
{ |(0)|{\cdots} & |(1)|{\cO_Q\otimes\kappa} &	|(i)|{\eta_*\cO_X\otimes\kappa} 	&			& |(k)|{\eta_*\cO_X\otimes\kappa} & |(3)|{C\otimes\kappa} & |(4)|{0}		\\  
&	&		& |(j)|{I\otimes\kappa} 	&			& &	\\};
\path[->,font=\scriptsize,>=to, thin]
(0) edge (1)
(1) edge (i)
(k) edge (3)
(3) edge (4)
(i) edge node[below,pos=0.4]{$\alpha$} (j) edge node[above,pos=0.5]{} (k) 
(j) edge node[below,pos=0.6]{$\beta$} (k) 
;
\end{tikzpicture}
$$ 
which is exact at the second term by (\ref{1seq}). Using (\ref{15seq}), this happens precisely when  $\beta$ is injective, or, equivalently, by (\ref{2seq}), when  
$$\text{\rm Tor}_1(\kappa, C)=0.$$
It now follows from the Local Criterion of Flatness \cite[00MK]{stacks-project} that $C$ is flat at $q$. Since $q$ was arbitrary, $C$ is flat, as required. 
\end{proof}

\begin{proposition}\label{eq-pointw-sh}
Suppose that a morphism of differential schemes $\eta:(X,\delta)\to (Q,0)$ satisfies (BC) condition. The following are equivalent:
\begin{enumerate}
\item the pointwise sheaf condition from \ref{pointw-cond-cf};
\item for every $q\in Q$, 
$$
\Const(H^0(X_q,\cO_q))=\kappa(q);
$$
\item for every geometric point $\bar{q}\in Q(L)$, 
$$
\Const(H^0(X_{\bar{q}},\cO_{\bar{q}}))=L.
$$
\end{enumerate}
Moreover, if $Q$ is of finite type over a field, it is enough to check either of these conditions at closed points of $Q$. 
\end{proposition}

\begin{proof}
Since $\eta$ satisfies (BC), we have that
$$
q^*\eta_*\cO_X\simeq \eta_{q,*}\cO_{X_q}\simeq \widetilde{H^0(X_q,\cO_q)}.
$$
Hence, the sequence from (1) can be written as
$$
0\to \widetilde{\kappa(q)}\to  \widetilde{H^0(X_q,\cO_q)}\to  \widetilde{H^0(X_q,\cO_q)},
$$
which is equivalent to the sequence of $\kappa(q)$-vector spaces
$$
0\to {\kappa(q)}\to  {H^0(X_q,\cO_q)}\to {H^0(X_q,\cO_q)}.
$$
This proves the equivalence of statements (1) and (2). 

The equivalence between statements (2) and (3) is a consequence of the flat base change formula along the flat morphism $\spec(L)\to \spec(\kappa(q))$, where $q\in Q$ is the image of the geometric point $\bar{q}\in Q(L)$, and we have
$$
H^0(X_{\bar{q}},\cO_{X_{\bar{q}}})\simeq H^0(X_q,\cO_{X_q})\otimes_{\kappa(q)}L.
$$
The last statement follows from \ref{rem-cfd0}.
\end{proof}

\begin{lemma}\label{pure-ui}
If $\eta:X\to Q$ is a pure scheme morphism, then the condition (UI) holds. In particular, a faithfully flat morphism satisfies (UI).  
\end{lemma}
\begin{proof}
Let $\eta':X'\to Q'$ be the base change of $\eta$ by a scheme morphism $q:Q'\to Q$.
Purity tell us that the composite $\cO_{Q'}\simeq q^*\cO_Q \to q^*\eta_*\cO_X\to \eta'_*\cO_{X'}$ is injective, where the last morphism is the base change morphism. It follows that the first morphism is also injective. It is well-known that a faithfully flat morphism is pure. 
\end{proof}

\begin{proposition}\label{univ-sh-c}
Let $\eta:(X,\delta)\to (Q,0)$ be a qcqs faithfully flat morphism satisfying the sheaf condition and the pointwise sheaf condition, i.e., for every $q\in Q$, we have
$$
\Const(H^0(X_q,\cO_q))=\kappa(q).
$$
Equivalently, we can assume an analogous condition for every geometric point of $Q$. When $Q$ is of finite type over a field, it suffices to assume the pointwise sheaf condition for closed points $q$. 

Suppose that  any of the following conditions is satisfied:
\begin{enumerate}
\item $\eta$ is affine;
\item $\eta$ is CFD0 (for sufficient conditions see \ref{cond-cfd0});
\item $\eta$ is a composite as in \ref{comp-bc} (1') or (2') with $g$ affine flat.
\end{enumerate}
Then the sheaf condition holds universally for $\eta$. 
\end{proposition}

\begin{proof}
By \ref{universal-sc}, we need to check conditions (BC), (UI) and (CF). 

Condition (BC) in cases (1) and (2) follow from \ref{suff-cond-BC}. In case (1), $\eta_*\cO_X$ is flat since $\eta$ is affine flat. In case (2), $\eta_*\cO_X$ is flat by \ref{suff-cond-BC}. Case (3) follows from \ref{comp-bc}, which also gives us flatness of $\eta_*\cO_X$.

Condition (UI) follows from \ref{pure-ui}. 

Since (BC) holds, \ref{eq-pointw-sh} gives the pointwise condition from \ref{pointw-cond-cf}. Moreover, 
since $\eta_*\cO_X$ is flat, \ref{pointw-cond-cf} yields (CF).
\end{proof}

\begin{theorem}\label{univ-geom-quot}
Let $\eta:(X,\delta)\to (Q,0)$ be a qcqs faithfully flat morphism satisfying the sheaf condition such that for every geometric point $\bar{q}$ of $Q$, the fibre $X_{\bar{q}}$ is classically simple, and $\eta$ satisfies any of the conditions (1)--(3) from \ref{univ-sh-c}. Then $\eta$ is a universal geometric quotient. 
\end{theorem}

\begin{proof}
Since $\eta$ is fpqc, it is universally submersive and surjective, so the topological condition is stable under arbitrary base change. 

By \ref{pointw-orbit}, we obtain that $\eta$ satisfies the orbit condition, which is stable under arbitrary base change by \ref{O-basech}.

Classical simplicity of geometric fibres implies the pointwise sheaf condition, so \ref{univ-sh-c} yields the universality of the sheaf condition.  
\end{proof}

\begin{lemma}\label{pushfwd-const-sh}
If a differential scheme $(X,\delta_X)$ is simple with respect to  Zariski open immersions, with categorical scheme of leaves given as coequaliser from \ref{def-simplicity}, then it satisfies the sheaf condition for quotients,
$$
\eta_*\Const(\cO_X,\delta_X)=\cO_{\pi_0(X)}.
$$
\end{lemma}
\begin{proof}
From the coequaliser diagram of schemes, we obtain a diagram of structure sheaves
$$
\begin{tikzcd}[cramped, column sep=normal, ampersand replacement=\&]
{\eta_*\cO_{X_1}} \&{\eta_*\cO_{X_0}} \ar[yshift=2pt]{l}[swap]{\eta_*\delta_0} \ar[yshift=-2pt]{l}{\eta_*\delta_1}  \& \cO_{\pi_0(X)} \ar{l}
\end{tikzcd}
$$
which yields a unique morphism 
$$\cO_{\pi_0(X)}\to \Eq(\eta_*\delta_0,\eta_*\delta_1)\simeq \eta_*\Eq(\delta_0,\delta_1)=\eta_*\Const(\cO_X,\delta_X).$$
By assumption, for every Zariski open $V\hookrightarrow\pi_0(X)$, the diagram
$$
\begin{tikzcd}[cramped, column sep=normal, ampersand replacement=\&]
{U_1}\ar[yshift=2pt]{r} \ar[yshift=-2pt]{r} \&{U} \ar{r}{} \& V
\end{tikzcd}
$$
with $U=\eta^{-1}V$ and $U_1$ the pullback of $V$ to $X_1$, remains a coequaliser. Substituting $V=\spec(A)$ in the above equaliser of sheaves, we obtain a ring homomorphism
$$
A\to \Eq(\cO_U(U)\doublerightarrow{}{}\cO_{U_1}(U)).
$$
The universal property of $V$ being coequaliser, applied to varying affine schemes, yields that in fact $$\cO_{\pi_0(X)}(V)=A\simeq \Eq(\cO_U(U)\doublerightarrow{}{}\cO_{U_1}(U)=\eta_*\Const(\cO_X)(V),$$
for an arbitrary affine open $V$,  whence we obtain the desired conclusion.  
\end{proof}

The following proposition is a universal version of the `going forward' approach to effectively proving categoricity of quotients that appears in \cite[3.2.4]{ayoub}.

\begin{proposition}\label{our-proof}

Let $(X,\delta_X)$ be an $S$-differential scheme with categorical scheme of leaves $S$ such that:
\begin{enumerate}
\item the canonical morphism
$$
\eta:X\to \pi_0(X)=S
$$
is fpqc and universally open (for example, fppf);
\item for every Zariski open $U$ in $X$, we have
$$
\pi_0(U)=\eta(U);
$$
\item the module $\coker(\eta_*\delta_X)$ is flat over $S$.
\end{enumerate}
Then $(X,\delta_X)$ is simple with respect to $S$-scheme morphisms, i.e., $\eta$ is a universal categorical quotient. 
\end{proposition}

\begin{proof}
Writing $\bbX$ for the precategory in $\Sch_{\ov S}$ associated to $(X,\delta_X)$, our first assumption means that the diagram 
$$
\begin{tikzcd}[cramped, column sep=normal, ampersand replacement=\&]
{X_1}\ar[yshift=2pt]{r}{d_0} \ar[yshift=-2pt]{r}[swap]{d_1} \&{X_0} \ar{r}{} \& S
\end{tikzcd}
$$
is a coequaliser. We need to show that it is universal for $S$-scheme morphisms, i.e., that the base change 
$$
 \begin{tikzpicture} 
\matrix(m)[matrix of math nodes, row sep=2em, column sep=2em, text height=1.5ex, text depth=0.25ex]
 {
|(u2)|{X_1\times_S Q}  & [1em] |(u1)|{X_0\times_S Q}		&[1em] |(u0)|{Q} \\
|(2)|{X_1}  & [1em] |(1)|{X_0}		&[1em] |(0)|{S} \\
 }; 
\path[->,font=\scriptsize,>=to, thin]
([yshift=.2em]2.east) edge node[above=-2pt]{$d_0$} ([yshift=.2em]1.west) 
([yshift=-.2em]2.east) edge node[below=-2pt]{$d_1$} ([yshift=-.2em]1.west) 
(1) edge node[above=-2pt]{$\eta$} (0) 
([yshift=.2em]u2.east) edge node[above=-2pt]{} ([yshift=.2em]u1.west) 
([yshift=-.2em]u2.east) edge node[below=-2pt]{} ([yshift=-.2em]u1.west) 
(u1) edge node[above=-2pt]{} (u0) 
(u2) edge (2)
(u1) edge (1)
(u0) edge node[right]{$q$} (0)
;
\end{tikzpicture}
$$
along any $S$-scheme morphism $Q\to S$ remains a coequaliser. Indeed, let $f:X_0\times_S Q\to Y$ be a morphism of $S$-schemes that coequalises the two arrows from $X_1\times_SQ$. Every point $p\in X_0\times_S Q$ has an open affine neighbourhood $U_p\times_{S_p}Q_p$, where $U_p$, $S_p$ and $Q_p$ are open affine with $\eta(U_p)\subseteq S_p$ and $q(Q_p)\subseteq S_p$ such that  $f_p=f\restriction_{U_p\times_{S_p}Q_p}$ factors through an affine open subset of $Y$. 
Since $\eta$ is open, we may assume that $\eta(U_p)= S_p$. 

As a base-change of a surjective morphism $X_0\to S$, the morphism $X_0\times_SQ\to Q$ is surjective, so, since  
$U_p\times_{S_p}Q_p$ cover $X_0\times_SQ$, we obtain that $Q_p$ cover $Q$. The morphisms $d_0$ and $d_1$ are identites on the underlying topological spaces of $X_1$ and $X_0$, and let us write  $U_{p,1}=U_p\times_S S_1$ for the preimages of $U_p$ in $X_1$.

By assumption,
$$
\begin{tikzcd}[cramped, column sep=normal, ampersand replacement=\&]
{U_{p,1}}\ar[yshift=2pt]{r}{} \ar[yshift=-2pt]{r}[swap]{} \&{U_p} \ar{r}{} \& S_p
\end{tikzcd}
$$
is a coequaliser, and by \ref{affine-CX-remark}, we obtain that
$$
\begin{tikzcd}[cramped, column sep=normal, ampersand replacement=\&]
{U_{p,1}\times_SQ_i}\ar[yshift=2pt]{r}{} \ar[yshift=-2pt]{r}[swap]{} \&{U_p\times_{S_p} Q_p} \ar{r}{\eta_p} \& Q_p
\end{tikzcd}
$$
is a coequaliser in $\Aff_{\ov S_p}$. 

Since $f_p$ factors through an affine open in $Y$, there exists a unique $h_p:Q_p\to Y$ such that 
$$
f_p=h_p\circ \eta_p.
$$
Every point in $Q_p\cap Q_{p'}$ has an affine open neighbourhood $W$ which is standard open in $Q_p$ and $Q_{p'}$, so by the same affine argument applied to $U_p\times_S W$ and $U_{p'}\times_S W$, we obtain that $h_p\restriction_W=h_{p'}\restriction_W$. 

Hence $$h_p\restriction_{Q_{p}\cap Q_{p'}}=h_{p'}\restriction_{Q_p\cap Q_{p'}},$$
so the $h_p$ can be glued into a morphism 
$$
h:Q\to Y
$$
verifying the universal property of coequaliser for $Q$. 
\end{proof}

\begin{proposition}\label{our-proof-cor}
If $(X,\delta)$ is a classically simple $S$-differential scheme, then it is simple with respect to $S$-scheme morphisms. 
\end{proposition}
\begin{proof}
By \ref{bardavid-simple}, for any open subscheme $U$ of $X$,  $\pi_0(U)=S$, so we may apply \ref{our-proof}.

Alternatively, it is straightforward to check that $\eta:X\to S$ is a geometric quotient, so universality follows from \ref{unif-geom-quot} or \ref{univ-geom-quot}.
\end{proof}
This proposition shows categorical simplicity both of classical Picard-Vessiot rings, and of torsors associated with strongly normal extensions.
\subsection{Polarised differential scheme morphisms}

\begin{definition}\label{def-polarised-qproj}
The category of \emph{polarised quasi-projective morphisms} has objects 
$$
(P,\cL_P)\stackrel{p}{\to}U,
$$
consisting of a scheme morphism $p:P\to U$ and a $p$-relatively ample invertible $\cO_P$-module $\cL_P$. 

A morphism $(f,\alpha)$ between polarised quasi-projective morphisms $(P,\cL_P)\stackrel{p}{\to}U$ and $(Q,\cL_Q)\stackrel{q}{\to}V$ is a commutative diagram
$$
 \begin{tikzpicture} 
\matrix(m)[matrix of math nodes, row sep=2em, column sep=2em, text height=1.9ex, text depth=0.25ex]
 {
 |(1)|{(P,\cL_P)}		& |(2)|{(Q,\cL_Q)} 	\\
 |(l1)|{U}		& |(l2)|{V} 	\\
 }; 
\path[->,font=\scriptsize,>=to, thin]
(1) edge node[above]{$(f,\alpha)$} (2) edge node[left]{$p$}  (l1)
(2) edge node[right]{$q$} (l2) 
(l1) edge node[above]{}  (l2);
\end{tikzpicture}
$$
consisting of a scheme morphism $f:P\to Q$ that makes the underlying scheme diagram commutative, together with an $\cO_P$-module \emph{isomorphism} $\alpha:f^*\cL_Q\to\cL_P$.

If $(g,\beta)$ is another morphism from $(Q,\cL_Q)\stackrel{q}{\to}V$ to $(R,\cL_R)\stackrel{r}{\to}W$, the composite is computed as
$$
(g,\beta)\circ(f,\alpha)=(g\circ f, \alpha\circ f^*\beta\circ\text{coherence}),
$$ 
where, more precisely, the $\cO_P$-module isomorphism is given as the composite
$$
(g\circ f)^*\cL_R\simeq f^*g^*\cL_R\stackrel{f^*\beta}{\to}f^*\cL_Q\stackrel{\alpha}{\to}\cL_P.
$$
\end{definition}

\begin{notation}\label{nota-polarised}
The category of polarised quasi-projective morphisms has a natural codomain fibration over the category of schemes.
In this subsection, we write
 $$\cP:\Sch_{\ov S}^\op\to\Cat$$ 
 for the associated pseudofunctor, i.e., $\cP(V)$ is the category consisting of pairs
$$ 
(Q\stackrel{q}{\to}V,\cL), 
$$
where $q$ is a morphism of finite type and $\cL$ is an invertible $q$-ample $\cO_Q$-sheaf.

We will write $$\cS=\mathrm{Self}(\Sch_{\ov S})$$ for the self-indexing of the category of $S$-schemes over itself, and 
$$
U:\cP\to \cS
$$
for the natural forgetful functor. 
\end{notation}

\begin{definition}\label{def-pol-dif}
The category of \emph{polarised quasi-projective differential scheme morphisms} is
the fibered category 
$$
\delta\da\cP
$$
obtained by construction \ref{delta-P}.
\end{definition}

\begin{lemma}\label{descr-poldif}
For an $S$-differential scheme $(X,\delta_X)$, the category of polarised quasi-projective differential scheme morphisms with codomain $(X,\delta)$ is equivalent to the category of pairs $$((P,\delta_P)\stackrel{p}{\to} (X,\delta_X),(\cL_P,\delta_{\cL_P})),$$
where $p:(P,\delta_P)\to (X,\delta_X)$ is a morphism of $S$-differential schemes, and $(\cL_P,\delta_{\cL_P})$ is an invertible $(\cO_P,\delta_P)$-module with $\cL_P$ relatively ample with respect to the underlying scheme morphism $P\to X$. 
\end{lemma}

\begin{proof}
An object $P\in\delta\da\cP(X,\delta)=\cP^\bbX$ gives rise to a diagram 
$$
 \begin{tikzpicture}
 [cross line/.style={preaction={draw=white, -,line width=4pt}}, proj/.style={dotted,-}]
 \def\triple#1#2{
 ([yshift=.4em]#1.east) edge 
 ([yshift=.4em]#2.west) 
(#1) edge 
(#2)
([yshift=-.4em]#1.east) edge 
([yshift=-.4em]#2.west) }
\def\double#1#2{
([yshift=.4em]#1.east) edge 
([yshift=.4em]#2.west) 
(#2)  edge 
(#1) 
([yshift=-.4em]#1.east) edge 
([yshift=-.4em]#2.west) }
\def\vtriple#1#2{
 ([xshift=.4em]#1.south) edge 
 ([xshift=.4em]#2.north) 
(#1) edge 
(#2)
([xshift=-.4em]#1.south) edge 
([xshift=-.4em]#2.north) }
\def\vdouble#1#2{
([xshift=.4em]#1.south) edge
([xshift=.4em]#2.north) 
(#2)  edge 
(#1) 
([xshift=-.4em]#1.south) edge[cross line] 
([xshift=-.4em]#2.north) }
\matrix(m)[matrix of math nodes, row sep=2em, column sep=1em,  
text height=1.2ex, text depth=0.25ex]
{
|(P22)|{(P_2,\cL_{P_2})} 	&[1em] |(P12)|{(P_1,\cL_{P_1})}  	&[1em]	 |(P02)|{(P_0,\cL_{P_0})}		\\[0em]
|(P21)|{X_2} 	&[1em] |(P11)|{X_1}  	&[1em]	 |(P01)|{X_0}		\\[0em]
|(P20)|{S_2} 	&[1em] |(P10)|{S_1}  	&[1em]	 |(P00)|{S_0}		\\};
\path[->,font=\scriptsize,>=to, thin,inner sep=1pt]
\triple{P22}{P12}
\double{P12}{P02}

\triple{P21}{P11}
\double{P11}{P01}

\triple{P20}{P10}
\double{P10}{P00}

(P22) edge (P21)
(P21) edge (P20)

(P12) edge (P11)
(P11) edge (P10)

(P02) edge (P01)
(P01) edge (P00)

;
\end{tikzpicture}
$$ 
where the bottom level is the precategory $\bbD(S)$ as in \ref{diff-sch-precats}, and all the arrows in the top and middle levels are cartesian. Hence, forgetting the polarisations, we see that $P$ is also an $S$-differential scheme, whose differential structure is given by an $S_1$-automorphism $P_1\stackrel{f}{\to} P_1$, which is identity on the underlying space, and it is given by an automorphism $\varphi:\cO_{P_0}[\epsilon]/(\epsilon^2)\to\cO_{P_0}[\epsilon]/(\epsilon^2)$ associated to an $S$-derivation $\delta_P$ on $\cO_P$ as in \ref{diff-sch-precats}, and the action on invertible sheaves is given by an isomorphism $(f,\alpha):(p_1,d_{P,0}^*\cL_0)\to (p_1,d_{P,0}^*\cL_0)$, i.e., by an isomorphism
\newcommand\leftidx[3]{%
  {\vphantom{#2}}#1#2#3%
}
$$
\alpha:\cO_{P_0}[\epsilon]/(\epsilon^2)~_\varphi\!{\otimes}_{\cO_{P_0}[\epsilon]/(\epsilon^2)}\cL_0[\epsilon]/(\epsilon^2)\to\cL_0[\epsilon]/(\epsilon^2).
$$
The identity section requirement yields that $$\alpha=\id+\epsilon \delta_\alpha\circ\eta_{L},$$ where $\eta_L: \cL_0[\epsilon]/(\epsilon^2)\to \cL_0$ is the augmentation morphism, and $\delta_\alpha:\cL_0\to \cL_0$ makes $\cL_0$ into an $(\cO_P,\delta_P)$-module.
\end{proof}

\begin{lemma}\label{coeq-desc-inv-sh}
Let $(X,\delta_X)$ be a differential scheme which is simple with respect to Zariski open immersions, with the categorical scheme of leaves given by the coequaliser
$$
\begin{tikzcd}[cramped, column sep=normal, ampersand replacement=\&]
{X_1}\ar[yshift=2pt]{r}{d_0} \ar[yshift=-2pt]{r}[swap]{d_1} \&{X_0} \ar{r}{\eta} \& \pi_0(X).
\end{tikzcd}
$$
Consider invertible sheaves  $\cL$ and $\cL'$ on $\pi_0(X)$ and  an $\cO_{X_0}$-module isomorphism 
$$
\alpha:\eta^*\cL\to \eta^*\cL'
$$
making the diagram
$$
 \begin{tikzpicture}
[cross line/.style={preaction={draw=white, -,
line width=4pt}}]
\matrix(m)[matrix of math nodes, row sep=1.8em, column sep=1.8em, text height=1.5ex, text depth=0.25ex]
{	|(0i)|{d_0^*\eta^*\cL}		& |(1i)| {d_0^*\eta^*\cL'}	\\   [.02em]
	|(0j)|{d_1^*\eta^*\cL}		& |(1j)| {d_1^*\eta^*\cL'}	\\};
\path[->,font=\scriptsize,>=to, thin]
(0i) edge node[above,pos=0.5]{$d_0^*\alpha$} (1i) edge node[below=-2pt,pos=0.5,sloped]{$\sim$} (0j) 
(0j) edge node[above,pos=0.5]{$d_1^*\alpha$} (1j) 
(1i) edge node[above=-2pt,pos=0.5,sloped]{$\sim$} (1j)
;
\end{tikzpicture}
$$ 
commutative. 

Then, there exists a unique isomorphism $\sigma:\cL\to\cL'$ such that $$\alpha=\eta^*\sigma.$$
\end{lemma}
\begin{proof}
Let $V_i$ constitute an open cover of $\pi_0(X)$ such that, for every $i$, $\cL\restriction_{V_i}$ and $\cL'\restriction_{V_i}$ are free of rank 1. Then $U_i=\eta^{-1}(V_i)$ cover $X_0$ and $\eta^*\cL\restriction_{U_i}\simeq \cO_{U_i}e_i$, $\eta^*\cL'\restriction_{U_i}\simeq \cO_{U_i}e'_i$ are free rank 1, for some $e_i$ and $e_i'$, so $\alpha\restriction_{U_i}$ is given as multiplication by some $a_i\in \cO_X(U_i)^{\times}$. 

By identifying 
$$
d_0^*\eta^*\cL\restriction_{U_i}=\cO_{U_i}[\epsilon]/(\epsilon^2)\otimes_{\cO_{U_i}}\cO_{U_i}e_i \simeq \cO_{U_i}(1\otimes e_i)\oplus\cO_{U_i}(\epsilon\otimes e_i),
$$
and similarly for $d_0^*\eta^*\cL'\restriction_{U_i}$, the matrix of $d_0^*\alpha\restriction_{U_i}$ in the pair of bases
$((1\otimes e_i),(\epsilon\otimes e_i))$, $((1\otimes e'_i),(\epsilon\otimes e'_i))$ becomes
$$
\begin{pmatrix}
a_i & 0\\
0 & a_i
\end{pmatrix}.
$$
In the same pair of bases, the matrix of $d_1^*\alpha\restriction_{U_i}$ is
$$
\begin{pmatrix}
a_i & 0\\
\delta_X(a_i) & a_i
\end{pmatrix}.
$$
By the assumption that $d_0^*\alpha$ and $d_1^*\alpha$ agree up to coherences, we conclude that $\delta_X(a_i)=0$, i.e., by \ref{pushfwd-const-sh}, 
$$
a_i\in\Const(\cO_X)(U_i)=\eta_*\Const(\cO_X)(V_i)\simeq\cO_{\pi_0(X)}(V_i),
$$
whence it gives an isomorphism $\cL\restriction_{V_i}\to \cL'\restriction_{V_i}$. 

By the same argument, these isomorphisms agree on the intersections $V_i\cap V_j$ and hence glue uniquely to an isomorphism $\cL\to \cL'$ with the desired property.
\end{proof}

\begin{proposition}\label{reflect-quasiproj-coeq}
The forgetful functor 
$$
\delta\da\cP\to\delta\da\cS
$$
reflects coequalisers associated with connected components of differential schemes which are simple with respect to Zariski open immersions. 
\end{proposition}

\begin{proof}
Let $(X,\delta_X)\in\delta\da\cS$ be simple for Zariski open immersions, and let $(p_2,p_1,p_0)\in \delta\da\cP(X,\delta_X)$ be such that, in the diagram
$$
 \begin{tikzpicture} 
\matrix(m)[matrix of math nodes, row sep=2em, column sep=2em, text height=1.5ex, text depth=0.25ex]
 {
|(u2)|{(P_1,\cL_{P_1})}  & [1em] |(u1)|{(P_0,\cL_{P_0})}		&[1em] |(u0)|{(Q,\cL_Q)} \\
|(2)|{X_1}  & [1em] |(1)|{X_0}		&[1em] |(0)|{\pi_0(X)} \\
 }; 
\path[->,font=\scriptsize,>=to, thin]
([yshift=.2em]2.east) edge node[above=-2pt]{} ([yshift=.2em]1.west) 
([yshift=-.2em]2.east) edge node[below=-2pt]{} ([yshift=-.2em]1.west) 
(1) edge node[above=-2pt]{} (0) 
([yshift=.2em]u2.east) edge node[above=-2pt]{$(d_0,\delta_0)$} ([yshift=.2em]u1.west) 
([yshift=-.2em]u2.east) edge node[below=-2pt]{$(d_1,\delta_1)$} ([yshift=-.2em]u1.west) 
(u1) edge node[above=-2pt]{$(r,\rho)$} (u0) 
(u2) edge node[left]{$p_1$} (2)
(u1) edge node[left]{$p_0$} (1)
(u0) edge (0)
;
\end{tikzpicture}
$$
we have that $(r,\rho)$ coequalises $(d_0,\delta_0)$ and $(d_1,\delta_1)$, and the underlying diagram of scheme morphisms is a coequaliser. 

We need to show that the diagram is a coequaliser of polarised quasi-projective morphisms. Let $(h,\chi):(P_0,\cL_{P_0})\to (T,\cL_T)$ be an arbitrary morphism that coequalises $(d_0,\delta_0)$ and $(d_1,\delta_1)$. Since the underlying diagram of scheme morphisms is a coequaliser, there is a unique morphism $s:Q\to T$ such that $s\circ r=h$. 

Let $\cL=s^*\cL_T$, $\cL'=\cL_Q$, and consider the isomorphism $\alpha:r^*\cL\to r^*\cL'$ given as the composite
$$
r^*s^*\cL_T\simeq h^*\cL_T\stackrel{\chi}{\to}\cL_{P_0}\stackrel{\rho^{-1}}{\to}r^*\cL_Q.
$$
Conditions $(h,\chi)\circ(d_0,\delta_1)=(h,\chi)\circ(d_0,\delta_1)$ and $(r,\rho)\circ(d_0,\delta_1)=(r,\rho)\circ(d_0,\delta_1)$ show that the diagram
$$
 \begin{tikzpicture} 
\matrix(m)[matrix of math nodes, row sep=1.5em, column sep=2em, text height=1.5ex, text depth=0.25ex]
 {
|(u2)|{d_0^*h^*\cL_T}  & [1em] |(u1)|{d_0^*\cL_{P_0}}		&[1em] |(u0)|{d_0^*r^*\cL_Q} \\
				    &  |(mid)|{\cL_{P_1}} 				&					\\
|(2)|{d_1^*h^*\cL_T}  & [1em] |(1)|{d_0^*\cL_{P_0}}		&[1em] |(0)|{d_1^*r^*\cL_Q} \\
 }; 
\path[->,font=\scriptsize,>=to, thin]
(2) edge node[above=-2pt]{$d_1^*\chi$} (1) 
(1) edge node[above=-2pt]{$d_1^*\rho^{-1}$} (0) 
(u2) edge node[above=-2pt]{$d_0^*\chi$} (u1) 
(u1) edge node[above=-2pt]{$d_0^*\rho^{-1}$} (u0) 
(u2) edge node[below=-2pt,pos=0.5,sloped]{$\sim$}(2)
(u1) edge node[left=-2pt]{$\delta_0$} (mid)
(1) edge node[left=-2pt]{$\delta_1$} (mid) 
(u0) edge node[above=-2pt,pos=0.5,sloped]{$\sim$}(0)
;
\end{tikzpicture}
$$
commutes, so $d_0^*\alpha$ and $d_1^*\alpha$ agree up to coherences. By \ref{coeq-desc-inv-sh}, we obtain an isomorphism 
$$
\sigma:s^*\cL_T=\cL\to\cL'=\cL_Q
$$
such that $\alpha=r^*\sigma$, and it follows that
$$
(s,\sigma)\circ(r,\rho)=(h,\chi),
$$
so $(Q,\cL_Q)$ has the universal property of a coequaliser, as required. 
\end{proof}

\subsection{Differential descent}

\begin{proposition}\label{diffl-desc}
Let $\cP:\Sch_{\ov S}^\op\to \Cat$ be a pseudofunctor, and let $$f:(X,\delta_X)\to (Y,\delta_Y)$$ be a morphism of $S$-differential schemes such that
\begin{enumerate}
\item the underlying morphism $f_0:X\to Y$ is a morphism of effective descent for $\cP$;
\item $f_1=d_0^*f_0: X_1\to Y_1$ is a descent morphism for $\cP$, where $d_0:S_1\to S_0=S$ is the source morphism of precategory $\bbD(S)$. 
\end{enumerate} 
Then $f$ is a morphism of effective descent for $\delta\da\cP$. 
\end{proposition}
\begin{proof}
We consider the morphism of $S$-differential schemes as a morphism $f:\bbX\to \mathbb{Y}$ of $\bbD(S)$-actions, and apply \ref{prop:desc-precat}, noting that the cocycle condition is superfluous for differential schemes so we may omit the condition on $f_2$. 
\end{proof}

\begin{corollary}\label{diffl-desc-qp}
Let $f:(X,\delta_X)\to (Y,\delta_Y)$ be a morphism of differential schemes with 
codomain the spectrum of a differential field and $X$ quasi-compact.  Then $f$ is a morphism of effective descent for the class of differential quasi-projective morphisms. 
\end{corollary}
\begin{proof}
The underlying morphism $f_0:X\to Y$ is (trivially) fpqc because the target is the spectrum of a field, and \ref{efdesc-qp} show that it is effective descent for quasi-projective scheme morphisms. Its base change $f_1=d_0^*f$ is again fpqc, hence a morphism of descent for all scheme morphisms by \cite[VIII, 5.2]{sga1}.
\end{proof}

\begin{corollary}\label{difl-qproj-desc}
Let $f:(X,\delta_X)\to (Y,\delta_Y)$ be a morphism of differential schemes whose underlying scheme morphism is fpqc. 
Then $f$ is a morphism of effective descent for the class of differential polarised quasi-projective morphisms. 
\end{corollary}
\begin{proof}
The underlying scheme morphism $f_0:X\to Y$ is fpqc, so of effective descent for polarised quasi-projective morphisms by \cite[VIII, 7.8]{sga1}, and the same holds for $f_1$ as the base change of $f_0$. 
\end{proof}

\section{Affine Picard-Vessiot theory}\label{s:aff-PV}

\subsection{Picard-Vessiot Galois theory for differential field extensions}\label{pv-classical}

In this section, we follow the Hopf-theoretic approach to Picard-Vessiot theory explained in \cite{amano-masuoka-takeuchi}. 

\begin{definition}
An extension $L/K$ of differential fields is \emph{Picard-Vessiot}, if
\begin{enumerate}
\item the extension $L$ contains no new constants, i.e., $L_0=K_0$, and we write $k$ for the common field of constants;
\item there exists a differential $K$-subalgebra $A$ of $L$ such that $\text{\rm Frac}(A)=L$ and 
$$
H=(A\otimes_KA)_0
$$
generates the left $A$-module $A\otimes_KA$ in the sense that
$$
(A\otimes_KK)H=A\otimes_KA.
$$
\end{enumerate}
Such an $A$ is called a \emph{Picard-Vessiot ring} for the extension $L/K$.
\end{definition}

\begin{fact}\label{fact-pvring}
In the situation from the above definition, we have the following.
\begin{enumerate}
\item $A$ is unique;
\item $H$ is a Hopf algebra;
\item there is a comodule structure $\theta:A\to A\otimes_kH$ such that 
$$
A^{\text{\rm co}H}=\{a\in A: \theta(a)=a\otimes 1\}= K,
$$
and
\begin{align}\label{selfsplitting}
A\otimes_KA\simeq A\otimes_{(k,0)}{(H,0)}.\tag{$\dagger$}
\end{align}
\item The linear algebraic group $G=\Gal^\text{\rm PV}(L/K)=\spec(H)$ over $k$ is called the \emph{Picard-Vessiot Galois group} of $L/K$ and we have that
$$
G(k)\simeq \Aut_{\DRng}(L/K).
$$
\end{enumerate}
\end{fact}

\begin{fact}[Classical Picard-Vessiot Galois correspondence]\label{class-pv-corr}
Let $L/K$ be a Picard-Vessiot extension.

There is a one-to-one correspondence between intermediate differential field extensions and closed subgroups of the linear algebraic group $\Gal^\text{\rm PV}(L/K)$ given by
$$
M\mapsto {\Gal}^\text{\rm PV}(L/M).
$$
Moreover, an intermediate field $M$ in $L/K$ is Picard-Vessiot over $K$ if and only if ${\Gal}^\text{\rm PV}(L/M)$ is normal in ${\Gal}^\text{\rm PV}(L/K)$. In this case, we have that
$$
{\Gal}^\text{\rm PV}(M/K)\simeq {\Gal}^\text{\rm PV}(L/K)/ {\Gal}^\text{\rm PV}(L/M).
$$
\end{fact}

\subsection{Janelidze's categorical framework for Picard-Vessiot theory}\label{pv-janelidze}

\begin{definition}\label{affine-pv}
In \cite{janelidze-pv}, Janelidze makes the following choices for the classical categorical setup as in \ref{ss:janelidze-gal}.

\begin{enumerate}
\item $\cA=\DAff=\DRng^\op$, the category of affine differential schemes;
\item $\cX=\Aff=\Rng^\op$, the category of affine schemes;
\item $S=\Const^\op$ is the functor of constants, i.e., $S(\spec(A,\delta))=\spec(\Const(A,\delta))$, often written as $S(X)=X_0$;
\item $C(Spec(R))=Spec(R,0)$ transforms a ring into a differential ring with a trivial derivation $0$. 
\end{enumerate}

\end{definition}

\begin{remark}
\begin{enumerate}
\item Given $X\in \cA$, the functor $C_X:\cX_{\ov X_0}\to \cA_{\ov X}$ is given by
$$
C_X(Q\to X_0)=X\times_{C(X_0)}C(Q)=(X,\delta)\times_{(X_0,0)}(Q,0).
$$
\item An object $P\stackrel{p}{\to} Y$ in $\cA_{\ov Y}$ is split by a morphism $X\stackrel{f}{\to}Y$ if the natural morphism 
$f^*p\to C_XS_X(f^*p)$ is an isomorphism, i.e., if
$$
X\times_YP\simeq X\times_{(X_0,0)}((X\times_YP)_0,0).
$$
\end{enumerate}
\end{remark}

\begin{theorem}[{\cite{janelidze-pv}}]\label{janelidze-affine}
Let $A$ be the Picard-Vessiot ring for a differential field extension $L/K$, and let $$f=\spec(K\to A):X=\spec(A)\to Y=\spec(K)$$ be the associated morphism in $\cA$. 

Then $f$ is a morphism of Galois descent, the categorical Galois groupoid agrees with the Picard-Vessiot Galois group,
$$
G=\Gal[f]=\textstyle{\Gal^{\mathop{\rm PV}}}(L/K),
$$
and $X$ is a $G$-torsor over $Y$ in the sense that
$$
X\times_YX \simeq X\times_{(X_0,0)}(G,0).
$$
There is an equivalence of categories 
$$
\Split_Y[f]\simeq [G,\cX]
$$
between the category of objects $P\stackrel{p}{\to} Y$ in $\cA_{\ov Y}$ split by $f$, in the sense that $$X\times_YP=f^*(p)\simeq C_X(q)=X\times_{X_0}Q$$ for some $Q\stackrel{q}{\to} X_0$ in $\cX_{\ov X_0}$, and the category of $G$-actions in $\cX$.
\end{theorem}
\begin{proof}
\begin{enumerate}
\item The morphism $f$ is faithfully flat (given as a spectrum of an algebra over a field) and hence it is a morphism of effective descent for affine morphisms in the sense of algebraic geometry. Using Benabou-Roubaud \cite{benabou-roubaud}, we obtain that the pullback functor $f^*$ is monadic.
\item Using \ref{affine-CX-remark}, the counit $S_XC_X\to \id$ is an isomorphism, or, equivalently, $C_X$ is fully faithful.
\item The morphism $f$ is self-split by the property \ref{selfsplitting} from \ref{fact-pvring}.
\end{enumerate}
Categorical Galois theory stipulates that the object of morphisms of the groupoid $\Gal[f]$ is
$$
G=\Gal[f]_1=S(f^*f)=(X\times_YX)_0=\spec((A\otimes_KA)_0)=\spec(H),
$$
while the object of objects 
$$
\Gal[f]_0=S(X)=X_0=\spec(k)
$$
is a point, so we obtain a linear algebraic group $G$ over $k$, exactly as in the Picard-Vessiot case. The torsor equation is precisely the self-splitting of $f$, written in terms of $G$, and the equivalence of categories follows from categorical Galois theory \ref{ss:janelidze-gal} specialised to the framework \ref{affine-pv}.
\end{proof}

\begin{fact}[Algebraic group quotients and effective subgroups]\label{effective-alg-subgp}
Let $H\to G$ be a monomorphism/closed immersion of algebraic groups over a field $k$. Combining \cite[5.24, 5.28, 8.42--8.44, B.37, B.38]{milne}, or, by using \cite[Expos\'e V]{sga3.1}, we obtain:
\begin{enumerate}
\item $G$ is quasi-projective; 
\item the quotient $G/H$ is representable by a quasi-projective scheme over $k$; 
\item the quotient morphism $G\to G/H$ is faithfully flat; 
\item we have 
$$
G\times H\simeq G\times_{G/H}G.
$$
\end{enumerate}
We deduce that all closed subgroups of an algebraic group over a field $k$ are effective in the category of schemes with quasi-projective morphisms. 

In the category of affine schemes, a closed subgroup $H$ of an affine algebraic group $G$ is effective if and only if $G/H$ is affine. If $H$ is a normal closed subgroup, then it is effective (\cite[5.29]{milne}).
\end{fact}

\begin{proposition}[Affine Picard-Vessiot correspondence]\label{affine-pv-corr}
With assumptions of \ref{janelidze-affine}, there is a one-to-one correspondence between split affine quotients of $f:X\to Y$
and effective subgroups of the linear algebraic group $G=\Gal[f]$ which  takes
$$
 \begin{tikzpicture}
[cross line/.style={preaction={draw=white, -,
line width=4pt}}]
\matrix(m)[matrix of math nodes, row sep=.9em, column sep=.5em, text height=1.5ex, text depth=0.25ex]
{			& |(ij)| {X}	&				\\   [.02em]
	|(i)|{P} 	&			& 		\\  [.02em]
			& |(0)|{Y} 		&				\\};
\path[->,font=\scriptsize,>=to, thin]
(ij) edge node[left,pos=0.2]{$$} (i) edge node[right,pos=0.5]{$f$} (0) 
(i) edge node[below left,pos=0.5]{$p$} (0) 
;
\end{tikzpicture}
$$ 
to 
$$\Gal[X\to P].$$
Conversely, if $G'$ is an effective subgroup in the sense that it is a closed subgroup such that the coset space $G/G'$ has a structure of an affine scheme (\ref{effective-alg-subgp}), it corresponds to the quotient 
$$
X/G',
$$
which is $f$-split by the scheme $G/G'$. 

Moreover, this correspondence restricts to a one-to-one correspondence between split quotients $P$ such that $P\to Y$ is Picard-Vessiot, and closed normal subgroups of $G$. In this case, 
$$
\Gal[P\to Y]\simeq \Gal[X\to Y]/\Gal[X\to P].
$$
\end{proposition}

\begin{proof}
The statement is a direct consequence of \ref{janelidze-affine} and \ref{th-cmj-corr}.
\end{proof}

\begin{remark}
\begin{enumerate}
\item The equivalence of categories form of Picard-Vessiot theory from \ref{janelidze-affine} is new and as of yet unexplored in differential algebra. 
\item The affine Galois correspondence from \ref{affine-pv-corr} does not fully recover the classical Picard-Vessiot Galois correspondence \ref{class-pv-corr} because it only refers to \emph{effective} subgroups of the Galois group, while the classical correspondence is for all closed subgroups. The `moreover' clause does recover the correspondence for Picard-Vessiot quotients and normal groups from \cite[8.1]{maurischat}.
\end{enumerate}

\end{remark}

\section{Categorical Galois theory for differential schemes}\label{s:sat-gal-dif}

\subsection{Indexed framework for scheme-theoretic Picard-Vessiot theory}


\begin{definition}\label{sch-pv-setup}
The indexed framework for \emph{pre-Picard-Vessiot} differential Galois theory consists of the following choices for objects needed to apply \ref{indexed-gal-th}. 

\begin{enumerate}
\item Let $S$ be a base scheme, and let $\cX=\Sch_{\ov S}$.

\item Let $\cA$ be the category of $S$-differential schemes that have a categorical scheme of leaves. By this choice, we have a `categorical scheme of leaves' functor 
$$
\pi_0:\cA\to \cX.
$$

\item Let $\cP:\cX^\op\to \Cat$ be a pseudofunctor, which yields a pseudofunctor $\delta\da\cP:\cA^\op\to \Cat$ by \ref{delta-P}.

\item Let $C=C^\cP: \cP\circ \pi_0\Rightarrow \delta\da\cP$ be the pseudo-natural transformation whose $(Z,\delta_Z)$-component is the canonical functor
$$
C_Z:\cP(\pi_0(Z,\delta_Z))\to \delta\da\cP(Z,\delta_Z)
$$
from \ref{CX-P}.
\newcounter{nameOfYourChoice}
\setcounter{nameOfYourChoice}{\value{enumi}}

\end{enumerate}

For \emph{Picard-Vessiot} differential Galois theory, we choose the following additional structure.
\begin{enumerate}
\setcounter{enumi}{\value{nameOfYourChoice}}
\item Let $\cS:\cX^\op\to \Cat$ be a full sub-pseudofunctor of the self-indexing $\mathrm{Self}(\Sch_{\ov S})$ of the category of $S$-schemes over itself, so that, for an $S$-scheme $Z$,  $\cS(Z)$ is a full subcategory of $\Sch_{\ov Z}$. It gives rise to  pseudofunctor  $\delta\da\cS:\cA^\op\to \Cat$. 
\item Let $C^\cS:\cS\circ \pi_0\Rightarrow \delta\da\cP$ be the pseudo-natural transformation corresponding to $\cS$.
\item Let $U:\cP\Rightarrow\cS$ be a faithful pseudo-natural transformation. It gives rise to a morphism of fibrations, taking cartesian morphisms to cartesian, hence we obtain a pseudo-natural transformation $\delta\da U:\delta\da\cP\Rightarrow \delta\da\cS$ such that the diagram
$$
 \begin{tikzpicture} 
\matrix(m)[matrix of math nodes, row sep=2em, column sep=2em, text height=1.9ex, text depth=0.25ex]
 {
 |(1)|{\cP\circ\pi_0}		& |(2)|{\delta\da\cP} 	\\
 |(l1)|{\cS\circ\pi_0}		& |(l2)|{\delta\da\cS} 	\\
 }; 
\path[->,font=\scriptsize,>=to, thin]
(1) edge node[above]{$C^\cP$} (2) edge node[left]{$U$}  (l1)
(2) edge node[right]{$\delta\da U$} (l2) 
(l1) edge node[above]{$C^\cS$}  (l2);
\end{tikzpicture}
$$
commutes.
\item We require that $\delta\da U$ reflects (coequalisers associated with) $\cS$-universal connected components, i.e., 
if $P\in \delta\da\cP(X,\delta_X)$ for some $(X,\delta_X)\in\cA$, and we have a diagram 
$$
\begin{tikzcd}[cramped, column sep=normal, ampersand replacement=\&]
{P_1}\ar[yshift=2pt]{r}{d_0} \ar[yshift=-2pt]{r}[swap]{d_1} \&{P_0} \ar{r}{r} \& Q.
\end{tikzcd}
$$
with $r\circ d_0=r\circ d_1$, which $U$ maps onto an $\cS$-universal coequaliser
$$
 \begin{tikzpicture} 
\matrix(m)[matrix of math nodes, row sep=2em, column sep=2em, text height=1.5ex, text depth=0.25ex]
 {
|(u2)|{UP_1}  & [1em] |(u1)|{UP_0}		&[1em] |(u0)|{\pi_0(\delta\da U(P))} \\
|(2)|{X_1}  & [1em] |(1)|{X_0}		&[1em] |(0)|{\pi_0(X)} \\
 }; 
\path[->,font=\scriptsize,>=to, thin]
([yshift=.2em]2.east) edge node[above=-2pt]{} ([yshift=.2em]1.west) 
([yshift=-.2em]2.east) edge node[below=-2pt]{} ([yshift=-.2em]1.west) 
(1) edge node[above=-2pt]{} (0) 
([yshift=.2em]u2.east) edge node[above=-2pt]{} ([yshift=.2em]u1.west) 
([yshift=-.2em]u2.east) edge node[below=-2pt]{} ([yshift=-.2em]u1.west) 
(u1) edge node[above=-2pt]{} (u0) 
(u2) edge node[left]{} (2)
(u1) edge node[left]{} (1)
(u0) edge (0)
;
\end{tikzpicture}
$$
associated to the $\cS$-simple differential scheme $\delta\da U(P)$,
then the original diagram was already a coequaliser.
\end{enumerate}

\end{definition}

\begin{remark}
In view of \ref{no-pi0-affine}, the above setup with the functor $\pi_0$ and the pseudofunctor $C$ is \emph{not} an extension of the adjunction $()_0\dashv C$ we used in \ref{affine-pv}, given that the functor $()_0$ on affine differential schemes does not extend to a functor on differential schemes. 
\end{remark}

\begin{remark}
The category 
$$
\Split_C(f)
$$
consists of objects $P\in\delta\da\cP(Y)$ such that, for some $Q\in\cP(S(X))$,
$$
f^*P\simeq C_X(Q).
$$
\end{remark}

\begin{definition}\label{def-pre-pv}
A morphism $f:(X,\delta_X)\to (Y,\delta_Y)$ of $S$-differential schemes is \emph{pre-Picard-Vessiot} with respect to $U:\cP\Rightarrow\cS$ provided:
\begin{enumerate}
\item[(0)] $f$ is a morphism in $\cA$, i.e., $(X,\delta_X)$ and $(Y,\delta_Y)$ have categorical schemes of leaves over $S$;
\item $f$ is a morphism of effective descent for $\delta\da\cP$;
\item $X$, $X\times_YX$ and $X\times_YX\times_YX$ are simple for $\cS$.
\end{enumerate}
We say that $f$ is \emph{pre-Picard-Vessiot} with respect to $\cP$, if it is so with respect to $\id: \cP\Rightarrow\cP$. 
\end{definition}

\begin{definition}\label{def-pv}
A morphism $f:(X,\delta_X)\to (Y,\delta_Y)$ of $S$-differential schemes is \emph{Picard-Vessiot} with respect to $U:\cP\Rightarrow\cS$ provided:
\begin{enumerate}
\item[(0)] $f$ is a morphism in $\cA$, i.e., $(X,\delta_X)$ and $(Y,\delta_Y)$ have categorical schemes of leaves over $S$;
\item $f$ is a morphism of effective descent for $\delta\da\cP$;
\item $X$ is simple for $\cS$;
\item $f$ is auto-split with respect to $\cS$, i.e., $f\in\Split_{C^\cS}(f)$.
\end{enumerate}

If $\cP$ is already a sub-pseudofunctor of the self-indexing of $\Sch_{\ov S}$, we say that $f$ is \emph{Picard-Vessiot} with respect to $\cP$, if it is such with respect to $\id:\cP\Rightarrow\cP$. 
\end{definition}

\begin{remark}\label{pv-for-P}
%
A morphism $f:(X,\delta_X)\to (Y,\delta_Y)$ in $\cA$ is \emph{Picard-Vessiot} with respect to $\cP$ provided:
\begin{enumerate}
\item $f$ is a morphism of effective descent for $\delta\da\cP$;
\item $X$ is simple for $\cP$;
\item $f$ is auto-split, i.e., $f\in\Split_C(f)$.
\end{enumerate}
\end{remark}

\begin{lemma}\label{properties-pv}
\begin{enumerate}
\item If $(X,\delta_X)$ is simple for $\cS$, then it is simple for $\cP$.
\item Given an $S$-differential scheme $(Y,\delta_Y)$, the functor $\delta\da U_Y$ restricts to $$\Split_{C^\cP}(f)\to \Split_{C^\cS}(f).$$
\end{enumerate}
\end{lemma}
\begin{proof}
For the first claim, suppose $(X,\delta_X)$ is simple with respect to $\cS$, and let $Q\in \cP(\pi_0(X))$. Let $P=C^\cP_X(Q)$ and consider the associated diagram 
$$
\begin{tikzcd}[cramped, column sep=normal, ampersand replacement=\&]
{P_1}\ar[yshift=2pt]{r}{d_0} \ar[yshift=-2pt]{r}[swap]{d_1} \&{P_0} \ar{r}{r} \& Q.
\end{tikzcd}
$$
Since $(X,\delta_X)$ is simple with respect to $\cS$, the analogous diagram for $\delta\da U (C^\cP_X(Q))=C^\cS_X(UQ)$ is an $\cS$-universal coequaliser and $\pi_0(C^\cS_X(UQ)=UQ$. Using the fact that $\delta\da U$ reflects connected components, we deduce that the original diagram is a coequaliser and $\pi_0(P)=Q$.

The second claim follows directly from the fact that $U$ preserves cartesian morphisms.
\end{proof}

\begin{lemma}\label{PV-pre-PV}
If $f$ is a Picard-Vessiot morphism of differential schemes with respect to $U:\cP\Rightarrow\cS$, then $f$ is pre-Picard-Vessiot with respect to $U$.  
\end{lemma}
\begin{proof}
Writing $G_0=\pi_0(X)$, since 
 $f$ is auto-split with respect to $C^\cS$,  for some $G_1\in \cS(G_0)$, 
$$
X\times_YX\simeq C^\cS_X(G_1)=X\times_{C(G_0)}C(G_1).
$$
By the definition of simplicity using universal coequalisers, any object in the essential image of $C^\cS_X$ is automatically simple for $\cS$, whence $X\times_YX$ is simple for $\cS$ and we have $\pi_0(X\times_YX)\simeq G_1$. Moreover, 
\begin{multline*}
X\times_YX\times_YX\simeq (X\times_YX)\times_X(X\times_YX)\\
\simeq (C(G_1)\times_{C(G_0)}X)\times_X(X\times_{C(G_0)}C(G_1))
\simeq X\times_{C(X_0)}C(G_1\times_{G_0} G_1)\\ 
\simeq C^\cS_X(G_1\times_{G_0}G_1),
\end{multline*}
whence $X\times_YX\times_YX$ is also simple for $\cS$ with $\pi_0(X\times_YX\times_YX)\simeq G_1\times_{G_0}G_1$, as required. 
\end{proof}

\begin{remark}\label{galprecatdef}
The assumption that $f$ is pre-Picard-Vessiot for $\cP$ ensures that the kernel-pair groupoid 
$$
 \begin{tikzpicture} 
\matrix(m)[matrix of math nodes, row sep=0em, column sep=3em, text height=1.5ex, text depth=0.25ex]
 {
|(n)|{\bbG_f:} &[1em] |(2)|{X\times_Y X\times_Y X}  & [1em] |(1)|{X\times_Y X}		&[1em] |(0)|{X} \\
 }; 
\path[->,font=\scriptsize,>=to, thin]
([yshift=1em]2.east) edge node[above=-2pt]{$\pi_{01}$} ([yshift=1em]1.west) 
(2) edge node[above=-2pt]{$\pi_{02}$} (1)
([yshift=-1em]2.east) edge node[above=-2pt]{$\pi_{12}$} ([yshift=-1em]1.west) 
([yshift=1em]1.east) edge node[above=-2pt]{$\pi_0$} ([yshift=1em]0.west) 
(0)  edge node[above=-2pt]{$\Delta$} (1) 
([yshift=-1em]1.east) edge node[above=-2pt]{$\pi_1$} ([yshift=-1em]0.west) 
;
\end{tikzpicture}
$$
is a category in $\cA$, and that the \emph{Galois precategory}
$$
\Gal[f]=\pi_0(\bbG_f)
$$
exists as a precategory in $\cX$.

If $f$ is Picard-Vessiot, then $G[f]$ is an internal \emph{category} (actually a groupoid without the inversion of arrows named) in $\cX$, by the argument in \ref{PV-pre-PV}.
 \end{remark}

\begin{theorem}\label{scheme-dif-Galois}
A pre-Picard-Vessiot morphism $f$ for $\cP$ induces an equivalence of categories
$$
\Split_C(f)\simeq \cP^{\Gal[f]}
$$
between the category of objects of $\delta\da\cP(Y)$ $C$-split by $f$ and the category of $\cP$-actions of the precategory $\Gal[f]$.

Moreover, if $f$ is Picard-Vessiot for $U:\cP\Rightarrow\cS$, the latter becomes the category of $\cP$-actions of the groupoid $\Gal[f]$. 
\end{theorem}

\begin{proof}
When $f$ is pre-Picard-Vessiot, \ref{galprecatdef} ensures that $\bbG_f$ gives a precategorical decomposition of $f$ in $\cA$, and that the Galois precategory $\Gal[f]$ is well-defined in $\cX$. By assumption, $f$ is of effective descent for $\bbG_f$. 

By the assumption on simplicity of $X$, $X\times_YX$, $X\times_YX\times_YX$ and \ref{CX-P}, we get that the functors $C_X$, $C_{X\times_YX}$, $C_{X\times_YX\times_YX}$ are fully faithful.  

Hence, all the assumptions of \ref{indexed-gal-th} are satisfied. When $f$ is Picard-Vessiot, it is also pre-Picard-Vessiot by \ref{PV-pre-PV}.

Therefore, in both cases we get the desired equivalence involving actions of the precategory $\Gal[f]$, which happens to be a groupoid in the Picard-Vessiot case, as observed in \ref{galprecatdef}.
\end{proof}

\subsection{Specialisation of the Galois precategory}

\begin{proposition}\label{special-gal}
Suppose $f:X\to Y$ is a pre-Picard-Vessiot morphism of differential schemes with respect to $U:\cP\Rightarrow\cS$, where $\cS$ is a class of morphisms stable under base change, and let $q:Q\to\pi_0(Y)$ be a morphism in $\cS$. Then $q^*(f)=f_Q:X_Q\to Y_Q$ is pre-Picard-Vessiot and its Galois precategory is
$$
\Gal[q^*(f)]\simeq q^*\Gal[f]=\Gal[f]\times_{\pi_0(Y)}Q.
$$
Moreover, if $f$ is Picard-Vessiot for $U$, so is $f_Q$, and the same specialisation formula holds for Galois groupoids. 
\end{proposition}
\begin{proof} 
Assuming that $f$ is pre-Picard-Vessiot, $f_Q$ is an effective descent morphism for $\delta\da\cP$ as a base change of an effective descent morphism $f$. 

Writing $X^2=X\times_YX$ and $X^3=X\times_YX\times_YX$, the Galois precategory of $f$ is $\Gal[f]=(G_2,G_1,G_0)$ with $G_i=\pi_0(X^{i+1})$. By defining $Q_i=G_i\times_{\pi_0(Y)}Q=q^*(G_i)$, we obtain a diagram
$$
 \begin{tikzpicture}
 [cross line/.style={preaction={draw=white, -,line width=4pt}}, proj/.style={dotted,-}]
 \def\triple#1#2{
 ([yshift=.4em]#1.east) edge 
 ([yshift=.4em]#2.west) 
(#1) edge 
(#2)
([yshift=-.4em]#1.east) edge 
([yshift=-.4em]#2.west) }
\def\double#1#2{
([yshift=.4em]#1.east) edge 
([yshift=.4em]#2.west) 
(#2)  edge 
(#1) 
([yshift=-.4em]#1.east) edge 
([yshift=-.4em]#2.west) }
\def\tripleov#1#2{
 ([yshift=.4em]#1.east) edge[cross line] 
 ([yshift=.4em]#2.west) 
(#1) edge[cross line] 
(#2)
([yshift=-.4em]#1.east) edge[cross line] 
([yshift=-.4em]#2.west) }
\def\doubleov#1#2{
([yshift=.4em]#1.east) edge[cross line] 
([yshift=.4em]#2.west) 
(#2)  edge[cross line] 
(#1) 
([yshift=-.4em]#1.east) edge[cross line] 
([yshift=-.4em]#2.west) }
\matrix(m)[matrix of math nodes, row sep=1em, column sep=1em,  
text height=1.2ex, text depth=0.25ex]
{
|(P2)|{X^3} 	&			&[0em] |(P1)|{X^2}  	&[0em]		& |(P0)|{X}	&			& |(P)|{Y}	&	\\[1.2em]
			& |(p2)|{X^3_Q}	& 				&|(p1)|{X^2_Q} 	&			& |(p0)|{X_Q} 	&		& |(p)|{Y_Q}\\[.6em]	
|(Q2)|{G_2} 	&			&[0em] |(Q1)|{G_1}  	&[0em]		& |(Q0)|{G_0}	&			& |(Q)|{\pi_0(Y)} &  \\[1.2em]
			& |(q2)|{Q_2}  	& 				&|(q1)|{Q_1} 	&			& |(q0)|{Q_0}  	&		& |(q)|{Q}\\};
\path[->,font=\scriptsize,>=to, thin,inner sep=1pt]
\triple{P2}{P1}
\double{P1}{P0}
(P0) edge (P)

\triple{Q2}{Q1}
\double{Q1}{Q0}
(Q0) edge (Q)

(P2) edge node[pos=0.25,left]{}(Q2)
(P1) edge node[pos=0.25,left]{}(Q1)
(P0) edge node[pos=0.25,left]{}(Q0)
(P) edge node[pos=0.25,left]{}(Q)
\tripleov{p2}{p1}
\doubleov{p1}{p0}
(p0) edge[cross line] (p)

\tripleov{q2}{q1}
\doubleov{q1}{q0}
(q0) edge[cross line] (q)

(q2) edge node[pos=0.5,below left]{} (Q2)
(q1) edge node[pos=0.5,below left]{}(Q1)
(q0) edge node[pos=0.5,below left]{}(Q0)
(q) edge (Q)

(p2) edge node[pos=0.6,above right]{} (P2)
(p1) edge node[pos=0.6,above right]{} (P1)
(p0) edge node[pos=0.6,above right]{} (P0)
(p) edge (P)

(p2) edge[cross line] node[pos=0.25,left]{}(q2)
(p1) edge[cross line] node[pos=0.25,left]{}(q1)
(p0) edge[cross line] node[pos=0.25,left]{}(q0)
(p) edge (q)
;
\end{tikzpicture}
$$ 
where all the squares but those on the front and the rear face are cartesian. 
Since $X^i$ is simple for $\cS$, so is $X^i_Q\simeq (X_Q)^i$ and 
$$
\pi_0((X_Q)^{i+1})=\pi_0(X^{i+1}\times_{G_i}Q_i)=\pi_0(X^{i+1}\times_{\pi_0(X^{i+1})}Q_i)\simeq Q_i,
$$
whence we obtain that $\Gal[f_Q]=(Q_2,Q_1,Q_0)$, proving the specialisation formula.   

If $f$ is Picard-Vessiot, then the middle rear square is cartesian, i.e.,  $X\times_Y X\simeq X\times_{G_0}G_1$, so the middle front square is too, and we obtain
$$
X_Q\times_{Y_Q}X_Q\simeq X^2_Q\simeq X_Q\times_{Q_0}Q_1,
$$
so $f_Q$ is auto-split and thus Picard-Vessiot.
\end{proof}

\section{Applications}\label{s:applications}

\subsection{Quasi-projective Picard-Vessiot theory}

\begin{theorem}\label{qproj-gal-th}
Let $(K,\delta)$ be a differential field of characteristic 0 with constants $k$, and write $S=\spec(k)$. Let $f:X\to Y=\spec(K,\delta)$ be a morphism of $S$-differential schemes such that
\begin{enumerate}
\item $X$ is integral quasi-projective over $Y$ and its only leaf is the generic point;
\item $f$ is auto-split, witnessed by a quasi-projective $k$-scheme $G$,
$$
X\times_YX\simeq X\times_{C(S)}C(G).
$$
Then $f$ is Picard-Vessiot, $\Gal[f]$ is an $S$-algebraic group isomorphic to $G$, and there is an equivalence of categories
\begin{multline*}
\{\text{quasi-projective $S$-differential morphisms to $Y$ split by $f$}\} \\
\simeq
\{\text{quasi-projective $S$-scheme actions of $G$}\} 
\end{multline*}
\end{enumerate}
\end{theorem}

\begin{proof}
We follow the template of \ref{scheme-dif-Galois} for $\cP$ the fibred category of quasi-projective morphisms. By assumption (1) and \ref{our-proof-cor}, $X$ is simple with categorical scheme of leaves $S$. By \ref{difl-qproj-desc}, $f$ is of effective descent for $\delta\da\cP$. Thus, $X$ is Picard-Vessiot and \ref{scheme-dif-Galois} gives the desired equivalence. 
\end{proof}

\begin{corollary}[Quasi-projective differential Galois correspondence]\label{qproj-pv-corr}
With notation of \ref{qproj-gal-th}, there is an order-reversing one-to-one correspondence between split $S$-differential quasi-projective fpqc quotients of $f:X\to Y$ in $\cA$
and closed subgroups of the algebraic group $G=\Gal[f]$ which  takes
$$
 \begin{tikzpicture}
[cross line/.style={preaction={draw=white, -,
line width=4pt}}]
\matrix(m)[matrix of math nodes, row sep=.9em, column sep=.5em, text height=1.5ex, text depth=0.25ex]
{			& |(ij)| {X}	&				\\   [.02em]
	|(i)|{P} 	&			& 		\\  [.02em]
			& |(0)|{Y} 		&				\\};
\path[->,font=\scriptsize,>=to, thin]
(ij) edge node[left,pos=0.2]{$$} (i) edge node[right,pos=0.5]{$f$} (0) 
(i) edge node[below left,pos=0.5]{$p$} (0) 
;
\end{tikzpicture}
$$ 
to 
$$\Gal[X\to P].$$
Conversely, a closed subgroup $G'$ corresponds to the quotient 
$$
X/G',
$$
which is $f$-split by the quasi-projective scheme $G/G'$. 

Moreover, this correspondence restricts to a one-to-one correspondence between split fpqc quotients $P$ such that $P\to Y$ is Picard-Vessiot, and closed normal subgroups of $G$. In this case, 
$$
\Gal[P\to Y]\simeq \Gal[X\to Y]/\Gal[X\to P].
$$
\end{corollary}

\begin{proof}
The claimed correspondence is more specific than a claim that could be extracted from \ref{th-cmj-corr}, so we provide an explicit proof following an analogous strategy. 

If $h:X\to P$ is $f$-split by $g:G\to Q$, and $h$ is fpqc, then $f^*(h)$ is fpqc, hence an universal effective epimorphism by \cite[Tag 023P]{stacks-project}. Since $S_X$ preserves colimits as a left adjoint, it follows that $g=U(h)=S_Xf^* (h)$ is a regular epimorphism, hence again an effective epimorphism in the presence of pullbacks. Since $g$ corresponds to a quotient of an algebraic group, it is fpqc again by \ref{effective-alg-subgp}.

Conversely, if $G'\leq G$ is a closed subgroup, then, by \ref{effective-alg-subgp}, $g:G\to G/G'=Q$ is fpqc. The quotient $h:X\to P$ corresponding to it satisfies $f^*(h)\simeq C_X(g)$, whose underlying scheme morphis is fpqc since $C_X$ acts as base change on $g$. Since $f$ is fpqc, using the fact \cite[Tag 02YJ]{stacks-project} that the properties of being `faithfully flat and quasi-compact' are local in the fpqc topology, it follows that $h$ is fpqc itself. 
\end{proof}

\begin{remark}
The theory above shows that even in \emph{linear Picard-Vessiot theory,} with $X=(A,\delta_A)$ the spectrum of a Picard-Vessiot ring over a differential field $(K,\delta)$, in order to extend the correspondence \ref{affine-pv-corr} to all closed subgroups od the Galois group, we are forced to consider split quasi-projective quotients of $X$. 
\end{remark}

\begin{remark}
The above theory applies to \emph{strongly normal} differential Galois theory of \cite{kolchin-sn}.
\end{remark}

\subsection{Polarised quasi-projective differential Galois theory}

\begin{theorem}\label{polarised-pv-thm}

Let $f:(X,\delta_X)\to (Y,\delta_Y)$ be a  morphism of $S$-differential schemes such that, 
with notation \ref{nota-polarised}, 
\begin{enumerate}
\item the underlying $S$-scheme morphism $X\to Y$ is fpqc;
\item $(X,\delta_X)$ is simple for $\cS$ with scheme of leaves $G_0$;
\item there is  
an $S$-morphism $G_1\to G_0$ such that
$$
(X,\delta_X)\times_{(Y,\delta_Y)}(X,\delta_X)\simeq (X,\delta_X)\times_{(G_0,0)}(G_1,0).
$$
\end{enumerate}
Then $f$ is Picard-Vessiot for $U:\cP\Rightarrow\cS$, $\Gal[f]$ is the groupoid $(G_1\doublerightarrow{}{} G_0)$ and we have an equivalence between the category of quasi-projective polarised $S$-differential morphisms $(P,\cL_P)\to Y$ split by $f$ and the category of quasi-projective polarised actions $(Q,\cL_Q)\to G_0$ of $\Gal[f]$.
\end{theorem}

\begin{proof}
Using \ref{difl-qproj-desc}, we obtain that $f$ is of effective descent for $\delta\da\cP$. By \ref{reflect-quasiproj-coeq}, the forgetful functor $\delta\da U$ reflects $\cS$-universal categorical schemes of leaves. Hence, $f$ is indeed Picard-Vessiot for $U$ and we can apply the template Theorem~\ref{scheme-dif-Galois}.
\end{proof}

\subsection{A parametrised family of strongly normal extensions}\label{s:elliptic}

This example is inspired by Kolchin's example of Weierstrassian extensions of differential fields \cite[III.6]{kolchin-sn}, with more explicit calculations in \cite{kovacic-tams03, kovacic-tams06} where the authors consider the strongly normal extension associated to a vector field on an elliptic curve. We are able to treat \emph{families} of elliptic curves over a base scheme as parametrised families of strongly normal extensions. 

We consider the Weierstrass family of elliptic curves $$E\to S=\spec \Q[u,v,(4u^3+27v^2)^{-1}]$$ obtained by projectivising the naive equation $y^2=x^3+ux+v$. It has a globally defined invariant differential $\omega_E\in\Omega^1_{E/S}$. Let 
$$(Y,\delta_Y)=\spec(\Q(t),\frac{d}{dt})\times (S,0),$$ 
and let $p:(Y,\delta_Y)\to (S,0)$ be the projection, considered as a morphism of $S$-differential schemes. We focus on the scheme
$$
X=Y\times_S E.
$$ 
Its sheaf of differentials is
$$
\Omega_{X/S}\simeq \pi_1^*\Omega_{Y/S}\oplus\pi_2^*\Omega_{E/S}=\pi_1^*\Omega_{Y/S}\oplus \langle\omega_X\rangle,
$$
for some $\omega_X$.

We endow $X$ with a differential scheme structure over $(Y,\delta_Y)$ by stipulating
$$
 \omega_X\mapsto \alpha\in\Gamma(X,\cO_X),
$$
and we write 
$$
f:(X,\delta_X)\to (Y,\delta_Y)
$$
for the corresponding morphism of $S$-differential schemes. 

We claim that, for s suitable $\alpha$, the morphism $$\eta=p\circ f:(X,\delta_X)\to (S,0)$$ is a universal geometric quotient. Indeed, we can verify this by \ref{univ-geom-quot} as follows.

\begin{enumerate}
\item The morphism $\eta$ is a composite of proper smooth morphism $f$ with geometrically integral fibres and an affine faithfully flat morphism $p$. In particular, it is qcqs faithfully flat.
\item Using the fact that $f$ is proper smooth with geometrically connected fibres, we obtain the sheaf condition
$$
\Const(\eta_*\cO_X)=\Const(p_*(f_*\cO_X))\simeq\Const(p_*\cO_Y)\simeq \cO_S.
$$

\item
The fibre at a geometric point $\bar{s}:\spec(L)\to S$ is an $L$-differential scheme
$$
(X_{\bar{s}},\delta)\simeq \spec(L(t),\frac{d}{dt})\times_{\spec(L,0)}(E_{\bar{s}},0).
$$
whose differential structure is associated with a logarithmic differential equation $$\ell\delta(P)=\alpha_{\bar{s}},$$
where $\alpha_{\bar{s}}$ is the pullback of $\alpha$ to $X_{\bar{s}}$. 
If $\alpha_{\bar{s}}$ is not a logarithmic derivative of some point in $X_{\bar{s}}(L(t)^{\text{alg}})$, then \cite[2.1]{bertrand-pillay} and the calculations of \cite{kolchin-sn,kovacic-tams03} prove that $(X_{\bar{s}},\delta)$ is a torsor associated with the strongly normal extension $\kappa(X_{\bar{s}})$ of $L(t)$ generated by a transcendental solution of the logarithmic differential equation, which is classically simple over $L$ with Galois group $E_{\bar{s}}$, as required. 

A sufficient condition for the above is that $\alpha=\delta_X(f^\sharp\lambda)$ for some $\lambda\in\Gamma(Y, \cO_Y)$  and $\alpha_{\bar{s}}\neq 0$ for all geometric points $\bar{s}$. Indeed, the solution of the above logarithmic differential equation is given by 
$$
P=\exp_{E_{\bar{s}}}(\lambda_{\bar{s}}(t))=(\wp_{\bar{s}}(\lambda_{\bar{s}}(t)),\wp'_{\bar{s}}(\lambda_{\bar{s}}(t))),
$$
where $\wp_{\bar{s}}$ is the Weierstrass function of $E_{\bar{s}}$. Using the Ax-Schanuel functional transcendence result \cite[1.1]{kirby}, we obtain that 
$\text{tr.deg}_L(L(\lambda_{\bar{s}},\wp(\lambda_{\bar{s}})))=2$, whence 
$\kappa(X_{\bar{s}})$ is generated over $L(t)$ by a transcendental solution.

\end{enumerate}

Thus, we have that $\pi_0(X)=\pi_0(Y)=S$, $X$ is simple with respect to arbitrary scheme morphisms and we claim that $f$ is self-split, i.e.,
$$
X\times_Y X\simeq X\times_{(S,0)}(E,0).
$$
The underlying scheme isomorphism is
$$
\psi: X\times_S E\simeq Y\times_S E\times_S E\stackrel{\id\times(\mu,\pi_1)}{\longrightarrow}Y\times_S E\times_S E\simeq(Y\times_SE)\times_Y(Y\times_SE)\simeq X\times_YX,
$$
where $\mu:E\times_SE\to E$ is the group operation on $E$. 

In order to show that $\psi$ is a morphism of differential schemes, it suffices to verify that the diagram
$$
 \begin{tikzpicture} 
\matrix(m)[matrix of math nodes, row sep=2em, column sep=2em, text height=1.9ex, text depth=0.25ex]
 {
 |(1)|{\psi^*\Omega_{X\times_YX/S}}		& |(2)|{\Omega_{X\times_SE/S}}	\\
								& |(l2)|{\cO_{X\times_SE}} 	\\
 }; 
\path[->,font=\scriptsize,>=to, thin]
(1) edge node[above]{} (2) 
(2) edge node[right]{} (l2) 
(1) edge (l2);
\end{tikzpicture}
$$
commutes. Bearing in mind that both $\Omega_{X\times_YX/S}$ and $\Omega_{X\times_SE/S}$ are isomorphic to 
$$\Omega_{Y\times_SE\times_SE/S}\simeq \pi_1^*\Omega_{Y/S}\oplus\pi_2^*\Omega_{E/S}\oplus\pi_3^*\Omega_{E/S},$$
we name the generators as
$$
\psi^*\Omega_{X\times_YX/S}=\pi_1^*\Omega_{Y/S}\oplus \langle\omega_{X,1},\omega_{X,2}\rangle, \ \ \ \ 
\Omega_{X\times_SE/S}=\pi_1^*\Omega_{Y/S}\oplus \langle\omega_X,\omega_E\rangle.
$$
The horizontal  arrow acts as identity on $\pi_1^*\Omega_{Y/S}$ and maps
$$
 \omega_{X,1}\mapsto \omega_X+\omega_E, \ \ \ \omega_{X,2}\mapsto\omega_X,
$$ 
while the arrows to $\cO_{X\times_SE}$ map
$$
 \omega_{X,1}\mapsto \alpha, \ \ \ \omega_{X,2}\mapsto \alpha, \ \ \ \omega_X\mapsto \alpha, \ \ \ \omega_E\mapsto 0,
$$
so the diagram commutes. These considerations follow formally from properties of invariant differentials and do not require explicit calculations in coordinates as in \cite{kolchin-sn,kovacic-tams03}. 

It follows that the Galois groupoid is 
$$
\Gal[f]=E\doublerightarrow{}{} S.
$$
Note that \ref{special-gal} explains the variation in parameters from $S$. If $s\in S(L)$ for a field $L$, then
$$
\Gal[f_s]\simeq \Gal[f]\times_S\spec(L)\simeq (E\doublerightarrow{}{} S)\times_S\spec(L)\simeq E_s\doublerightarrow{}{}\spec(L)\simeq E_s,
$$
considered as an algebraic group over $L$. Hence, our Galois groupoid specialises to the classical differential Galois groups of strongly normal extensions associated to logarithmic-differential equations on elliptic curves $f_s:X_s\to Y_s$ uniformly in parameter $s$. 

\subsection{Galois groupoid of the Airy equation}\label{s:airy}

On the example of the Airy equation
$$
y''=xy,
$$
we show that Galois theory of linear differential equations can be done in a more canonical way through a Galois groupoid, rather than following the classical route of constructing a Picard-Vessiot extension in a non-canonical way in order to obtain a Galois group. 

Let $S=\spec(k)$ be the spectrum of a field of characteristic 0, let $Y=\spec (k(x),\frac{d}{dx})$ and $X=\spec(A)$ with $A=k(x)[u,\det(u)^{-1}]$, the `full universal solution algebra'  of the Airy equation. We make the projection
$$
f:X=Y\times_S\GL_2=\spec(A)\to Y
$$
into  a  morphism of differential $S$-schemes by endowing $X$ with a vector field
$$
\frac{\partial}{\partial x}+u_{21}\frac{\partial}{\partial u_{11}}+u_{22}\frac{\partial}{\partial u_{12}}+xu_{11}\frac{\partial}{\partial u_{21}}+xu_{12}\frac{\partial}{\partial u_{22}},
$$
which also determines the differential structure on $Y$. In terms of differentials,
$$
\Omega_{X/S}=f^*\Omega_{Y/S}\oplus\pi_2^*\Omega_{\GL_2/S}=\langle dx,\omega_{11},\omega_{12},\omega_{21},\omega_{22}\rangle,
$$
where $\omega=du\cdot u^{-1}$ is the invariant differential on $\GL_2$, the $S$-differential scheme structure on $X$ and $Y$ is given by assigning
$$
dx\mapsto 1, \ \ \ \omega\mapsto\begin{pmatrix}0 & 1\\ x & 0\end{pmatrix}. 
$$

We claim that the composite
$$
\eta: X\stackrel{\pi_2}{\longrightarrow}{\GL_2}\stackrel{\det}{\longrightarrow} \bbG_m
$$
is a geometric quotient. Indeed, we can deduce it from \ref{univ-geom-quot} as follows.

\begin{enumerate}
\item We observe that $\eta$ is affine faithfully flat. 
\item Given that we are in the affine setting, the sheaf condition follows from the folklore fact that 
$$
\Const(A)=k[\det(u),\det(u)^{-1}],
$$
where $A=k[x,u,\det(u)^{-1}]$, i.e., the `only' algebraic relation that a fundamental set of solutions of the Airy equation satisfies is that its Wronskian is a constant. 

\item A choice of a geometric point $\bar{s}:\spec(L)\to \bbG_m$ correspond to fixing a value $\lambda\in L$ for $\det(u)$, whence 
$$
X_{\bar{s}}=\spec(L(x)[u]/(\det(u)-\lambda)),
$$
where the latter is a Picard-Vessiot ring over $L(x)$ by construction \cite[3.4, 4.29]{magid-lect}, so
$X_{\bar{s}}$ is classically simple over $L$,  as required. 
\end{enumerate}

Thus, we obtain that $X$ is simple with respect to arbitrary scheme morphisms with
$$
\pi_0(X)=G_0=\bbG_m.
$$
We claim that $f:X\to Y$ is self-split, i.e., there is a natural quotient morphism
$$
\eta_1:X\times_YX\simeq Y\times_S\GL_2\times_S\GL_2\to G_1=(\GL_2\times_S\GL_2)/\SL_2. 
$$
that yields an isomorphism of differential schemes
$$
X\times_Y X\simeq X\times_{(G_0,0)}(G_1,0).
$$
Indeed, for the last isomorphism in the definition of $\eta_1$, given a pair of invertible matrices $(u,v)$, consider the augmented matrix
$$
\begin{pmatrix}
u_{11} & u_{12} & v_{11} & v_{12}\\
u_{21} & u_{22} & v_{21} & v_{22}\\
\end{pmatrix}
$$
and send it to the the homogeneous sextuple of determinants $p_{ij}$ of its $2\times 2$ minors for $1\leq i<j\leq 4$, considered as the Pl\"ucker coordinates of the Grassmanian $\text{\rm Gr}[2,4]$, satisfying the familiar relation $p_{12}p_{34}-p_{13}p_{24}+p_{14}p_{23}=0$. The definition is $SL_2$ invariant, so it factors through $(\GL_2\times_S\GL_2)/\SL_2$. Using the derivation on $X\times_Y X$, we explicitly verify that all the $p_{ij}$ are constants, so $\eta_1$ is a differential scheme morphism  to $(\text{\rm Gr}[2,4],0)$. 

The torsor isomorphism 
\begin{equation*}
\begin{split}
\psi:X\times_S (\GL_2\times_S\GL_2)/\SL_2\simeq Y\times_S\GL_2\times_S (\GL_2\times_S\GL_2)/\SL_2 \\ 
\longrightarrow Y\times_S\GL_2\times_S\GL_2\simeq X\times_YX,
\end{split}
\end{equation*}
takes a tuple $(y,u,[u_1,v_1])$ and maps it to $(y,u,sv_1)$, where $s\in\SL_2$ is the unique matrix such that $u=su_1$. Its inverse is given as $(\pi_{Y\times_S\GL_2)},\eta_1)$, i.e., in coordinates, by $(y,u,v)\mapsto (y,u,[u,v])$.

In order to show that $\psi$ is a morphism of differential schemes, it suffices to verify that the diagram
$$
 \begin{tikzpicture} 
\matrix(m)[matrix of math nodes, row sep=2em, column sep=2em, text height=1.9ex, text depth=0.25ex]
 {
 |(1)|{\psi^*\Omega_{X\times_YX/S}}		& |(2)|{\Omega_{X\times_{G_0}G_1/S}}	\\
								& |(l2)|{\cO_{X\times_{G_0}G_1}} 	\\
 }; 
\path[->,font=\scriptsize,>=to, thin]
(1) edge node[above]{} (2) 
(2) edge node[right]{} (l2) 
(1) edge (l2);
\end{tikzpicture}
$$
commutes. 
Using the fact that both $\Omega_{X\times_YX/S}$ and $\Omega_{X\times_{G_0}G_1/S}$ are isomorphic to 
$$\Omega_{Y\times_S\GL_2\times_S\GL_2/S}\simeq \pi_1^*\Omega_{Y/S}\oplus\pi_2^*\Omega_{\GL_2/S}\oplus\pi_3^*\Omega_{GL_2/S},$$
we can name the generators, thought of as pullbacks of the invariant differentials on $GL_2$, as
$$
\psi^*\Omega_{X\times_YX/S}= \langle dx, \omega_{1},\omega_{2}\rangle, \ \ \ \ 
\Omega_{X\times_{G_0}G_1/S}=\langle dx,\omega,\omega_0\rangle,
$$
so that the horizontal  arrow maps
$$
dx\mapsto dx, \ \ \  \omega_{1}\mapsto \omega, \ \ \ \omega_{2}\mapsto\omega+\omega_0,
$$ 
while the arrows to $\cO_{X\times_SE}$ map
$$
 \omega_{1}\mapsto \begin{pmatrix}0 & 1\\ x & 0\end{pmatrix}, \ \ \ \omega_{2}\mapsto \begin{pmatrix}0 & 1\\ x & 0\end{pmatrix}, \ \ \ \omega\mapsto \begin{pmatrix}0 & 1\\ x & 0\end{pmatrix}, \ \ \ \omega_0\mapsto 0,
$$
so the diagram commutes.

From the self-splitting, we conclude that $$\pi_0(X\times_YX)\simeq G_1,$$
and our Galois theory gives the Galois groupoid 
\begin{align*}
\Gal[f] &=G_1\doublerightarrow{}{} G_0\\ 
& = (\GL_2\times_S\GL_2)/\SL_2 \doublerightarrow{}{} \GL_2/\SL_2=\bbG_m.
\end{align*}

Intuitively, at least in the case where $k$ is algebraically closed, the points of the object of objects $G_0$ correspond to a choice of a Picard-Vessiot extension with Wronskian $\lambda$, and the object of morphisms $G_1$ encodes the isomorphisms between different choices of $\lambda$. The stabiliser of a chosen object is its usual Picard-Vessiot Galois group $\SL_2$. 

The Malgrange groupoid in this example is given by \cite[2.8]{casale-blazquez} as
$$
\bbA^1\times_S\bbA^1\times_S\SL_2\doublerightarrow{}{} \bbA^1.
$$
While both groupoids show that the symmetries of the Airy equation are governed by the group $\SL_2$, they appear to do so in different ways.

\bibliographystyle{plain}

\end{document}